\documentclass[12pt,a4paper]{article}
\usepackage[cp1251]{inputenc}
\usepackage[hidelinks]{hyperref}
\usepackage[english,russian]{babel}
\usepackage{amsmath}
\usepackage{amsthm}
\usepackage{MPatr_Style_English_09.2014}
\usepackage{dsfont}
\usepackage{pifont}
\usepackage{amsfonts}
\usepackage{MnSymbol}
\DeclareMathOperator{\domain}{\mathsf{domain}}

\DeclareMathOperator{\nodes}{\hspace{0.4pt}\mathsf{nodes}\hspace{-0.8pt}}
\DeclareMathOperator{\maxel}{\hspace{0.4pt}\mathsf{max}\hspace{-0.4pt}}
\DeclareMathOperator{\minel}{\hspace{0.4pt}\mathsf{min}\hspace{-0.6pt}}
\DeclareMathOperator{\height}{\hspace{0.4pt}\mathsf{height}}

\DeclareMathOperator{\level}{\hspace{0.4pt}\mathsf{level}}
\DeclareMathOperator{\branches}{\hspace{0.4pt}\mathsf{branches}\hspace{-0.8pt}}
\DeclareMathOperator{\sons}{\hspace{0.4pt}\mathsf{sons}}

\DeclareMathOperator{\roott}{\hspace{0.4pt}\mathsf{root}}

\DeclareMathOperator{\impl}{\hspace{0.4pt}\mathsf{implant}\hspace{-0.8pt}}
\DeclareMathOperator{\expl}{\hspace{0.4pt}\mathsf{explant}\hspace{0.5pt}}
\DeclareMathOperator{\supp}{\hspace{0.4pt}\mathsf{support}\hspace{0.5pt}}
\DeclareMathOperator{\hybr}{\hspace{0.4pt}\mathsf{hybrid}\hspace{0.5pt}}
\DeclareMathOperator{\skeleton}{\hspace{0.4pt}\mathsf{skeleton}\hspace{-0.8pt}}

\DeclareMathOperator{\flesh}{\hspace{0.4pt}\mathsf{flesh}}
\DeclareMathOperator{\fruit}{\hspace{0.4pt}\mathsf{fruit}\hspace{-0.8pt}}
\DeclareMathOperator{\yield}{\hspace{0.4pt}\mathsf{yield}\hspace{-0.8pt}}

\DeclareMathOperator{\shoot}{\hspace{0.4pt}\mathsf{shoot}\hspace{0.5pt}}
\DeclareMathOperator{\scope}{\hspace{0.4pt}\mathsf{scope}\hspace{0.5pt}}

\DeclareMathOperator{\cut}{\hspace{0.4pt}\mathsf{cut}\hspace{0.5pt}}
\DeclareMathOperator{\loss}{\hspace{0.4pt}\mathsf{loss}\hspace{0.5pt}}
\DeclareMathOperator{\fhybr}{\hspace{0.4pt}\mathsf{fol\hspace{-0.7pt}.\hspace{-2pt}hybr}\hspace{0.5pt}}

\DeclareMathOperator{\length}{\hspace{0.4pt}\mathsf{length}\hspace{-0.8pt}}
\DeclareMathOperator{\nbhds}{\hspace{0.4pt}\mathsf{nbhds}\hspace{0.5pt}}

\author{Mikhail Patrakeev\footnote{Ural Federal University, 620002, 19 Mira street, Yekaterinburg, Russia and Krasovskii Institute of Mathematics and Mechanics of UB RAS, 620990, 16 Sofia Kovalevskaya street, Yekaterinburg, Russia; patrakeev@mail.ru}
}

\title{The complement of a $\hspace{1pt}\mathsurround=0pt\sigma$-compact subset
of a space with a $\hspace{1pt}\mathsurround=0pt\pi$-tree also has a $\hspace{1pt}\mathsurround=0pt\pi$-tree\footnote{2010 Mathematics Subject Classification\textup{:} Primary 54{E}99; Secondary 54{H}05. Keywords\textup{:} the Sorgenfrey line, the Baire space, Souslin scheme, Lusin scheme, Lusin pi-base, pi-tree, foliage tree, the foliage hybrid operation, sigma-compact.}}
\date{}

\begin{document}
\hyphenation{pa-ra-com-pact no-n-in-cre-a-s-ing tran-si-ti-ve re-bu-ild-ing sche-me dis-jo-int-ness}
\renewcommand{\proofname}{\textup{\textbf{Proof}}}
\renewcommand{\abstractname}{\textup{Abstract}}
\renewcommand{\refname}{\textup{References}}
\maketitle
\begin{abstract}
We prove that the complement of a $\hspace{1pt}\mathsurround=0pt\sigma$-compact subset of a topological space that has a $\hspace{1pt}\mathsurround=0pt\pi$-tree also has a $\hspace{1pt}\mathsurround=0pt\pi$-tree. To do this, we construct the \emph{foliage hybrid operation}, which deals with \emph{foliage trees} (that is, set-theoretic trees with a `leaf' at each node). Then using this operation we modify a $\hspace{1pt}\mathsurround=0pt\pi$-tree of a space and get a $\hspace{1pt}\mathsurround=0pt\pi$-tree for its subspace.
\end{abstract}

\section{Introduction}
\label{section0}

We study topological spaces that have a $\hspace{1pt}\mathsurround=0pt\pi$-tree; this notion is equivalent to the notion of a Lusin $\hspace{1pt}\mathsurround=0pt\pi$-base, which was introduced in~\cite{MPatr} (see details in Definition~\ref{def.B.f.tree} and Remark~\ref{l.eqiv.lusin}).
The Sorgenfrey line and the Baire space $\mathcal{N}$ (that is, ${}^\omega\hspace{-1pt}\omega$ with the product topology) are examples of spaces with a $\hspace{1pt}\mathsurround=0pt\pi$-tree~\cite{MPatr}.
Every space that has a $\hspace{1pt}\mathsurround=0pt\pi$-tree shares many good properties with the Baire space. One reason for this is expressed in Lemma~\ref{lem.pi.and.B.f.trees.vs.S}, another two are the following:
if a space ${X}\hspace{-1pt}$ has a $\hspace{1pt}\mathsurround=0pt\pi$-tree, then ${X}\hspace{-1pt}$ can be mapped onto $\mathcal{N}$ by a continuous one-to-one map~\cite{MPatr} and also ${X}\hspace{-1pt}$ can be mapped onto $\mathcal{N}$ by a continuous open map~\cite{MPatr} (hence ${X}\hspace{-1pt}$ can be mapped  by a continuous open map onto an arbitrary Polish space).

In this paper we prove  Theorem~\ref{corollary.2}, which states that if a space ${X}\hspace{-1pt}$ has a $\hspace{1pt}\mathsurround=0pt\pi$-tree and ${Y}\subseteq{X}\hspace{-1pt}$ is the complement of a $\hspace{1pt}\mathsurround=0pt\sigma$-compact subset of ${X},$ then ${Y}\hspace{-1.7pt}$ also has a $\hspace{1pt}\mathsurround=0pt\pi$-tree.
This result reflects the following property of the Baire space: if ${Y}\hspace{-0.5pt}\subseteq\mathcal{N}$ is the complement of a $\hspace{1pt}\mathsurround=0pt\sigma$-compact subset of $\mathcal{N},$ then ${Y}\hspace{-1.5pt}$ is homeomorphic to $\mathcal{N}$ (this property of $\mathcal{N}$ can be easily derived from the Alexandrov-Urysohn characterization of the Baire space and from the characterization of its Polish subspaces --- see Theo\-rems~3.11 and~7.7 in~\cite{kech}).

Theorem~\ref{corollary.2} is a corollary to Theorem~\ref{corollary.1}, which in combination with Lemma~\ref{lem.pi.and.B.f.trees.vs.S} allows to find many more subspaces ${Y}\hspace{-1pt}$ of a space ${X}\hspace{-1pt}$ with a $\hspace{1pt}\mathsurround=0pt\pi$-tree such that ${Y}\hspace{-1pt}$ also has a $\hspace{1pt}\mathsurround=0pt\pi$-tree; for example, a dense ${Y}\hspace{-1pt}$ such that $|{X}\!\setminus\!{Y}|={2}^{\aleph_{0}}$ (Theorem~\ref{corollary.2} does not allow to find such  ${Y}\hspace{-1pt}$ in the Sorgenfrey line because every $\hspace{1pt}\mathsurround=0pt\sigma$-compact subset of the Sorgenfrey line is at most countable). In contrast to Theorems~\ref{corollary.1} and~\ref{corollary.2}, a dense open subspace ${Y}\hspace{-1pt}$ of a space ${X}\hspace{-1pt}$ that has a $\hspace{1pt}\mathsurround=0pt\pi$-tree can be without $\hspace{1pt}\mathsurround=0pt\pi$-tree even if ${X}\hspace{-1pt}$ is separable metrizable (this result is in preparation for publication).
Both Theorems~\ref{corollary.1} and~\ref{corollary.2} are corollaries to Theorem~\ref{main.theorem}, which is the main technical result of this paper.

\section{Notation and terminology}

We use standard set-theoretic notation from \cite{kun,jech}, according to which $\hspace{1pt}\mathsurround=0pt\omega=$ the set of natural numbers $\mathsurround=0pt =$ the set of finite ordinals $\mathsurround=0pt =$ the first limit ordinal $\mathsurround=0pt=$ the first infinite cardinal $\mathsurround=0pt =\aleph_{0},$ and each ordinal is equal to the set of smaller ordinals, so that ${n}=\{0,\ldots,{n}\,{-}\,1\}$ for all ${n}\in\omega.$
We use terminology from~\cite{top.enc} when we work with (topological) spaces.
Also we use several less common notations:

\label{section1}
\begin{notation}\label{not01}%
   The symbol $\coloneq$ means ``equals by definition''\textup{;}
   the symbol ${\colon}{\longleftrightarrow}$ is used to show that an expression on the left side is the abbreviation for expression on the right side\textup{;}

   \begin{itemize}
   \item [\ding{46}\ ]
      $\mathsurround=0pt
      {x}\subset{y}
      \quad{\colon}{\longleftrightarrow}\quad
      {x}\subseteq{y}\ \enskip\mathsf{and}\enskip\ {x}\neq{y};$
   \item [\ding{46}\ ]
      $\mathsurround=0pt
      \forall{v}\,{\neq}\,{w}\,{\in}\,{A}\ \varphi({v},{w})
      \quad{\colon}{\longleftrightarrow}\quad
      \forall{v},{w}\,{\in}\,\! {A}\;\big[\,{v}\neq {w}\,\to\,\varphi({v},{w})\,\big];$
   \item [\ding{46}\ ]
      $\mathsurround=0pt
      \exists\,{v}\,{\neq}\,{w}\,{\in}\,{A}\ \varphi({v},{w})
      \quad{\colon}{\longleftrightarrow}\quad
      \exists\,{v},{w}\in\! {A}\;\big[\,{v}\neq {w}\,\ \mathsf{and}\ \,\varphi({v},{w})\,\big];$
   \item [\ding{46}\ ]
      $\mathsurround=0pt
      {A}\,\equiv\,\bigsqcup_{{\lambda}\in{\Lambda}}{B}_{\lambda}
      \quad{\colon}{\longleftrightarrow}\quad
      {A}=\bigcup_{{\lambda}\in{\Lambda}}{B}_{\lambda}\enskip\mathsf{and}\enskip
      \forall{\lambda}\,{\neq}\,{\lambda}'\,{\in}\,{\Lambda}
      \;[\,{B}_{\lambda}\cap{B}_{{\lambda}'}=\varnothing\,];$
   \item [\ding{46}\ ]
      $\mathsurround=0pt
      {A}\equiv {B}_{0}\sqcup\ldots\sqcup {B}_{n}
      \quad{\colon}{\longleftrightarrow}\quad
      {A}\equiv\bigsqcup_{\,{i}\in\{{0},\ldots,{n}\}}{B}_{i}\,.$
   \end{itemize}
\end{notation}

When we work with sequences, we use the following notation:

\begin{notation}\label{not02}%
   Suppose that $\alpha,\beta$ are ordinals, ${n}\in\omega,$ and ${s},{t}$ are transfinite sequences (that is, ${s}$ and ${t}$ are functions whose domains are ordinals). Then\textup{:}

   \begin{itemize}
   \item [\ding{46}\ ]
      $\mathsurround=0pt
      \length\hspace{0.5pt}{s}\coloneq$ the domain of ${s};$
   \item [\ding{46}\ ]
      $\mathsurround=0pt
      \langle\hspace{1pt}{r}_0,\ldots,{r}_{{n}-1}\rangle\coloneq$
      the sequence ${r}$ such that  $\length{r}={n}$ and ${r}(i)={r}_{i}$ for all ${i}<{n};$

      in particular, $\langle\rangle\coloneq\!$ the empty sequence ($\mathsurround=0pt=$ the empty set);
   \item [\ding{46}\ ]
      $\mathsurround=0pt
      {}^{x}\hspace{-2pt}{A}\coloneq$ the set of functions from ${x}$ to ${A};$

      in particular, ${}^{0}\hspace{-2pt}{A}=\big\{\langle\rangle\big\};$
   \item [\ding{46}\ ]
      $\mathsurround=0pt
      {}^{<\hspace{0.2pt}\alpha}\hspace{-2pt}{A}
      \coloneq\bigcup_{\beta<\hspace{0.3pt}\alpha}{}^\beta\hspace{-2.5pt}{A};$

      in particular, ${}^{<\hspace{0.2pt}\omega}\hspace{-2.5pt}{A}$ is the set of finite sequences in ${A};$
   \item [\ding{46}\ ]
      if ${s}=\langle{s}_0,\ldots,{s}_{{n}-1}\rangle,$ then

      $\mathsurround=0pt
      {s}\hspace{2pt}\hat{}\hspace{1.4pt}\langle\hspace{-0.3pt}{a}\hspace{-0.5pt}\rangle
      \coloneq\langle{s}_0,\ldots,{s}_{{n}-1},{a}\rangle;$
   \item [\ding{46}\ ]
      $\mathsurround=0pt
      {s}\hspace{0.5pt}{\upharpoonright}\,{x}\coloneq$ the restriction of ${s}$ to ${x};$

      in particular,
      ${s}\hspace{0.5pt}{\upharpoonright}\,0=\langle\rangle\ $ for any ${s};$
   \item [\ding{46}\ ]
      note that
      
      $\mathsurround=0pt
      {s}\subseteq{t}
      \enskip\;\:\mathsf{iff}\enskip\:
      \length\hspace{0.5pt}{s}\leqslant\length\hspace{1.5pt}{t}\,\enskip\text{and}\enskip\,
      {s}={t}\hspace{1pt}{\upharpoonright}\hspace{1pt}\length{s}.$
   \end{itemize}
\end{notation}

Also we work with partial orders and we use the following notation:
\begin{notation}
   Suppose that
   $\mathcal{P}=( {Q}, <)$ is a strict partially ordered set; that is, $<$ is irreflexive and transitive on~${Q}.$ Let ${x},{y}\in{Q}$ and ${A},{B}\subseteq {Q}.$
   Then\textup{:}

   \begin{itemize}
   \item [\ding{46}\ ]
      $\mathsurround=0pt
      \nodes\mathcal{P}=\nodes\,( {Q}, <)
      \coloneq
      {Q}$

      (we use the word node because we intend to work with trees);
   \item [\ding{46}\ ]
      $\mathsurround=0pt
      {x}<_\mathcal{P} {y}
      \quad{\colon}{\longleftrightarrow}\quad
      {x}< {y};$
   \item [\ding{46}\ ]
      $\mathsurround=0pt
      {x}\leqslant_\mathcal{P} {y}
      \quad{\colon}{\longleftrightarrow}\quad
      {x}<_\mathcal{P} {y}\enskip\mathsf{or}\enskip{x}={y};$
   \item [\ding{46}\ ]
      $\mathsurround=0pt
      {x}\:{\parallel}_{\hspace{-1pt}\mathcal{P}}\: {y}
      \quad{\colon}{\longleftrightarrow}\quad
      {x}\nleqslant_\mathcal{P}\! {y}\enskip\mathsf{and}\enskip{x}\ngtr_\mathcal{P}\!{y};$
   \item [\ding{46}\ ]
      $\mathsurround=0pt
      {x}\hspace{-0.3pt}{\upspoon}_{\!\mathcal{P}}\coloneq\{{v}\in
      \nodes\mathcal{P}:{v}<_\mathcal{P} {x}\},\quad
      {x}\hspace{-0.3pt}{\downspoon}_{\hspace{-0.5pt}\mathcal{P}}\coloneq
      \{{v}\in\nodes\mathcal{P}:{v}>_\mathcal{P} {x}\};$
   \item [\ding{46}\ ]
      $\mathsurround=0pt
      {x}\hspace{-0.5pt}{\upfilledspoon}_{\!\mathcal{P}}\coloneq
      \{{v}\in\nodes\mathcal{P}:{v}\leqslant_\mathcal{P} {x}\},\quad {x}\hspace{-0.7pt}{\downfilledspoon}_{\hspace{-0.5pt}\mathcal{P}}\coloneq
      \{{v}\in\nodes\mathcal{P}:{v}\geqslant_\mathcal{P} {x}\};$
   \item [\ding{46}\ ]
      $\mathsurround=0pt
      {A}{\upfootline}_{\!\mathcal{P}}\coloneq
      \bigcup\{{v}\hspace{-0.5pt}{\upfilledspoon}_{\!\mathcal{P}}:{v}\in{A}\},\quad
      {A}{\downfootline}_\mathcal{P}\coloneq
      \bigcup\{{v}\hspace{-0.5pt}{\downfilledspoon}_{\hspace{-0.5pt}\mathcal{P}}:{v}\in{A}\};$
   \item [\ding{46}\ ]
      $\mathsurround=0pt
      \hspace{-0.5pt}(\hspace{-0.5pt}{x},{y}\hspace{-0.5pt})_\mathcal{P}\coloneq
      {x}\hspace{-0.3pt}{\downspoon}_{\hspace{-0.5pt}\mathcal{P}}\,\cap\,
      {y}{\upspoon}_{\!\mathcal{P}},\quad
      [{x},{y}]_\mathcal{P}\coloneq {x}\hspace{-0.7pt}{\downfilledspoon}_{\hspace{-0.5pt}\mathcal{P}}\,\cap\,
      {y}{\upfilledspoon}_{\!\mathcal{P}};$
   \item [\ding{46}\ ]
      $\mathsurround=0pt
      [{x},{y}\hspace{-0.5pt})_\mathcal{P}\coloneq {x}\hspace{-0.7pt}{\downfilledspoon}_{\hspace{-0.5pt}\mathcal{P}}\,\cap\,
      {y}{\upspoon}_{\!\mathcal{P}},\quad
      (\hspace{-0.5pt}{x},{y}]_\mathcal{P}\coloneq {x}\hspace{-0.3pt}{\downspoon}_{\hspace{-0.5pt}\mathcal{P}}\,\cap\,
      {y}{\upfilledspoon}_{\!\mathcal{P}};$
   \item [\ding{46}\ ]
      $\mathsurround=0pt
      {x}\hspace{1pt}\sqsubset_{\mathcal{P}}{y}
      \quad{\colon}{\longleftrightarrow}\quad
      {x}<_\mathcal{P}{y}\enskip\mathsf{and}\enskip(\hspace{-0.5pt}{x},{y}\hspace{-0.5pt})_\mathcal{P}=\varnothing;$
   \item [\ding{46}\ ]
      $\mathsurround=0pt
      \sons_{\hspace{0.5pt}\mathcal{P}\hspace{-0.5pt}}(\hspace{-0.7pt}{x}\hspace{-1pt})\coloneq
      \{\,{s}\in\nodes\mathcal{P}\,:\,{x}\hspace{1pt}\sqsubset_{\mathcal{P}}{s}\,\};$
   \item [\ding{46}\ ]
      $\mathsurround=0pt
      {A}\,$ is $\hspace{1pt}\mathsurround=0pt\mathcal{P}$-\textbf{cofinal} in $\!{B}
      \quad{\colon}{\longleftrightarrow}\quad
      {A}\,\subseteq {B}\enskip\mathsf{and}\enskip{B}\,\subseteq {A}{\upfootline}_{\!\mathcal{P}};$
   \item [\ding{46}\ ]
      $\mathsurround=0pt
      {A}\:$ is an \textbf{antichain} in $\mathcal{P}
      \quad{\colon}{\longleftrightarrow}\quad
      \forall{v}\,{\neq}\,{w}\,{\in}\,{A}\:[\,{v}\:{\parallel}_{\hspace{-1pt}\mathcal{P}}\: {w}\,];$
   \item [\ding{46}\ ]
      $\mathsurround=0pt{A}
      \:$ is a \textbf{chain} in $\mathcal{P}
      \quad{\colon}{\longleftrightarrow}\quad
      \forall{v},{w}\,{\in}\, {A}
      \:[\,
      {v}\leqslant_\mathcal{P} {w}\ \ \mathsf{or}\ \ {v}>_\mathcal{P} {w}
      \,];$
   \item [\ding{46}\ ]
      $\mathsurround=0pt
      \mathcal{P}\,$ has \textbf{bounded chains}
      $\!\quad{\colon}{\longleftrightarrow}\quad\!$

      for each nonempty chain ${C}$ in $\mathcal{P}$ there is ${z}\,{\in}\nodes\mathcal{P}$ such that
      ${C}\subseteq {z}\hspace{-0.5pt}{\upfilledspoon}_{\!\mathcal{P}};$
   \item [\ding{46}\ ]
      $\mathsurround=0pt
      \maxel{\mathcal{P}}\coloneq
      \{\,{m}\in\nodes\mathcal{P}\::\:
      {m}\hspace{-0.4pt}{\downspoon}_{\hspace{-0.5pt}\mathcal{P}}=\varnothing\,\},\quad
      \minel{\mathcal{P}}\coloneq
      \{\,{m}\in\nodes\mathcal{P}\::\:
      {m}\hspace{-0.6pt}{\upspoon}_{\!\mathcal{P}}=\varnothing\,\};$
   \item [\ding{46}\ ]
      for $\mathcal{P}$ with the least node,

      ${0}_{\mathcal{P}}\coloneq\mathsurround=0pt$ the least node of $\mathcal{P}.$
   \end{itemize}
\end{notation}

When a partially ordered set is a (set-theoretic) tree~\cite{jech,kun}, we use the following terminology:

\begin{notation}\label{not.trees}
   Suppose that
   $\mathcal{T}\hspace{-1pt}$ is a tree;  that is, $\mathcal{T}\hspace{-1pt}$ is a strict partially ordered set
   such that for each ${x}\in\nodes\mathcal{T}\hspace{-1pt},$ the set ${x}\hspace{-0.3pt}{\upspoon}_{\!\mathcal{T}}$ is well-ordered by $<_{\,\mathcal{T}}.$  Let ${x}\in\nodes\mathcal{T}\hspace{-1pt},$ let $\alpha$ be  an ordinal, and let $\kappa$ be a cardinal.  Then\textup{:}

   \begin{itemize}
   \item [\ding{46}\ ]
      $\mathsurround=0pt
      \height_{\mathcal{T}}(\hspace{-0.7pt}{x}\hspace{-1pt})\coloneq$
      the ordinal isomorphic to $( {x}\hspace{-0.3pt}{\upspoon}_{\!\mathcal{T}}, <_{\,\mathcal{T}});$
   \item [\ding{46}\ ]
      $\mathsurround=0pt
      \level_{\mathcal{T}}(\hspace{-0.7pt}\alpha\hspace{-0.8pt})\coloneq\big\{{v}\in\nodes\mathcal{T}:
      \height_{\mathcal{T}}(\hspace{-0.3pt}{v}\hspace{-0.3pt})=\alpha\big\};
      $
   \item [\ding{46}\ ]
      $\mathsurround=0pt
      \height\hspace{-0.4pt}\mathcal{T}\coloneq$
      the minimal ordinal $\beta$ such that $\level_{\mathcal{T}}(\beta)=\varnothing;$
   \item [\ding{46}\ ]
      $\mathsurround=0pt
      {B}\:$ is a \textbf{branch} in $\mathcal{T}
      \quad{\colon}{\longleftrightarrow}\quad
      {B}$ is a $\subseteq\mathsurround=0pt $-maximal chain in $\mathcal{T};$
   \item [\ding{46}\ ]
      $\mathsurround=0pt
      \branches\mathcal{T}\coloneq
      \{B\subseteq\nodes\mathcal{T}:{B}\,$ is a branch in $\!\mathcal{T}\hspace{1pt}\};$
   \item [\ding{46}\ ]
      if ${A}\subseteq\nodes\mathcal{T}\hspace{-1pt}$ is an antichain in $\mathcal{T}\hspace{-1pt}$ and ${x}\in{A}{\downfootline}_{\mathcal{T}\hspace{-1pt}},$ then

      $\mathsurround=0pt
      \roott_{\mathcal{T}}({x},{A})\coloneq$
      the ${r}$ in ${A}$ such that ${x}\in{r}{\downfilledspoon}_{\mathcal{T}};$
   \item [\ding{46}\ ]
%      for a cardinal number $\kappa,$
%
      $\mathsurround=0pt
      \mathcal{T}$ is $\kappa$-\textbf{branching} $\quad{\colon}{\longleftrightarrow}\quad
      \forall{v}\,{\in}\nodes\mathcal{T}\hspace{-1pt}\setminus\maxel\mathcal{T}
      \,\big[\;
      |\hspace{-1pt}\sons_{\hspace{0.6pt}\mathcal{T}}({v})\,|=\kappa
      \;\big];$
   \item [\ding{46}\ ]
      $\mathsurround=0pt
      \mathcal{T}$ is an $\alpha,\hspace{-1pt}\kappa\hspace{1pt}$-\hspace{1pt}\textbf{tree} $\quad{\colon}{\longleftrightarrow}\quad
      \mathcal{T}$ is isomorphic to the tree
      $( {}^{<\hspace{0.2pt}\alpha}\hspace{-1pt}\kappa, {\subset}).$
   \end{itemize}
\end{notation}

The following example illustrates the usage of the above terminology:

\begin{pri}\label{pri_tree}
Let $\mathcal{T}=( {}^{<\hspace{0.2pt}\omega}\hspace{-2.5pt}{A}, {\subset}),$ where ${A}$ is nonempty. Then $\mathcal{T}\hspace{-1pt}$ is an $|{A}|\mathsurround=0pt $-branching tree with the least node,
$\nodes\mathcal{T}={}^{<\hspace{0.2pt}\omega}\hspace{-2.5pt}{A},$ ${0}_{\mathcal{T}}=\langle\rangle,$  and
$\maxel\mathcal{T}=\varnothing.$
Suppose that ${a},{b},{c},{d}\in{A}$ are different. Then we have\textup{:}
$$
\langle{c},{a},{b},{a}\rangle{\upspoon}_{\!\mathcal{T}} = \big\{\langle\rangle, \langle{c}\rangle, \langle{c},{a}\rangle, \langle{c},{a},{b}\rangle\big\},\quad
\height_{\hspace{0.6pt}\mathcal{T}}\big(\langle{c},{a},{b},{a}\rangle\big)=4,\quad
\height_{\hspace{0.6pt}\mathcal{T}}\big(\langle\rangle\big)=0,
$$
$$\level_{\mathcal{T}}(2)=\big\{\langle{x},{y}\rangle:{x},{y}\in{A}\,\big\},\quad
\level_{\mathcal{T}}(0)=\big\{\langle\rangle\big\},\quad
\level_{\hspace{0.6pt}\mathcal{T}}(\omega\hspace{-0.6pt})=\varnothing,\quad
\height\hspace{-0.4pt}\mathcal{T}=\omega,$$
$$
\sons_{\hspace{0.6pt}\mathcal{T}}\big(\langle{c},{a}\rangle\big)=\big\{\langle{c},{a},{x}\rangle: {x}\in{A}\big\},\quad
\roott_{\hspace{0.6pt}\mathcal{T}}      \Big(\langle{c},{b},{a},{d}\rangle,\big\{\langle\hspace{-0.7pt}{a}\hspace{-1pt}\rangle,\langle{c},{b}\rangle,
\langle{d}\rangle\big\}\Big)=\langle{c},{b}\rangle.
$$
\end{pri}

Also we list here several simple facts about trees, which we use in this paper:

\begin{lem}\label{about trees}
   Suppose that $\mathcal{T}\hspace{-1pt}$ is a tree. Then\textup{:}

   \begin{itemize}
   \item[\textup{(a)}]
      $\mathsurround=0pt\maxel\mathcal{T}=\big\{{v}\,{\in}\hspace{1pt}\nodes\mathcal{T}:\,
      \sons_{\hspace{0.6pt}\mathcal{T}}({v})=\varnothing\big\}.$
   \item[\textup{(\hspace{0.3pt}b\hspace{-1pt})}]
      If ${x},{y},{z}\in\nodes\mathcal{T}\hspace{-1pt},$ ${x}\geqslant_{\mathcal{T}}\!{y},$ and ${y}\:{\parallel}_{\!\mathcal{T}}\:{z},$

      then ${x}\:{\parallel}_{\!\mathcal{T}}\:{z}.$
   \item[\textup{(c)}]
      If ${C}$ is a chain in $\mathcal{T}\hspace{-1pt},$

      then there is ${B}\in\branches\mathcal{T}\hspace{-1pt}$ such that ${C}\subseteq{B}.$
   \item[\textup{(\hspace{-0.4pt}d)}]
      If ${B}\,{\in}\hspace{0.3pt}\branches\mathcal{T}\hspace{-1pt}$ and ${x}\,{\in}\,(\nodes\mathcal{T}\,)\,{\setminus}\,{B},$

      then there is ${b}\in\hspace{1pt}\!{B}$ such that ${x}\:{\parallel}_{\!\mathcal{T}}\:{b}.$
   \item[\textup{(e)}]
      If ${B}\,{\in}\branches\mathcal{T}\hspace{-1pt}$ and  ${b}\in\hspace{1pt}\!{B}\,{\setminus}\maxel\mathcal{T}\hspace{-1pt},$

      then ${B}\cap\hspace{0.4pt}\sons_\mathcal{T}({b})\neq\varnothing.$
   \item[\textup{(f)}]
      If ${B}\,{\in}\branches\mathcal{T}\hspace{-1pt}$ and ${b}\in\hspace{1pt}\!{B},$

      then ${b}\hspace{0.3pt}{\upfilledspoon}_{\!\mathcal{T}}\subseteq {B}.$
   \item[\textup{(g)}]
      If ${m}\in\maxel\mathcal{T}\hspace{-1pt},$

      then ${m}{\upfilledspoon}_{\!\mathcal{T}}$ is a branch in $\mathcal{T}\hspace{-1pt}.$
   \item[\textup{(h)}]
      If $\mathcal{T}\hspace{-1pt}$ has bounded chains,

      then $\branches\mathcal{T}=\{{m}{\upfilledspoon}_{\!\mathcal{T}}:{m}\in\maxel\mathcal{T}\hspace{1pt}\}.$
   \item[\textup{(i)}]
      The following are equivalent\textup{:}
      \begin{itemize}
         \item[\ding{226}\ ]
            $\mathcal{T}\hspace{-1pt}$ is an $\omega,\aleph_{0}\mathsurround=0pt $-tree.
         \item[\ding{226}\ ]
            $\mathcal{T}\hspace{-1pt}$ has the least node,
            $\mathcal{T}\hspace{-1pt}$ is $\aleph_0\mathsurround=0pt $-branching,
            $\maxel\mathcal{T}=\varnothing,$
            and $\height\hspace{-0.4pt}\mathcal{T}\hspace{-1pt}\leqslant\omega.$
      \hfill$\Box$
      \end{itemize}
   \end{itemize}
\end{lem}

\section{Foliage trees}
\label{sect.f.trees}

Informally, a foliage tree is a tree with a leaf at each node, where by a leaf we mean an arbitrary set.
Here is the formal definition:

\begin{deff}\label{def.fol.tree}
   {A} \textbf{foliage tree} is a pair $\mathbf{F}=(\mathcal{T}\hspace{-1pt},{l}) $ such that
   $\mathcal{T}\hspace{-1pt}$ is a (set-theoretic) tree
   and ${l}$~is a function with $\domain{l}=\nodes\mathcal{T}.$
   For each ${x}\in\nodes\mathcal{T}\hspace{-1pt},$ the ${l}(\hspace{-0.7pt}{x}\hspace{-1pt})$ is called the \textbf{leaf} of $\mathbf{F}$ at node ${x}$ and is denoted by $\mathbf{F}_{\hspace{-1.2pt}{x}}.$
   The tree $\mathcal{T}\hspace{-1pt}$ is called the \textbf{skeleton} of~$\mathbf{F}$ and is denoted by $\skeleton\mathbf{F}.$
\end{deff}

\begin{convention}\label{convention}
   Let $\mathbf{F}$ be a foliage tree and let $\mathcal{O}$ be an operation or a notion that is defined on trees. Then we use $\mathcal{O}(\mathbf{F})$ as the abbreviation for
   $\mathcal{O}(\hspace{-0.5pt}\skeleton\mathbf{F} ).$ For example,

   \begin{itemize}
   \item [\ding{46}\ ]
      $\mathsurround=0pt
      \nodes\mathbf{F}\coloneq\nodes\hspace{1.2pt}(\hspace{-1pt}\skeleton\mathbf{F}),$
   \item [\ding{46}\ ]
      $\mathsurround=0pt
      {0}_\mathbf{F}\coloneq{0}_{\skeleton\mathbf{F}},$
   \item [\ding{46}\ ]
      $\mathsurround=0pt
      \mathbf{F}\,$ has bounded chains
      $\quad{\colon}{\longleftrightarrow}\quad
      \!\skeleton\mathbf{F}$ has bounded chains,
   \item [\ding{46}\ ]
      $\mathsurround=0pt
      {x}\,\sqsubset_\mathbf{F}\,{y}
      \quad{\colon}{\longleftrightarrow}\quad
      {x}\,\sqsubset_{\skeleton\mathbf{F}}\,{y}.$
   \end{itemize}
\end{convention}

\begin{notation}\label{not.flesh}
   Let $\mathbf{F}$ be a foliage tree, let $\varnothing\neq{A}\subseteq\nodes\mathbf{F},$ and let ${z}\in\nodes\mathbf{F}.$
   Then\textup{:}

   \begin{itemize}
   \item [\ding{46}\ ]
      $\mathsurround=0pt
      \fruit_{\hspace{0.4pt}\mathbf{F}}\hspace{-1pt}(\hspace{-0.8pt}{A}\hspace{-0.2pt})\coloneq
      \bigcap\{\mathbf{F}_{\hspace{-1.2pt}{x}}:{x}\,{\in}\,{A}\};$
   \item [\ding{46}\ ]
      $\mathsurround=0pt
      \yield\mathbf{F}\coloneq
      {\bigcup}\big\{\!\fruit_{\hspace{0.4pt}\mathbf{F}}\hspace{-1pt}(\hspace{-1.1pt}{B}\hspace{-0.6pt}):
      {B}\in\branches\mathbf{F}\big\};$
   \item [\ding{46}\ ]
      $\mathsurround=0pt
      \flesh\hspace{-1pt}\mathbf{F}\coloneq
      \bigcup\{\mathbf{F}_{\hspace{-1.2pt}{x}}:{x}\in\nodes\mathbf{F}\};$
   \item [\ding{46}\ ]
      $\mathsurround=0pt
      \flesh_\mathbf{F}\hspace{-0.5pt}(\hspace{-0.8pt}{A}\hspace{-0.2pt})\coloneq
      \bigcup\{\mathbf{F}_{\hspace{-1.2pt}{x}}:{x}\,{\in}\,{A}\};$
   \item [\ding{46}\ ]
      $\mathsurround=0pt
      \shoot_\mathbf{F}\hspace{-0.5pt}(\hspace{-0.5pt}{z}\hspace{-0.5pt})\coloneq
      \big\{\flesh_\mathbf{F}\hspace{-0.5pt}(\hspace{-0.5pt}{C}):
      {C}$ is a cofinite subset of $\sons_{\hspace{0.5pt}\mathbf{F}\hspace{-0.5pt}}(\hspace{-0.5pt}{z}\hspace{-0.5pt})\,\big\};$
%      We say that $\mathbf{F}$ is \textbf{isomorphic} to $\mathbf{G}$ iff there exists an isomorphism between $\mathbf{F}$ and $\mathbf{G};$
   \item [\ding{46}\ ]
      $\mathsurround=0pt
      \scope_\mathbf{F}\hspace{-0.5pt}({p})\coloneq
      \{\hspace{0.5pt}{y}\,{\in}\nodes\mathbf{F}:\mathbf{F}_{\hspace{-1.2pt}{y}}\,{\ni}\,{p}\hspace{0.5pt}\};$
   \item [\ding{46}\ ]
      for a space ${X}$ and a point ${p}$ in ${X},$

      $\mathsurround=0pt\nbhds({p},{X})\coloneq$
      the family of (not necessarily open) neighbourhoods of ${p}$ in ${X};$
   \item [\ding{46}\ ]
      for arbitrary sets $\gamma$ and $\delta,$

      $\mathsurround=0pt
      \gamma\gg\delta
      \quad{\colon}{\longleftrightarrow}\quad
      \gamma\,$ $\pi\!$-\emph{refines} $\,\delta
      \quad{\colon}{\longleftrightarrow}\quad
      \forall\hspace{-1pt}{D}\,{\in}\:\delta\hspace{0.7pt}{\setminus}\{\varnothing\}
      \ \,\exists{G}\,{\in}\,\gamma\hspace{0.3pt}{\setminus}\{\varnothing\}
      \ \,[\,{G}\subseteq{D}\,].$
   \end{itemize}
\end{notation}

%\begin{notation}\label{notation.shoot}
%   Let $\mathbf{F}$ be a foliage tree
%   and let ${x}\in\nodes\mathbf{F}.$
%   Then
%
%   \begin{itemize}
%   \item [\ding{46}\ ]
%      $\mathsurround=0pt\shoot_\mathbf{F}\hspace{-0.5pt}(\hspace{-0.7pt}{x}\hspace{-1pt})\coloneq
%      \big\{\flesh_\mathbf{F}\hspace{-0.5pt}(\hspace{-1.1pt}{B}\hspace{-0.6pt}):
%      {B}$ is a cofinite subset of $\sons_{\hspace{0.5pt}\mathbf{F}\hspace{-0.5pt}}(\hspace{-0.7pt}{x}\hspace{-1pt})\,\big\}.$
%   \item [\ding{46}\ ]
%      $\mathsurround=0pt\scope_\mathbf{F}\hspace{-0.5pt}(\hspace{-0.3pt}{p}\hspace{-0.5pt})\coloneq
%      \{\hspace{0.5pt}{y}\,{\in}\nodes\mathbf{F}:\mathbf{F}_{\hspace{-1.2pt}{y}}\,{\ni}\,{p}\hspace{0.5pt}\}.$
%   \end{itemize}
%\end{notation}

\begin{deff}\label{def.B.f.tree}
   Let $\mathbf{F}$ be a foliage tree, ${X}$ a space, $\alpha$ an ordinal, and $\kappa$ a cardinal.
   Then\textup{:}

   \begin{itemize}
   \item [\ding{46}\ ]
      $\mathsurround=0pt
      \mathbf{F}\,$ has \textbf{nonempty leaves}$
      \quad{\colon}{\longleftrightarrow}\quad
      \forall{x}\,{\in}\nodes\mathbf{F}
       \ [\,\mathbf{F}_{\hspace{-1.2pt}{x}}\neq\varnothing
      \,];$
   \item [\ding{46}\ ]
      $\mathsurround=0pt
      \mathbf{F}\,$ is \textbf{nonincreasing}$
      \quad{\colon}{\longleftrightarrow}\quad
      \forall{x},{y}\,{\in}\nodes\mathbf{F}
      \ [\,
      {y}\geqslant_{\mathbf{F}} {x}\ \rightarrow\ \mathbf{F}_{\hspace{-1.2pt}{y}}\subseteq\mathbf{F}_{\hspace{-1.2pt}{x}}
      \,];$
   \item [\ding{46}\ ]
      $\mathsurround=0pt
      \mathbf{F}\,$ is \textbf{splittable}$
      \quad{\colon}{\longleftrightarrow}\quad
      \forall{x},{y}\,{\in}\nodes\mathbf{F}
      \ [\,
      {x}\:{\parallel}_{\!\mathbf{F}}\: {y}\ \rightarrow\ \mathbf{F}_{\hspace{-1.2pt}{x}}\cap\mathbf{F}_{\hspace{-1.2pt}{y}} =\varnothing
      \,];$
   \item [\ding{46}\ ]
      $\mathsurround=0pt
      \mathbf{F}\,$ is \textbf{complete}$
      \quad{\colon}{\longleftrightarrow}\quad
      \nodes\mathbf{F}\neq\varnothing\enskip\mathsf{and}\enskip
      \forall{B}\,{\in}\branches\mathbf{F}
      \ \big[\hspace{-0.7pt}
      \fruit_{\hspace{0.4pt}\mathbf{F}}\hspace{-1pt}(\hspace{-1.1pt}{B}\hspace{-0.6pt})\neq\varnothing
      \,\big];$
   \item [\ding{46}\ ]
      $\mathsurround=0pt
      \mathbf{F}\,$ has \textbf{strict branches}$
      \quad{\colon}{\longleftrightarrow}\quad
      \nodes\mathbf{F}\neq\varnothing\enskip\mathsf{and}\enskip
      \forall{B}\,{\in}\branches\mathbf{F}
      \ \big[\hspace{-0.7pt}
      \fruit_{\hspace{0.4pt}\mathbf{F}}\hspace{-1pt}(\hspace{-1.1pt}{B}\hspace{-0.6pt})\text{ is a singleton}
      \hspace{1pt}\big];$
   \item [\ding{46}\ ]
      $\mathsurround=0pt
      \mathbf{F}\,$ is \textbf{locally strict}$
      \quad{\colon}{\longleftrightarrow}\quad
      \forall{x}\,{\in}\nodes\mathbf{F}\,{\setminus}\maxel\mathbf{F}
      \ [\,\mathbf{F}_{\hspace{-1.2pt}{x}}\equiv
      \bigsqcup_{{s}\in\hspace{0.1pt}\sons_{\hspace{0.5pt}\mathbf{F}\hspace{-0.5pt}}
      (\hspace{-0.7pt}{x}\hspace{-1pt})}\mathbf{F}_{\hspace{-1.2pt}{s}}\,];$
   \item [\ding{46}\ ]
      $\mathsurround=0pt
      \mathbf{F}\,$ is \textbf{open} in  ${X}
      \quad{\colon}{\longleftrightarrow}\quad
      \forall{z}\,{\in}\nodes\mathbf{F}
      \ [\,\mathbf{F}_{\hspace{-1.2pt}{z}}$ is an open subset of ${X}
      \,];$
   \item [\ding{46}\ ]
      $\mathsurround=0pt
      \mathbf{F}\,$ is a \textbf{foliage} $\alpha,\hspace{-1pt}\kappa\hspace{1pt}$-\hspace{1pt}\textbf{tree}$
      \quad{\colon}{\longleftrightarrow}\quad
      \skeleton\mathbf{F}$
      is an $\alpha,\hspace{-1pt}\kappa\hspace{1pt}$-\hspace{1pt}tree \textup{(see Notation~\ref{not.trees});}
   \item [\ding{46}\ ]
      $\mathsurround=0pt
      \mathbf{F}\,$ is a \textbf{Baire foliage tree} on  ${X}
      \quad{\colon}{\longleftrightarrow}\quad
      \mathbf{F}$ is an open in ${X}$ locally strict foliage $\mathsurround=0pt \omega,\hspace{-1pt}\aleph_{0}\hspace{1pt}$-tree with strict branches
      and such that $\mathbf{F}_{\hspace{-1.2pt}{0}_{\hspace{-0.6pt}\mathbf{F}}}={X};$
   \item [\ding{46}\ ]
      $\mathsurround=0pt
      \mathbf{F}$ \textbf{grows into} ${X}
      \quad{\colon}{\longleftrightarrow}\quad
      \forall{p}\,{\in}\,{X}\ \forall{U}{\in}\nbhds({p},{X})\ \exists{z}\,{\in}\scope_\mathbf{F}\hspace{-0.5pt}({p})   \;\big[\,\shoot_\mathbf{F}\hspace{-0.5pt}(\hspace{-0.5pt}{z}\hspace{-0.5pt})\gg\{{U}\}\,\big];$
   \item [\ding{46}\ ]
      $\mathsurround=0pt
      \mathbf{F}\,$ is a $\hspace{1pt}\mathsurround=0pt\pi$-\textbf{tree} on  ${X}
      \quad{\colon}{\longleftrightarrow}\quad
      \mathbf{F}$ is a Baire foliage tree on ${X}$ and $\mathbf{F}$ grows into ${X}.$
   \end{itemize}
   Note that leaves of a $\hspace{1pt}\mathsurround=0pt\pi$-tree on ${X}\hspace{-1pt}$ are closed-and-open in ${X}\hspace{-1pt}$ and that the set of these leaves forms a countable
   $\hspace{1pt}\mathsurround=0pt\pi$-base and pseudo-base for ${X}.$
\end{deff}

The notion of a $\hspace{1pt}\mathsurround=0pt\pi$-tree is equivalent to the notion of a Lusin $\hspace{1pt}\mathsurround=0pt\pi$-base, which was introduced in~\cite{MPatr}; the only difference is that a Lusin $\hspace{1pt}\mathsurround=0pt\pi$-base is a family indexed by nodes of the tree
$({}^{<\hspace{0.2pt}\omega}\hspace{-1pt}\omega,\subset),$ while a $\hspace{1pt}\mathsurround=0pt\pi$-tree is a foliage tree whose skeleton is isomorphic to
$({}^{<\hspace{0.2pt}\omega}\hspace{-1pt}\omega,\subset).$
From a topological point of view, there is no difference between these two notions because of the following remark:

\begin{rem}\label{l.eqiv.lusin}
For any space ${X},$ the following are equivalent\textup{:}

   \begin{itemize}
   \item[\ding{226}\ ]
      $\mathsurround=0pt
      {X}\,$ has a $\hspace{1pt}\mathsurround=0pt\pi$-tree.
   \item[\ding{226}\ ]
      $\mathsurround=0pt
      {X}\,$ has a Lusin $\hspace{1pt}\mathsurround=0pt\pi$-base.
      \hfill$\Box$
   \end{itemize}
\end{rem}

Recall that the Baire space $\mathcal{N}$ is the set ${}^\omega\hspace{-1pt}\omega$ endowed with the Tychonov product topology, where $\omega$ carries the discrete topology. The Baire space has a basis $\big\{\{{p}\,{\in}\,{}^\omega\hspace{-1pt}\omega:{x}\subseteq{p}\}:
\hspace{1pt}{x}\,{\in}\,{}^{<\hspace{0.2pt}\omega}\hspace{-1pt}\omega\hspace{1pt}\big\},$ which is called~\cite{kech} the standard basis for ${}^\omega\hspace{-1pt}\omega.$ This standard basis can be viewed as a foliage tree:

\begin{notation}\label{def.stdrt.f.tr.of.B.sp}
   We denote by $\mathbf{S}$ the foliage tree such that

   \begin{itemize}
   \item[\ding{226}\ ]
      $\mathsurround=0pt
      \skeleton\mathbf{S}\coloneq({}^{<\hspace{0.2pt}\omega}\hspace{-1pt}\omega,\subset)\ \ $
      and
   \item[\ding{226}\ ]
      $\mathsurround=0pt
      \mathbf{S}_{{x}}\coloneq\{{p}\,{\in}\,{}^\omega\hspace{-1pt}\omega:{x}\subseteq{p}\}\ \ $
      for all ${x}\in{}^{<\hspace{0.2pt}\omega}\hspace{-1pt}\omega.$
   \end{itemize}
   We call this foliage tree the \textbf{standard foliage tree} of ${}^\omega\hspace{-1pt}\omega.$
\end{notation}

%\begin{notation}\label{def.stdrt.f.tr.of.B.sp}
%   The \textbf{standard foliage tree} of ${}^\omega\hspace{-1pt}\omega,$ which we denote by $\mathbf{S},$ is the foliage tree such that
%   \begin{itemize}
%   \item[\ding{226}\ ]
%      $\mathsurround=0pt
%      \skeleton\mathbf{S}\coloneq({}^{<\hspace{0.2pt}\omega}\hspace{-1pt}\omega,\subset)\ \ $
%      and
%   \item[\ding{226}\ ]
%      $\mathsurround=0pt
%      \mathbf{S}_{x}\coloneq\{{p}\,{\in}\,{}^\omega\hspace{-1pt}\omega:{x}\subseteq{p}\}\ \ $
%      for all ${x}\in{}^{<\hspace{0.2pt}\omega}\hspace{-1pt}\omega.$
%   \end{itemize}
%\end{notation}

%Recall that $\mathbf{S}$ is the standard foliage tree of ${}^\omega\hspace{-1pt}\omega$ and $\mathcal{N}=({}^\omega\hspace{-1pt}\omega,\tau_{\hspace{-1pt}\scriptscriptstyle\mathcal{N}})$ is the Baire space.
\begin{lem}\label{lem.pi.and.B.f.trees.vs.S}
   \begin{itemize}
   \item[\textup{(a)}]
      $\mathsurround=0pt
      \mathbf{S}$ is a $\hspace{1pt}\mathsurround=0pt\pi$-tree on the Baire space $\mathcal{N}=({}^\omega\hspace{-1pt}\omega,\tau_{\hspace{-1pt}\scriptscriptstyle\mathcal{N}}).$
   \item[\textup{(\hspace{0.3pt}b\hspace{-1pt})}]
      $\mathsurround=0pt
      \mathbf{S}$ is a Baire foliage tree on a space $({}^\omega\hspace{-1pt}\omega,\tau)
      \quad\textup{\textsf{iff}}\quad
      \tau\supseteq\tau_{\scriptscriptstyle\mathcal{N}}.$
%   \item[\textup{(c)}]
%      A space ${X}$ has a Baire foliage tree if and only if ${X}$ admits a weaker topology in which it is homeomorphic to the Baire space.
   \item[\textup{(c)}]
      A space ${X}\hspace{-1pt}$ has a Baire foliage tree
      \quad\textup{\textsf{iff}}\quad

      $\mathsurround=0pt
      {X}$ is homeomorphic to some space $({}^\omega\hspace{-1pt}\omega,\tau)$
      such that $\tau\supseteq\tau_{\scriptscriptstyle\mathcal{N}}.$
   \item[\textup{(\hspace{-0.4pt}d)}]
      A space ${X}\hspace{-1pt}$ has a $\hspace{1pt}\mathsurround=0pt\pi$-tree
      \quad\textup{\textsf{iff}}\quad

      $\mathsurround=0pt
      {X}$ is homeomorphic to some space $({}^\omega\hspace{-1pt}\omega,\tau)$
      such that $\hspace{1pt}\mathbf{S}\hspace{-1pt}$ is a $\hspace{1pt}\mathsurround=0pt\pi$-tree on $({}^\omega\hspace{-1pt}\omega,\tau).$
   \end{itemize}
\end{lem}

\begin{proof}
   Part (a) and the $\rightarrow$ direction of (\hspace{0.3pt}b\hspace{-0.5pt}) follow from the fact that $\{\mathbf{S}_{{x}}:{x}\,{\in}\,{}^{<\hspace{0.2pt}\omega}\hspace{-1pt}\omega\}$ is a basis for the Baire space. The $\leftarrow$ direction of (\hspace{0.3pt}b\hspace{-0.5pt}) follows from (a). The $\rightarrow$ direction of (c) is a reformulation of Lemma~3.3 from~\cite{MPatr} and the $\leftarrow$ direction of (c) follows from (\hspace{0.3pt}b\hspace{-0.5pt}). The $\rightarrow$ direction of (d) is a reformulation of Lemma~3.9 from~\cite{MPatr}, the opposite direction of (d) is trivial.
\end{proof}

\begin{lem}\label{about foliage trees}
   Suppose that $\mathbf{F}$ is a foliage tree. Then\textup{:}

   \begin{itemize}
   \item [\textup{(a)}]
      If $\,\mathbf{F}$ is nonincreasing, $\varnothing\,{\neq}\,{A}\subseteq{B}\subseteq\nodes\mathbf{F},$ and ${A}$ is $\hspace{1pt}\mathsurround=0pt\mathbf{F}$-cofinal in ${B},$

      then
      $\fruit_{\hspace{0.4pt}\mathbf{F}}\hspace{-1pt}(\hspace{-0.7pt}{A}\hspace{-0.4pt})=
      \fruit_{\hspace{0.4pt}\mathbf{F}}\hspace{-1pt}(\hspace{-1.1pt}{B}\hspace{-0.6pt}).$
   \item [\textup{(\hspace{0.3pt}b\hspace{-1pt})}]
      If $\,\mathbf{F}$ has the least node and $\height\hspace{-0.4pt}\mathbf{F}\leqslant\omega,$
      then the following are equivalent\textup{:}
      \begin{itemize}
      \item[\ding{226}\ ]
         $\mathbf{F}$ is locally strict\textup{;}
      \item[\ding{226}\ ]
         $\mathbf{F}$ is splittable  and $\,\displaystyle\flesh\hspace{-1pt}\mathbf{F}=\yield\mathbf{F}.$
      \hfill$\Box$
      \end{itemize}
   \end{itemize}
\end{lem}

\section{Hybrid operation}
\label{sect.hybr}

In this paper we build a $\hspace{1pt}\mathsurround=0pt\pi$-tree for a subspace ${Y}\hspace{-1pt}$ of a space ${X}\hspace{-1pt}$ that already has a $\hspace{1pt}\mathsurround=0pt\pi$-tree by using the \emph{foliage hybrid operation} (see Definition~\ref{def.f.hybr} in Section~\ref{sect.f.hybr}). Thе foliage hybrid operation deals with foliage trees and we construct it by using another operation --- the \emph{hybrid operation} --- which deals with trees.
These two operations are quite complicated, you can look at pictures that illustrate all definitions in~\cite{patr.slides}.

In this section we build the hybrid operation (see Definition~\ref{def.hybr}), prove that the result of the hybrid operation is always a tree (see Proposition~\ref{hybr.is.tree}), and establish properties of this operation (see Proposition~\ref{prop.hybr}).

The hybrid operation modifies a given tree $\mathcal{T}\hspace{-1pt}$ in two steps: first we cut out several pieces from $\mathcal{T}\hspace{-1pt},$ after that we engraft special trees onto the places of cut out pieces. The special trees that are engrafted onto $\mathcal{T}\hspace{-1pt}$ are called \emph{grafts}, the cut out pieces are called \emph{explants}, and the parts of grafts that replace explants are called \emph{implants}:

\begin{deff}\label{def.graft}
   Let
   $\mathcal{T}\hspace{-1pt}$ be a tree.
   Then
   a \textbf{graft} for $\mathcal{T}\hspace{-1pt}$ is a tree $\mathcal{G}
   $  such that\textup{:}

   \begin{itemize}
   \item[\textup{(a)}]
      $|\hspace{1pt}{\nodes}\,\mathcal{G}\hspace{1pt}|>1;$
   \item[\textup{(\hspace{0.3pt}b\hspace{-1pt})}]
      $\mathcal{G}$ has the least node;
   \item[\textup{(c)}]
      ${0}_\mathcal{G}\in\nodes\mathcal{T}$ and
      $\maxel\mathcal{G}\subseteq\nodes\mathcal{T}\hspace{-1pt};$
   \item[\textup{(\hspace{-0.4pt}d)}]
      $\maxel\mathcal{G}\subseteq
      (\hspace{-1pt}{0}_\mathcal{G}\hspace{-1.4pt})\hspace{-0.8pt}{\downspoon}_{\mathcal{T}};$
   \item[\textup{(e)}]
      $\maxel\mathcal{G} $
      is an antichain in $\mathcal{T}\hspace{-1pt};$
   \item[\textup{(f)}]
      $\impl\mathcal{G}\cap\nodes\mathcal{T}=\varnothing,$
   \end{itemize}
   where the set
   $$
   \impl\mathcal{G}\coloneq\nodes\mathcal{G}\setminus
   \big(\{\hspace{-1pt}{0}_\mathcal{G}\hspace{-1pt}\}\cup\hspace{0.4pt}\maxel\mathcal{G}\big)
   $$
   is called the \textbf{implant} of $\mathcal{G}.$
   The set
   $$
   \expl(\mathcal{T}\hspace{-1pt},\mathcal{G})\coloneq  (\hspace{-1pt}{0}_\mathcal{G}\hspace{-1.4pt})\hspace{-0.8pt}{\downspoon}_{\mathcal{T}}\!
   \setminus(\hspace{-0.4pt}\maxel\mathcal{G}){\downfootline}_{\mathcal{T}}
   $$
   is called the \textbf{explant} of $\mathcal{T}\hspace{-1pt}$ and
   $\mathcal{G}.$
\end{deff}

Note that $\maxel\mathcal{G}$ may be empty and then
$(\hspace{-0.4pt}\maxel\mathcal{G}){\downfootline}_{\mathcal{T}}=
\varnothing{\downfootline}_{\mathcal{T}}=\varnothing.$ The following example is given to clarify Definition~\ref{def.graft}.

\begin{pri}
   Suppose that $\mathcal{T}=( {}^{<\hspace{0.2pt}\omega}\hspace{-2.5pt}{A}, {\subset})$ is a tree from Example~\ref{pri_tree} and ${a},{b},{c},{d}\in{A}$ are different. Then
   $\big\{\langle{a},{d}\rangle,\langle{a},{b},{c}\rangle\big\}\subseteq \langle\hspace{-0.7pt}{a}\hspace{-1pt}\rangle\hspace{-0.8pt}{\downspoon}_{\mathcal{T}}$ and
   $\big\{\langle{a},{d}\rangle,\langle{a},{b},{c}\rangle\big\}$ is an antichain in $\mathcal{T}\hspace{-1pt}.$
   Let $\mathsf{IMP}$ be a set disjoint from $\nodes\mathcal{T}$
   and let $\mathcal{G}$ be a tree such that

   \begin{itemize}
   \item [\ding{226}]
      $\nodes\mathcal{G}\,=\big\{\langle\hspace{-0.7pt}{a}\hspace{-1pt}\rangle,\langle{a},{d}\rangle,\langle{a},{b},{c}\rangle\big\}\cup\hspace{1pt} \mathsf{IMP},$
   \item [\ding{226}]
      ${0}_\mathcal{G}=\langle\hspace{-0.7pt}{a}\hspace{-1pt}\rangle,\ $ and
   \item [\ding{226}]
      $\maxel\mathcal{G}=\big\{\langle{a},{d}\rangle,\langle{a},{b},{c}\rangle\big\}.$
   \end{itemize}
   Then $\mathcal{G}$ is a graft for $\mathcal{T}\hspace{-1pt},$
   $\impl\mathcal{G}=\mathsf{IMP},$ and
   $$\expl(\mathcal{T}\hspace{-1pt},\mathcal{G})=
   \big\{\,{s}\,{\in}\,{}^{<\hspace{0.2pt}\omega}\hspace{-2.5pt}{A}\,:\,
   \langle\hspace{-0.7pt}{a}\hspace{-1pt}\rangle\subset{s}\,\big\}\setminus
   \big\{\,{s}\,{\in}\,{}^{<\hspace{0.2pt}\omega}\hspace{-2.5pt}{A}\,:\,
   \langle{a},{d}\rangle\subseteq{s}\enskip\textsf{or}\enskip\langle{a},{b},{c}\rangle\subseteq{s}\,\big\}.
   $$
\end{pri}

We want to engraft onto $\mathcal{T}\hspace{-1pt}$ many grafts at once, so we need to find conditions which guarantee that different grafts do not conflict with each other (for example, nodes of one graft should not lie in the explant of another graft).

\begin{deff}\label{def.c.f.grafts}
   Let $\mathcal{T}\hspace{-1pt}$ be a tree.
   Then
   $\gamma$ is a \textbf{consistent} family of grafts for $\mathcal{T}\hspace{-1pt}$
   iff

   \begin{itemize}
   \item [\textup{(a)}]
      $\forall\mathcal{G}\,{\in}\,\gamma\ [\,\mathcal{G}$ is a graft for $\mathcal{T}\,]$;
   \item [\textup{(\hspace{0.3pt}b\hspace{-1pt})}]
      $\forall\mathcal{D}\,{\neq}\,\mathcal{E}\,{\in}\:\gamma
      \ [\,\impl\mathcal{D}\cap\impl\mathcal{E}=\varnothing\,];$
   \item [\textup{(c)}]
      $\forall\mathcal{D}\,{\neq}\,\mathcal{E}\,{\in}\:\gamma$
      \begin{itemize}
      \item[\ding{226}\ ]
            ${0}_\mathcal{D}\:{\parallel}_{\!\mathcal{T}}\:
            {0}_\mathcal{E}\ \ \mathsf{or}$
      \item[\ding{226}\ ]
            ${0}_\mathcal{D}\in(\hspace{-0.5pt}\maxel\mathcal{E}) {\downfootline}_{\mathcal{T}}\!\ \ \mathsf{or}$
      \item[\ding{226}\ ]
            ${0}_\mathcal{E}\in(\hspace{-0.5pt}\maxel\mathcal{D}){\hspace{-1pt}\downfootline}_{\mathcal{T}}.$
     \end{itemize}
  \end{itemize}
   The set
   $$\supp(\mathcal{T}\hspace{-1pt},\gamma)\coloneq
    \nodes\mathcal{T}\setminus\bigcup_{\mathcal{G}\in\gamma}\expl(\mathcal{T}\hspace{-1pt},\mathcal{G})$$
   is called the \textbf{support} of $\mathcal{T}\hspace{-1pt}$ for
   $\gamma.$
\end{deff}

\begin{lem}\label{lem.graft}
   Suppose that $\gamma$ is a consistent family of grafts for a tree $\mathcal{T}\hspace{-1pt}$ and $\mathcal{G}\in\gamma.$
   Then\textup{:}

   \begin{itemize}
   \item [\textup{(a)}]
      $\nodes\mathcal{G}\equiv
      \{\hspace{-1pt}{0}_\mathcal{G}\hspace{-1pt}\}\sqcup\maxel\mathcal{G}\sqcup\impl\mathcal{G};$
   \item [\textup{(\hspace{0.3pt}b\hspace{-1pt})}]
      $\{\hspace{-1pt}{0}_\mathcal{G}\hspace{-1pt}\}\hspace{-0.5pt}\cup\maxel\mathcal{G}\cup\minel\mathcal{T}
      \,\subseteq\,\supp(\mathcal{T}\hspace{-1pt},\gamma);$
   \item [\textup{(c)}]
      $\impl\mathcal{G}\cap\hspace{0.4pt}\supp(\mathcal{T}\hspace{-1pt},\gamma)
      =\varnothing;$
   \item [\textup{(\hspace{-0.4pt}d)}]
      $\forall{s}\,{\in}\supp(\mathcal{T}\hspace{-1pt},\gamma)
      \ \big[\,
      {s}>_{\mathcal{T}}{0}_{\mathcal{G}}
      \ \longleftrightarrow\
      {s}\in(\hspace{-0.5pt}\maxel\mathcal{G}){\hspace{-1pt}\downfootline}_{\mathcal{T}}
      \,\big];$
   \item [\textup{(e)}]
      $\forall{s}\,{\in}\supp(\mathcal{T}\hspace{-1pt},\gamma)\ \forall{e}\,{\in}\expl(\mathcal{T}\hspace{-1pt},\mathcal{G})
      \ [\,
      {s}\leqslant_{\mathcal{T}}{0}_{\mathcal{G}}
      \ \longleftrightarrow\
      {s}<_{\mathcal{T}}{e}
      \,];$
   \item [\textup{(f)}]
      $\forall\mathcal{D}\,{\neq}\,\mathcal{E}\,{\in}\:\gamma\:
      [\,{0}_\mathcal{D}\neq{0}_\mathcal{E}
      \enskip\mathsf{and}\enskip
      \maxel\mathcal{D}\cap\maxel\mathcal{E}=\varnothing
      \,].$
            \hfill$\Box$
   \end{itemize}
\end{lem}

Now we can give a definition of the hybrid operation:

\begin{deff}\label{def.hybr}
   Let $\gamma$ be a consistent family of grafts for a tree $\mathcal{T}\hspace{-1pt}.$ Then
   the \textbf{hybrid} of $\mathcal{T}\hspace{-1pt}$ and $\gamma$
    --- in symbols, $\hybr(\mathcal{T}\hspace{-1pt},\gamma)$ ---
   is a pair $( {H},<)$
   such that\textup{:}

   \begin{itemize}
   \item[\textup{(a)}]
      $\mathsurround=0pt\displaystyle {H}\,\coloneq\,\supp(\mathcal{T}\hspace{-1pt},\gamma)\,\cup\,\hspace{0.4pt}
      \bigcup_{\mathcal{G}\in\hspace{0.5pt}\gamma}\impl\mathcal{G}\vspace{-0.5ex}$

      (note that  all these sets are pairwise disjoint by (\hspace{0.3pt}b\hspace{-0.5pt}) of Definition~\ref{def.c.f.grafts} and (c) of Lemma~\ref{lem.graft});
   \item[\textup{(\hspace{0.3pt}b\hspace{-1pt})}]
      $\mathsurround=0pt{<}\;$ is a relation on ${H}$ defined by\textup{:}

      $\mathsurround=0pt{x}< {y}\quad{\colon}{\longleftrightarrow}$
      \begin{itemize}
      \item[\textup{(b1)}]
         $\mathsurround=0pt{x},{{y}}\in\supp(\mathcal{T}\hspace{-1pt},\gamma)\ \,\ \mathsf{and}\ \ {x}<_{\mathcal{T}}{{y}}$

         $\mathsurround=0pt\mathsf{or}$
      \item[\textup{(b2)}]
         $\mathsurround=0pt\exists\hspace{0.5pt}\mathcal{G}\hspace{1pt}{\in}\,\gamma\,$ such that
         \begin{itemize}
         \item[\ding{226}\ ]
            ${x},{{y}}\in\impl\mathcal{G}\ \ \ \mathsf{and}$
         \item[\ding{226}\ ]
            ${x}<_{\mathcal{G}} {{y}}\ \ $
         \end{itemize}
         $\mathsurround=0pt\mathsf{or}$
      \item[\textup{(b3)}]
         $\mathsurround=0pt\exists\hspace{0.5pt}\mathcal{G}\hspace{1pt}{\in}\,\gamma\,$ such that
         \begin{itemize}
         \item[\ding{226}\ ]
            ${x}\in\supp(\mathcal{T}\hspace{-1pt},\gamma)\ \ \ \mathsf{and}$
         \item[\ding{226}\ ]
            ${{y}}\in\impl\mathcal{G}\ \ \ \mathsf{and}$
         \item[\ding{226}\ ]
            ${x}\leqslant_{\mathcal{T}}\!{0}_\mathcal{G}$
         \end{itemize}
         $\mathsurround=0pt\mathsf{or}$
      \item[\textup{(b4)}]
         $\mathsurround=0pt\exists\hspace{0.5pt}\mathcal{G}\hspace{1pt}{\in}\,\gamma\,$ such that
         \begin{itemize}
         \item[\ding{226}\ ]
            ${x}\in\impl\mathcal{G}\ \ \ \mathsf{and}$
         \item[\ding{226}\ ]
            ${{y}}\in\supp(\mathcal{T}\hspace{-1pt},\gamma)\ \ \ \mathsf{and}$
         \item[\ding{226}\ ]
            ${{y}}\in(\hspace{-0.5pt}\maxel\mathcal{G}){\hspace{-1pt}\downfootline}_{\mathcal{T}}\!\ \ \ \mathsf{and}$
         \item[\ding{226}\ ]
            ${x}<_{\mathcal{G}}\roott_{\mathcal{T}}\!\big({{y}},\maxel\mathcal{G}\big)$
         \end{itemize}
         $\mathsurround=0pt\mathsf{or}$
      \item[\textup{(b5)}]
         $\exists\hspace{0.5pt}\mathcal{D}\hspace{1pt}{\neq}\,\mathcal{E}\,{\in}\,\gamma$ such that
         \begin{itemize}
         \item[\ding{226}\ ]
            ${x}\in\impl\mathcal{D}\ \ \ \mathsf{and}$
         \item[\ding{226}\ ]
            ${{y}}\in\impl\mathcal{E}\ \ \ \mathsf{and}$
         \item[\ding{226}\ ]
            ${0}_\mathcal{E}\in(\hspace{-0.5pt}\maxel\mathcal{D}){\hspace{-1pt}\downfootline}_{\mathcal{T}}\!\ \ \ \mathsf{and}$
         \item[\ding{226}\ ]
            ${x}<_{\mathcal{D}}\roott_{\mathcal{T}}({0}_\mathcal{E}
            ,\maxel\mathcal{D}).$
         \end{itemize}
      \end{itemize}
   \end{itemize}
\end{deff}

We could give a shorter (but less suitable for our aims) definition for the hybrid operation in the following equivalent way:

\begin{rem}\label{rem.hybg}
   Clause \textup{(\hspace{0.3pt}b\hspace{-0.5pt})} of Definition~\ref{def.hybr} is equivalent to the assertion that
   ${<}$ is the transitive closure of relation
   $$
   \big(\!<_{\mathcal{T}}\:\cup\ \bigcup_{\mathcal{G}\in\gamma}<_\mathcal{G}\!\big)\hspace{1pt}
   \,\cap\, ({H}\times {H}).
   $$
\end{rem}
\begin{proof} Let ${\vartriangleleft}\:\coloneq\,\big({<}_{\mathcal{T}}\ {\cup}\:
\bigcup_{\mathcal{G}\in\gamma}{<_\mathcal{G}}\big)\hspace{0.5pt}\:\cap\:({H}\times {H}).$
We have ${\vartriangleleft}\hspace{1pt}\subseteq{<}$ by (a)--(\hspace{0.3pt}b\hspace{-0.5pt}) of Lemma~\ref{lem.hybr} and $<$ is transitive by Proposition~\ref{hybr.is.tree} (we do not use Remark~\ref{rem.hybg} in the proofs of Lemma~\ref{lem.hybr} and Proposition~\ref{hybr.is.tree}).

It remains to show that if ${\vartriangleleft}\subseteq{\lessdot}$ and $\lessdot$ is a transitive relation on ${H},$  then ${<}\subseteq{\lessdot}\,.$
Suppose $(\hspace{-0.5pt}{x},{y}\hspace{-0.5pt})\hspace{0.5pt}\in{<};$  this means that one of conditions (b1)--(b5) of Definition~\ref{def.hybr} holds. For example, if (b3) holds, then ${x}\in\supp(\mathcal{T}\hspace{-1pt},\gamma),$ ${y}\in\impl\mathcal{G},$
${x}\leqslant_{\mathcal{T}}\!{0}_\mathcal{G},$ and
${0}_\mathcal{G}<_\mathcal{G}{y},$
so ${x}\trianglelefteq{0}_\mathcal{G}\vartriangleleft {y}.$
Then ${x}\leqdot{0}_\mathcal{G}\lessdot {y},$
whence $(\hspace{-0.5pt}{x},{y}\hspace{-0.5pt})\in{\lessdot}$ by transitivity. The other cases are similar.
\end{proof}

\begin{lem}\label{lem.hybr}
   Suppose that $\gamma$ is a consistent family of grafts for a tree $\mathcal{T}\hspace{-1pt},$
   $\mathcal{H}=\hybr(\mathcal{T}\hspace{-1pt},\gamma),$ and $\mathcal{G}\in\gamma.$
   Then\textup{:}

   \begin{itemize}
   \item [\textup{(a)}]
      $\mathsurround=0pt\hspace{0.5pt}\nodes\mathcal{G}\subseteq\nodes\mathcal{H}\ \ \mathsf{and}$

      $\mathsurround=0pt\forall{x},{y}\,{\in}\nodes\mathcal{G}
      \ [\,
      {x}<_\mathcal{H} {y}\ \leftrightarrow\ {x}<_{\mathcal{G}} {y}
      \,];$
       \item [\textup{(\hspace{0.3pt}b\hspace{-1pt})}]
      $\mathsurround=0pt\hspace{0.5pt}\supp(\mathcal{T}\hspace{-1pt},\gamma)=\nodes\mathcal{H}\cap\nodes\mathcal{T}\ \ \mathsf{and}$

      $\mathsurround=0pt\forall{x},{y}\,{\in}\supp(\mathcal{T}\hspace{-1pt},\gamma)
      \ [\,
      {x}<_{\mathcal{H}} {y}\ \leftrightarrow\ {x}<_{\mathcal{T}} {y}
      \,];$
   \item [\textup{(c)}]
      $\mathsurround=0pt\forall{h}\,{\in}\,\nodes\mathcal{H}\ \forall{i}\,{\in}\,\impl\mathcal{G}
      \ [\,
      {h}\geqslant_\mathcal{H}{i}\ \rightarrow\ {h}>_{\mathcal{H}\!}{0}_{\mathcal{G}}
      \,];$
   \item [\textup{(\hspace{-0.4pt}d)}]
      $\mathsurround=0pt\forall{h}\,{\in}\,\nodes\mathcal{H}\ \forall{i}\,{\in}\,\impl\mathcal{G}
      \ [\,
      {h}\leqslant_{\mathcal{H}\!}{0}_{\mathcal{G}}
      \ \rightarrow\
      {h}<_\mathcal{H}{i}
      \,];$
   \item [\textup{(e)}]
      $\mathsurround=0pt\forall{h}\,{\in}\hspace{1pt}\nodes\mathcal{H}\hspace{1pt}{\setminus}\impl\mathcal{G}\ \forall{i}\,{\in}\,\impl\mathcal{G}
      \ [\,
      {h}\leqslant_{\mathcal{H}\!}{0}_{\mathcal{G}}
      \ \leftrightarrow\
      {h}<_\mathcal{H}{i}
      \,];$
   \item [\textup{(f)}]
      $\mathsurround=0pt\forall{h}\,{\in}\,\nodes\mathcal{H}\hspace{1pt}{\setminus}\impl\mathcal{G}
      \ \big[\,
      {h}>_{\mathcal{H}\!}{0}_{\mathcal{G}}
      \ \leftrightarrow\
      {h}\in(\hspace{-0.5pt}\maxel\mathcal{G}){\hspace{-1pt}\downfootline}_\mathcal{H}
      \,\big];$
   \item [\textup{(g)}]
      $\mathsurround=0pt\forall{g}\,{\in}\nodes\mathcal{G}
      \ \big[\,
      {g}{\upspoon}_{\!\mathcal{H}}\equiv {g}{\upspoon}_{\!\mathcal{G}}
      \sqcup (\hspace{-1pt}{0}_\mathcal{G}\hspace{-1.4pt})\hspace{-1pt}{\upspoon}_{\!\mathcal{H}}
      \,\big].$
   \end{itemize}
\end{lem}

\begin{proof}
(a)--(e) are straightforward; (f) follows from (\hspace{0.3pt}b\hspace{-0.5pt}) of Lemma~\ref{lem.hybr}, (\hspace{-0.4pt}d) and (f) of Lemma~\ref{lem.graft}, and (e) of Definiton~\ref{def.graft};
(g) can be proved by using (a)--(f) of Lemma~\ref{lem.hybr}, (\hspace{-0.4pt}d) and (f) of Lemma~\ref{lem.graft}, and (e) of Definiton~\ref{def.graft}.
\end{proof}

First we show that a result of the hybrid operation is always a tree:

\begin{prop}\label{hybr.is.tree}
   Suppose that $\gamma$ is a consistent family of grafts for a tree $\mathcal{T}\hspace{-1pt}.$
   Then
    $\hybr(\mathcal{T}\hspace{-1pt},\gamma)$ is a tree.
\end{prop}

\begin{proof}
Let $\mathcal{H}\coloneq\hybr(\mathcal{T}\hspace{-1pt},\gamma).$
The irreflexivity of $<_\mathcal{H}$ is trivial, let us prove that ${x}<_\mathcal{H}{y}<_\mathcal{H}{z}$ implies ${x}<_\mathcal{H}{z}.$ We consider several cases\textup{:}

\begin{itemize}
\item[\textup{(i)\ }]
   $\mathsurround=0pt{z}\in\supp(\mathcal{T}\hspace{-1pt},\gamma).$
   \begin{itemize}
   \item[\textup{(i.1)\ }]
      $\mathsurround=0pt{y}\in\supp(\mathcal{T}\hspace{-1pt},\gamma).$

         The case ${x}\in\supp(\mathcal{T}\hspace{-1pt},\gamma)$  is trivial.
         If there is $\mathcal{D}\in\gamma$ such that ${x}\in\impl\mathcal{D},$
         then ${x}<_\mathcal{D}\roott_{\mathcal{T}}({y},\maxel\mathcal{D}).$
         Since ${y}<_{\mathcal{T}}{z},$ we have
         $\roott_{\mathcal{T}}({z},\maxel\mathcal{D})=\roott_{\mathcal{T}}({y},\maxel\mathcal{D}),$
         so
         ${x}<_\mathcal{H}{z}.$
   \item[\textup{(i.2)\ }]
      $\mathsurround=0pt\exists\hspace{0.5pt}\mathcal{E}\hspace{1pt}{\in}\,\gamma\ [\,{y}\in\impl\mathcal{E}\,].$

      The case ${x}\in\impl\mathcal{E}$ is trivial.
      If ${x}\notin\impl\mathcal{E},$ then ${x}\leqslant_{\mathcal{H}\!}{0}_{\mathcal{E}}$ by (e) of Lemma~\ref{lem.hybr} and ${0}_{\mathcal{E}}<_\mathcal{H}{z}$ by (c) of Lemma~\ref{lem.hybr}. Therefore ${x}\leqslant_{\mathcal{H}\!}{0}_{\mathcal{E}}<_\mathcal{H}{z}$
      and we may use (i.1), since ${0}_\mathcal{E}\in\supp(\mathcal{T}\hspace{-1pt},\gamma).$
   \end{itemize}
\item[\textup{(ii)\ }]
   $\mathsurround=0pt\exists\hspace{0.5pt}\mathcal{G}\hspace{1pt}{\in}\,\gamma\ [\,{z}\in\impl\mathcal{G}\,].$
   \begin{itemize}
   \item[\textup{(ii.1)\ }]
      $\mathsurround=0pt{y}\in\impl\mathcal{G}.$

         The case ${x}\in\impl\mathcal{G}$ is trivial.
         If ${x}\notin\impl\mathcal{G},$ then using (e) of Lemma~\ref{lem.hybr} twice, we get
         ${x}<_\mathcal{H}{z}.$
   \item[\textup{(ii.2)\ }]
      $\mathsurround=0pt{y}\notin\impl\mathcal{G}.$
      \\
      By (e) of Lemma~\ref{lem.hybr}, ${x}<_\mathcal{H}{y}\leqslant_{\mathcal{H}\!}{0}_\mathcal{G}.$
      Since ${0}_\mathcal{G}\in\supp(\mathcal{T}\hspace{-1pt},\gamma),$
      we have ${x}<_{\mathcal{H}\!}{0}_\mathcal{G}$ by (i),
      so ${x}<_\mathcal{H}{z}$ by (\hspace{-0.4pt}d) of Lemma~\ref{lem.hybr}.
   \end{itemize}
\end{itemize}
\medskip

Now we prove that for each ${z}\in\nodes\mathcal{H},$ the set ${z}{\upspoon}_{\!\mathcal{H}}$ is a chain in $\mathcal{H}.$ We must show that ${x},{y}\in{z}{\upspoon}_{\!\mathcal{H}}$ implies  ${x}\leqslant_{\mathcal{H}}{y}\ \ \mathsf{or}\ \ {x}>_{\mathcal{H}}{y}.$ Again, we consider several cases\textup{:}

\begin{itemize}
\item[\textup{(i)\ }]
   $\mathsurround=0pt{z}\in\supp(\mathcal{T}\hspace{-1pt},\gamma).$
   \begin{itemize}
   \item[\textup{(i.1)\ }]
      $\mathsurround=0pt{y}\in\supp(\mathcal{T}\hspace{-1pt},\gamma).$
      \begin{itemize}
      \item[\textup{(i.1.1)\ }]
         $\mathsurround=0pt{x}\in\supp(\mathcal{T}\hspace{-1pt},\gamma).$
         \\
         This case is trivial.
      \item[\textup{(i.1.2)\ }]
         $\mathsurround=0pt\exists\hspace{0.5pt}\mathcal{E}\hspace{1pt}{\in}\,\gamma\ [\,{x}\in\impl\mathcal{D}\,].$
         \\
         By (c) of Lemma~\ref{lem.hybr}, ${0}_\mathcal{D}\in{z}{\upspoon}_{\!\mathcal{H}},$
         so  ${0}_\mathcal{D}\in\supp(\mathcal{T}\hspace{-1pt},\gamma)$ implies
         ${y}\leqslant_{\mathcal{H}}\!{0}_\mathcal{D}\ \ \mathsf{or}\ \ {y}>_{\mathcal{H}}\!{0}_\mathcal{D}$ by  (i.1.1).
         If ${y}\leqslant_\mathcal{H}\!{0}_\mathcal{D},$ then
         ${y}<_\mathcal{H}{x}$ by (\hspace{-0.4pt}d) of Lemma~\ref{lem.hybr}.
         If ${y}>_\mathcal{H}\!{0}_\mathcal{D},$ then
         ${y}\in(\hspace{-0.5pt}\maxel\mathcal{D}){\hspace{-1pt}\downfootline}_{\mathcal{T}}\!$ by (\hspace{-0.4pt}d) of Lemma~\ref{lem.graft}.
         Then  $\roott_{\mathcal{T}}({y},\maxel\mathcal{D})=\roott_{\mathcal{T}}({z},\maxel\mathcal{D}),$
         so ${x}<_\mathcal{D}\roott_{\mathcal{T}}({y},\maxel\mathcal{D}),$
         whence ${x}<_\mathcal{H}{y}.$
      \end{itemize}
   \item[\textup{(i.2)\ }]
      $\mathsurround=0pt\exists\hspace{0.5pt}\mathcal{E}\hspace{1pt}{\in}\,\gamma\ [\,{y}\in\impl\mathcal{E}\,].$
      \begin{itemize}
      \item[\textup{(i.2.1)\ }]
         $\mathsurround=0pt{x}\in\supp(\mathcal{T}\hspace{-1pt},\gamma).$
         \\
         This case is the same as (i.1.2).
      \item[\textup{(i.2.2)\ }]
         $\mathsurround=0pt\exists\hspace{0.5pt}\mathcal{E}\hspace{1pt}{\in}\,\gamma\ [\,{x}\in\impl\mathcal{D}\,].$
         \\
         The case $\mathcal{D}=\mathcal{E}$ is trivial.
         If $\mathcal{D}\neq\mathcal{E},$ then by (c) of Lemma~\ref{lem.hybr},
         ${0}_\mathcal{D},{0}_\mathcal{E}\in{z}{\upspoon}_{\!\mathcal{H}},$
         so ${0}_\mathcal{D}\leqslant_{\mathcal{H}}{0}_\mathcal{E}\ \ \mathsf{or}\ \
         {0}_\mathcal{D}>_{\mathcal{H}}{0}_\mathcal{E}$
         by (i.1.1). By (f) of Lemma~\ref{lem.graft}, ${0}_\mathcal{D}\neq{0}_\mathcal{E},$
         so by (\hspace{-0.4pt}d) of Lemma~\ref{lem.graft} we may assume without loss of generality that
         ${0}_\mathcal{E}\in(\hspace{-0.5pt}\maxel\mathcal{D}){\hspace{-1pt}\downfootline}_{\mathcal{T}}.$
         Since ${0}_\mathcal{E}<_{\mathcal{T}}{z},$ we have
         $\roott_{\mathcal{T}}({0}_\mathcal{E},\maxel\mathcal{D})
         =\roott_{\mathcal{T}}({z},\maxel\mathcal{D}),$
         so ${x}<_{\mathcal{H}\!}{0}_\mathcal{E},$
         whence ${x}<_\mathcal{H}{y}$ by transitivity.
      \end{itemize}
   \end{itemize}
\item[\textup{(ii)\ }]
   $\mathsurround=0pt\exists\hspace{0.5pt}\mathcal{G}\hspace{1pt}{\in}\,\gamma\ [\,{z}\in\impl\mathcal{G}\,].$
   \begin{itemize}
   \item[\textup{(ii.1)\ }]
      $\mathsurround=0pt{y}\in\impl\mathcal{G}\,\ \mathsf{or}\ {x}\in\impl\mathcal{G}.$

      This case is similar to case (ii.1) from the proof of transitivity.
   \item[\textup{(ii.2)\ }]
      $\mathsurround=0pt{x},{y}\notin\impl\mathcal{G}.$
      \\
      By (e) of Lemma~\ref{lem.hybr}, ${x},{y}\in(\hspace{-1pt}{0}_\mathcal{G}\hspace{-1.4pt}){\upfilledspoon}_{\!\mathcal{H}}.$
      Then either $\{{x},{y}\}\cap\{\hspace{-1pt}{0}_\mathcal{G}\hspace{-1pt}\}\neq\varnothing$
      or ${x},{y}\in(\hspace{-1pt}{0}_\mathcal{G}\hspace{-1.4pt}){\upspoon}_{\!\mathcal{H}}$
      and the proof from (i) for ${z}\coloneq{0}_\mathcal{G}$ works.
   \end{itemize}
\end{itemize}
\medskip

It remains to prove that for each ${z}\in\nodes\mathcal{H}$ and each nonempty
${A}\subseteq {z}{\upspoon}_{\!\mathcal{H}},$ there is a $<_\mathcal{H}\mathsurround=0pt $-minimal node in ${A}.$ We consider several cases\textup{:}

\begin{itemize}
\item[\textup{(i)\ }]
   $\mathsurround=0pt{z}\in\supp(\mathcal{T}\hspace{-1pt},\gamma).$

   Consider a nonempty set
   $${B}\,\coloneq\,\big({A}\cap\hspace{0.4pt}\supp(\mathcal{T}\hspace{-1pt},\gamma)\big)
   \:\cup\,\{\,{0}_\mathcal{G}\,:\,\mathcal{G}\in\gamma
   \enskip\mathsf{and}\enskip{A}\cap\impl\mathcal{G}\neq\varnothing\,\}.$$

   We have $\!{B}\subseteq\supp(\mathcal{T}\hspace{-1pt},\gamma),$
   so it follows by (c) of Lemma~\ref{lem.hybr} that $\!{B}\subseteq {z}{\upfilledspoon}_{\!\mathcal{T}}.$
   Then there is a $<_{\mathcal{T}}\mathsurround=0pt$-minimal node ${m}$ in $\!{B}.$ Note that ${m}\in\supp(\mathcal{T}\hspace{-1pt},\gamma).$
   \begin{itemize}
   \item[\textup{(i.1)\ }]
      $\mathsurround=0pt{m}\in{A}.$

      Let us show that ${m}$ is a $<_\mathcal{H}\mathsurround=0pt$-minimal node of ${A}.$ Suppose ${x}\in{A}$ and ${x}\leqslant_\mathcal{H}{m}.$
      \begin{itemize}
      \item[\textup{(i.1.1)\ }]
      $\mathsurround=0pt{x}\in\supp(\mathcal{T}\hspace{-1pt},\gamma).$

      In this case ${x}\in{B}$ and ${x}\leqslant_{\mathcal{T}}{m},$ so ${x}={m}.$
      \item[\textup{(i.1.2)\ }]
      $\mathsurround=0pt\exists\hspace{0.5pt}\mathcal{E}\hspace{1pt}{\in}\,\gamma\ [\,{x}\in\impl\mathcal{D}\,].$

      By (c) of Lemma~\ref{lem.hybr}, ${0}_\mathcal{D}<_{\mathcal{T}}{m}.$
      But ${0}_\mathcal{D}\in{B},$ since ${x}\in{A}\cap\impl\mathcal{D}.$
      This contradicts the $<_{\mathcal{T}}\mathsurround=0pt$-minimality of ${m}$ in $\!{B} .$
      \end{itemize}
   \item[\textup{(i.2)\ }]
      $\mathsurround=0pt{m}\notin {A}.$

      In this case ${m}={0}_\mathcal{G}$ for some $\mathcal{G}\in\gamma$ such that ${A}\cap\impl\mathcal{G}\neq\varnothing.$
      Since ${A}$ is a chain in $\mathcal{H},$ it follows that
      ${A}\cap\impl\mathcal{G}$ is a chain in $\mathcal{G}.$
      Then it is not hard to prove that there is a $<_\mathcal{G}\mathsurround=0pt$-minimal node ${l}$ in ${A}\cap\impl\mathcal{G}.$
      Let us show that ${l}$ is a $<_\mathcal{H}\mathsurround=0pt$-minimal node of~${A}.$
      Suppose ${x}\in{A}$ and ${x}\leqslant_\mathcal{H} {l}.$
      \begin{itemize}
      \item[\textup{(i.2.1)\ }]
         $\mathsurround=0pt{x}\in\supp(\mathcal{T}\hspace{-1pt},\gamma).$
         \\
         Then ${x}<_\mathcal{H} {l}$ (since ${l}\in\impl\mathcal{G}\not\ni {x}$), so by (e) of Lemma~\ref{lem.hybr}, ${x}\leqslant_{\mathcal{T}}{0}_\mathcal{G}={m}.$
         Since ${x}\in{A}\not\ni {m},$ we have ${x}<_{\mathcal{T}}{m}$ and ${x}\in{B}.$
         This contradicts the $<_{\mathcal{T}}\mathsurround=0pt$-minimality of ${m}$ in~${B}.$
      \item[\textup{(i.2.2)\ }]
         $\mathsurround=0pt\exists\hspace{0.5pt}\mathcal{E}\hspace{1pt}{\in}\,\gamma\ [\,{x}\in\impl\mathcal{D}\,].$

         We have ${0}_\mathcal{D},{0}_\mathcal{G}\in{B}$ and
         ${0}_\mathcal{D}\leqslant_{\mathcal{H}\!}{0}_\mathcal{G}$ by (c) and (e) of Lemma~\ref{lem.hybr}. Then ${0}_\mathcal{D}={0}_\mathcal{G}$ by the $<_{\mathcal{T}}\mathsurround=0pt $-minimality of ${0}_\mathcal{G}={m}$ in $\!{B},$ so
         $\mathcal{D}=\mathcal{G}$ by (f) of Lemma~\ref{lem.graft}. This implies
         ${x}\in{A}\cap\impl\mathcal{G}$ and ${x}\leqslant_\mathcal{G}{l},$ so ${x}={l}$ by the
         $<_\mathcal{G}\mathsurround=0pt$-minimality of ${l}$ in ${A}\cap\impl\mathcal{G}.$
      \end{itemize}
   \end{itemize}
\item[\textup{(ii)\ }]
   $\mathsurround=0pt\exists\hspace{0.5pt}\mathcal{G}\hspace{1pt}{\in}\,\gamma\ [\,{z}\in\impl\mathcal{G}\,].$

   By (g) of Lemma~\ref{lem.hybr},\vspace{-1ex}
   $${A}\subseteq {z}{\upspoon}_{\!\mathcal{G}}
      \cup (\hspace{-1pt}{0}_\mathcal{G}\hspace{-1.4pt})\hspace{-1pt}{\upspoon}_{\!\mathcal{H}}.$$
      If
      ${A}\cap(\hspace{-1pt}{0}_\mathcal{G}\hspace{-1.4pt})\hspace{-1pt}{\upspoon}_{\!\mathcal{H}}\neq\varnothing,$
      then a $<_\mathcal{H}\mathsurround=0pt $-minimal node of
      ${A}\cap(\hspace{-1pt}{0}_\mathcal{G}\hspace{-1.4pt})\hspace{-1pt}{\upspoon}_{\!\mathcal{H}},$ which exists by (i), is a $<_\mathcal{H}\mathsurround=0pt $-minimal node of~${A}.$
      Otherwise,
      ${A}\subseteq {z}{\upspoon}_{\!\mathcal{G}},$
      and then a $<_\mathcal{G}\mathsurround=0pt $-minimal node of ${A}$ is a $<_\mathcal{H}\mathsurround=0pt $-minimal node of ${A}.$
   \end{itemize}
\end{proof}

Now we establish several properties of the hybrid operation:

\begin{prop}\label{prop.hybr}
   Suppose that $\gamma$ is a consistent family of grafts for a tree $\mathcal{T}\hspace{-1pt}$ and  $\mathcal{H}=\hybr(\mathcal{T}\hspace{-1pt},\gamma).$

   Then\textup{:}
   \begin{itemize}
   \item [\textup{(a)}]
      For each ${x}\in\nodes\mathcal{H},$
      \\
      $
      \sons_{\hspace{0.5pt}\mathcal{H}\hspace{-0.5pt}}(\hspace{-0.7pt}{x}\hspace{-1pt})=
      \begin{cases}
         \,\sons_{\hspace{0.5pt}\mathcal{G}\hspace{-0.5pt}}(\hspace{-0.7pt}{x}\hspace{-1pt}),
         &\text
         {if \ ${x}\in\{\hspace{-1pt}{0}_\mathcal{G}\hspace{-1pt}\}\cup\hspace{0.4pt}\impl\mathcal{G}\ $ for some
         $\ \mathcal{G}\in\gamma;$
          }
      \\ \,\sons_{\hspace{0.6pt}\mathcal{T}}(\hspace{-0.7pt}{x}\hspace{-1pt}),
         &\text
         {otherwise \textup{\big(}i.e., when ${x}\in\supp(\mathcal{T}\hspace{-1pt},\gamma)
         \setminus
         \{{0}_\mathcal{G}:\mathcal{G}\in\gamma\}\textup{\big)}.\vphantom{\tilde{\big\langle\rangle}}$}
      \end{cases}
      $
   \item [\textup{(\hspace{0.3pt}b\hspace{-1pt})}]
      If ${x},{y}\in\nodes\mathcal{H}$ and ${x}\:{\parallel}_{\!\mathcal{H}}\:{y},$

      then there are ${x}'\in{x}\hspace{-0.5pt}{\upfilledspoon}_{\!\mathcal{H}}\!$
      and ${y}'\in{y}{\upfilledspoon}_{\!\mathcal{H}}$ such that
      \begin{itemize}
      \item[\textup{(b1)}\ ]
         $\mathsurround=0pt\big[\,{x}'\!,{y}'\,{\in}\,\supp(\mathcal{T}\hspace{-1pt},\gamma)\ \mathsf{and}\ {x}'\,{\parallel}_{\!\mathcal{T}}\:{y}'\,\big]\quad
         \mathsf{or}$
      \item[\textup{(b2)}\ ]
         $\mathsurround=0pt\exists\hspace{0.5pt}\mathcal{G}\hspace{1pt}{\in}\,\gamma\,\big[\,
         {x}'\!,{y}'\,{\in}\hspace{1pt}\nodes\mathcal{G}\ \mathsf{and}\ {x}'\,{\parallel}_{\!\mathcal{G}}\:{y}'\,\big].$
      \end{itemize}
   \item [\textup{(c)}]
      If $\mathcal{T}\hspace{-1pt}$ has the least node, \\
      then $\mathcal{H}$ has the least node, ${0}_\mathcal{H}={0}_{\mathcal{T}}\!,$ and
      ${0}_\mathcal{H}\in\supp(\mathcal{T}\hspace{-1pt},\gamma).$
   \item [\textup{(\hspace{-0.4pt}d)}]
      If $\maxel\mathcal{T}=\varnothing,$\\
      then
      $\maxel\mathcal{H}=\varnothing.$
   \item [\textup{(e)}]
      If
      $\mathcal{T}\hspace{-1pt}$ is $\kappa\mathsurround=0pt $-branching and $\forall\mathcal{G}\,{\in}\,\gamma
      \;[\,\mathcal{G}$ is $\kappa\mathsurround=0pt $-branching\,$],$
      \\
      then $\mathcal{H}$ is $\kappa\mathsurround=0pt $-branching.
   \item [\textup{(f)}]
      If
      $\height\hspace{-0.4pt}\mathcal{T}\hspace{-1pt}\leqslant\omega$ and $\forall\mathcal{G}\,{\in}\,\gamma
      \:[\,\height\hspace{-0.4pt}\mathcal{G}\leqslant\omega\,],$
      \\
      then $\height\hspace{-0.4pt}\mathcal{H}\leqslant\omega.$
   \end{itemize}
\end{prop}

\begin{proof}
(a) Suppose ${x}\in\nodes\mathcal{H}.$ We consider two cases:
\smallskip

\textit{\hspace{-1pt}Case 1.} $\exists\hspace{0.5pt}\mathcal{G}\hspace{1pt}{\in}\,\gamma\ \big[\,{x}\in\{\hspace{-1pt}{0}_\mathcal{G}\hspace{-1pt}\}\cup\hspace{0.4pt}\impl\mathcal{G}\,\big].$

   First we prove $\sons_{\hspace{0.5pt}\mathcal{H}\hspace{-0.5pt}}(\hspace{-0.7pt}{x}\hspace{-1pt})\supseteq
   \sons_{\hspace{0.5pt}\mathcal{G}\hspace{-0.5pt}}(\hspace{-0.7pt}{x}\hspace{-1pt}).$
   If not, then there is ${s}\in\hspace{0.4pt}\sons_{\hspace{0.5pt}\mathcal{G}\hspace{-0.5pt}}(\hspace{-0.7pt}{x}\hspace{-1pt})\setminus
   \sons_{\hspace{0.5pt}\mathcal{H}\hspace{-0.5pt}}(\hspace{-0.7pt}{x}\hspace{-1pt}).$
   Then ${x}<_\mathcal{H}{s}$ by (a) of Lemma~\ref{lem.hybr},
   so ${s}\notin\sons_{\hspace{0.5pt}\mathcal{H}\hspace{-0.5pt}}(\hspace{-0.7pt}{x}\hspace{-1pt})$ implies there is ${v}\in({x},{s})_\mathcal{H}.$
   We have  ${v}\notin(\hspace{-1pt}{0}_\mathcal{G}\hspace{-1.4pt})\hspace{-1pt}{\upspoon}_{\!\mathcal{H}},$
   so ${v}\in{s}{\upspoon}_{\!\mathcal{G}}\subseteq\nodes\mathcal{G}$  by (g) of Lemma~\ref{lem.hybr},
   whence ${v}\in({x},{s})_\mathcal{G}.$  This contradicts ${s}\in\hspace{0.4pt}\sons_{\hspace{0.5pt}\mathcal{G}\hspace{-0.5pt}}(\hspace{-0.7pt}{x}\hspace{-1pt}).$

   Now we prove $\sons_{\hspace{0.5pt}\mathcal{H}\hspace{-0.5pt}}(\hspace{-0.7pt}{x}\hspace{-1pt})
   \subseteq\hspace{0.4pt}\sons_{\hspace{0.5pt}\mathcal{G}\hspace{-0.5pt}}(\hspace{-0.7pt}{x}\hspace{-1pt}).$
   If not, then there is ${s}\in\hspace{0.4pt}\sons_{\hspace{0.5pt}\mathcal{H}\hspace{-0.5pt}}(\hspace{-0.7pt}{x}\hspace{-1pt})
   \setminus\sons_{\hspace{0.5pt}\mathcal{G}\hspace{-0.5pt}}(\hspace{-0.7pt}{x}\hspace{-1pt}).$ We consider several subcases\textup{:}
\begin{itemize}
\item[\textup{(i)\ }]
   $\mathsurround=0pt{x}\in\impl\mathcal{G}.$
   \begin{itemize}
   \item[\textup{(i.1)\ }]
      $\mathsurround=0pt{s}\in\nodes\mathcal{G}.$

      Then ${x}<_\mathcal{G}{s},$ so $({x},{s})_\mathcal{G}\neq\varnothing,$
      whence $({x},{s})_\mathcal{H}\neq\varnothing$ by (a) of Lemma~\ref{lem.hybr}.
      This contradicts ${s}\in\hspace{0.4pt}\sons_{\hspace{0.5pt}\mathcal{H}\hspace{-0.5pt}}(\hspace{-0.7pt}{x}\hspace{-1pt}).$
   \item[\textup{(i.2)\ }]
      $\mathsurround=0pt{s}\notin\nodes\mathcal{G}.$

      Then ${x}<_{\mathcal{H}}{s}$ implies ${x}<_{\mathcal{H}}{r}<_\mathcal{H}{s}$ for some ${r}\in\maxel\mathcal{G}.$
      This contradicts  ${s}\in\hspace{0.4pt}\sons_{\hspace{0.5pt}\mathcal{H}\hspace{-0.5pt}}(\hspace{-0.7pt}{x}\hspace{-1pt}).$
   \end{itemize}
\item[\textup{(ii)\ }]
   $\mathsurround=0pt{x}={0}_\mathcal{G}.$
   \begin{itemize}
   \item[\textup{(ii.1)\ }]
      $\mathsurround=0pt{s}\in\impl\mathcal{G}.$

      This case is similar to (i.1).
   \item[\textup{(ii.2)\ }]
      $\mathsurround=0pt{s}\notin\impl\mathcal{G}.$

     Then ${s}\in\hspace{0.4pt}\sons_{\hspace{0.5pt}\mathcal{H}\hspace{-0.5pt}}(\hspace{-1pt}{0}_\mathcal{G}\hspace{-1.4pt})$
     with (f) of Lemma~\ref{lem.hybr} imply ${s}\in\maxel\mathcal{G},$
     so $({0}_\mathcal{G},{s})_\mathcal{H}=\varnothing$
     implies $({0}_\mathcal{G},{s})_\mathcal{G}=\varnothing.$
     This contradicts ${s}\notin\sons_{\hspace{0.5pt}\mathcal{G}\hspace{-0.5pt}}(\hspace{-1pt}{0}_\mathcal{G}\hspace{-1.4pt}).$
   \end{itemize}
\end{itemize}
\smallskip

\textit{\hspace{-1pt}Case 2.}  ${x}\in\supp(\mathcal{T}\hspace{-1pt},\gamma)\setminus
   \{{0}_\mathcal{G}:\mathcal{G}\in\gamma\}.$

   First we prove $\sons_{\hspace{0.5pt}\mathcal{H}\hspace{-0.5pt}}(\hspace{-0.7pt}{x}\hspace{-1pt})\subseteq
   \hspace{0.4pt}\sons_{\hspace{0.6pt}\mathcal{T}}(\hspace{-0.7pt}{x}\hspace{-1pt}).$
   If not, then there is ${s}\in\hspace{0.4pt}\sons_{\hspace{0.5pt}\mathcal{H}\hspace{-0.5pt}}(\hspace{-0.7pt}{x}\hspace{-1pt})
   \setminus\sons_{\hspace{0.6pt}\mathcal{T}}(\hspace{-0.7pt}{x}\hspace{-1pt}).$ We consider two subcases\textup{:}
   \begin{itemize}
   \item[\textup{(i)\ }]
      $\mathsurround=0pt
      {s}\notin\supp(\mathcal{T}\hspace{-1pt},\gamma).$

      Then there is $\mathcal{E}\in\gamma$ such that ${s}\in\impl\mathcal{E}.$
      Then ${x}\leqslant_{\mathcal{H}\!}{0}_\mathcal{E}<_\mathcal{H}{s}$ by (e) of Lemma~\ref{lem.hybr},
      so ${x}<_{\mathcal{H}\!}{0}_\mathcal{E}<_\mathcal{H}{s}$ by Case~2.
      This contradicts ${s}\in\hspace{0.4pt}\sons_{\hspace{0.5pt}\mathcal{H}\hspace{-0.5pt}}(\hspace{-0.7pt}{x}\hspace{-1pt}).$
   \item[\textup{(ii)\ }]
      $\mathsurround=0pt{s}\in\supp(\mathcal{T}\hspace{-1pt},\gamma).$

      Then ${x}<{_{\mathcal{T}}}{s},$ so
      ${s}\notin\sons_{\hspace{0.6pt}\mathcal{T}}(\hspace{-0.7pt}{x}\hspace{-1pt})$ implies there is ${v}\in({x},{s})_{\mathcal{T}}.$
      Since $({x},{s})_{\mathcal{H}} =\varnothing,$ we have
      ${v}\notin\supp(\mathcal{T}\hspace{-1pt},\gamma),$ so there is $\mathcal{E}\in\gamma$
      such that ${v}\in\expl(\mathcal{T}\hspace{-1pt},\mathcal{E}).$
      Then ${x}\leqslant_{\mathcal{T}}{0}_\mathcal{E}<_{\mathcal{T}}{v}<_{\mathcal{T}}{s}$
      by (e) of Lemma~\ref{lem.graft},
      so ${x}\leqslant_{\mathcal{H}\!}{0}_\mathcal{E}<_\mathcal{H}{s},$
      whence    ${x}<_{\mathcal{H}\!}{0}_\mathcal{E}<_\mathcal{H}{s}$ by Case~2.
      This contradicts ${s}\in\hspace{0.4pt}\sons_{\hspace{0.5pt}\mathcal{H}\hspace{-0.5pt}}(\hspace{-0.7pt}{x}\hspace{-1pt}).$
   \end{itemize}

   Now we prove $\sons_{\hspace{0.5pt}\mathcal{H}\hspace{-0.5pt}}(\hspace{-0.7pt}{x}\hspace{-1pt})
   \supseteq\sons_{\hspace{0.6pt}\mathcal{T}}(\hspace{-0.7pt}{x}\hspace{-1pt}).$
   If not, then there is ${s}\in\hspace{0.4pt}\sons_{\hspace{0.6pt}\mathcal{T}}(\hspace{-0.7pt}{x}\hspace{-1pt})\setminus
   \sons_{\hspace{0.5pt}\mathcal{H}\hspace{-0.5pt}}(\hspace{-0.7pt}{x}\hspace{-1pt}).$ Again, there are two subcases\textup{:}
   \begin{itemize}
   \item[\textup{(i)\ }]
      $\mathsurround=0pt{s}\notin\supp(\mathcal{T}\hspace{-1pt},\gamma).$

      Then there is $\mathcal{E}\in\gamma$ such that ${s}\in\expl({\mathcal{T}\hspace{-1pt},\mathcal{E}}).$
      Then  ${x}\leqslant_{\mathcal{T}}{0}_\mathcal{E}<_{\mathcal{T}}{s}$ by (e) of Lemma~\ref{lem.graft},
      so  ${x}<_{\mathcal{T}}{0}_\mathcal{E}<_{\mathcal{T}}{s}$ by Case~2.
      This contradicts ${s}\in\hspace{0.4pt}\sons_{\hspace{0.6pt}\mathcal{T}}(\hspace{-0.7pt}{x}\hspace{-1pt}).$
   \item[\textup{(ii)\ }]
      $\mathsurround=0pt{s}\in\supp(\mathcal{T}\hspace{-1pt},\gamma).$

      Then ${x}<_\mathcal{H}{s},$ so
      ${s}\notin\sons_{\hspace{0.5pt}\mathcal{H}\hspace{-0.5pt}}(\hspace{-0.7pt}{x}\hspace{-1pt})$ implies there is ${v}\in({x},{s})_\mathcal{H}.$
      Since $({x},{s})_{\mathcal{T}} =\varnothing,$
      we have ${v}\notin\supp(\mathcal{T}\hspace{-1pt},\gamma),$ so
      there is $\mathcal{E}\in\gamma$ such that ${v}\in\impl\mathcal{E}.$
      Then ${x}\leqslant_{\mathcal{H}\!}{0}_\mathcal{E}<_\mathcal{H}{v}<_\mathcal{H}{s}$
      by (e) of Lemma~\ref{lem.hybr},
      so ${x}\leqslant_{\mathcal{T}}{0}_\mathcal{E}<_{\mathcal{T}}{s},$
      whence ${x}<_{\mathcal{T}}{0}_\mathcal{E}<_{\mathcal{T}}{s}$ by Case~2.
      This contradicts ${s}\in\hspace{0.4pt}\sons_{\hspace{0.6pt}\mathcal{T}}(\hspace{-0.7pt}{x}\hspace{-1pt}).$
   \end{itemize}
\medskip

(\hspace{0.3pt}b\hspace{-0.5pt}) Suppose ${x},{y}\in\nodes\mathcal{H}$ and ${x}\:{\parallel}_{\!\mathcal{H}}\:{y}.$
We consider several cases\textup{:}
\begin{itemize}
\item[\textup{(i)\ }]
   $\mathsurround=0pt{x},{y}\in\supp(\mathcal{T}\hspace{-1pt},\gamma).$

   Then by (\hspace{0.3pt}b\hspace{-0.5pt}) of Lemma~\ref{lem.hybr}, ${x}'\coloneq {x}$ and ${y}'\coloneq {y}$ satisfy (b1) of Proposition~\ref{prop.hybr}.
\item[\textup{(ii)\ }]
   $\mathsurround=0pt\big|\,\{{x},{y}\}\cap\hspace{0.4pt}\supp(\mathcal{T}\hspace{-1pt},\gamma)\,\big|\,=\,1.$

   We may assume without loss of generality that
   ${x}\in\supp(\mathcal{T}\hspace{-1pt},\gamma)$ and ${y}\notin\supp(\mathcal{T}\hspace{-1pt},\gamma).$
   Then there is $\mathcal{G}\,{\in}\,\gamma$ such that ${y}\in\impl\mathcal{G}.$
   \begin{itemize}
   \item[\textup{(ii.1)\ }]
      $\mathsurround=0pt{x}\:{\parallel}_{\!\mathcal{H}}\:{0}_{\mathcal{G}}.$

   Then ${x}'\coloneq {x}$ and ${y}'\coloneq {0}_{\mathcal{G}}$
   satisfy (b1) of Proposition~\ref{prop.hybr}.
   \item[\textup{(ii.2)\ }]
      $\mathsurround=0pt{x}\leqslant_{\mathcal{H}\!}{0}_{\mathcal{G}}.$

      Then ${x}\leqslant_\mathcal{H} {y},$ which contradicts ${x}\:{\parallel}_{\!\mathcal{H}}\:{y}.$

   \item[\textup{(ii.3)\ }]
      $\mathsurround=0pt{x}>_{\mathcal{H}\!}{0}_{\mathcal{G}}.$

      Then by (f) of Lemma~\ref{lem.hybr},
      ${x}\in(\hspace{-0.5pt}\maxel\mathcal{G}){\hspace{-1pt}\downfootline}_\mathcal{H}.$
      Let ${r}\coloneq\roott_\mathcal{H}({x},\maxel\mathcal{G}).$
      We have ${r}\:{\parallel}_{\!\mathcal{G}}\:{y}$
      (else ${r}\geqslant_\mathcal{G}{y},$ which contradicts ${x}\:{\parallel}_{\!\mathcal{H}}\:{y}\text{)},$
      so ${x}'\coloneq {r}$ and ${y}'\coloneq {y}$ satisfy (b2) of Proposition~\ref{prop.hybr}.
   \end{itemize}
\item[\textup{(iii)\ }]
   $\mathsurround=0pt{x},{y}\notin\supp(\mathcal{T}\hspace{-1pt},\gamma).$

   Then there are $\mathcal{D},\mathcal{E}\,{\in}\,\gamma$ such that ${x}\in\impl\mathcal{D}$
   and ${y}\in\impl\mathcal{E}.$
   \begin{itemize}
   \item[\textup{(iii.1)\ }]
      $\mathsurround=0pt\mathcal{D}=\mathcal{E}.$

      Then by (a) of Lemma~\ref{lem.hybr}, ${x}'\coloneq {x}$ and ${y}'\coloneq {y}$ satisfy (b2) of Proposition~\ref{prop.hybr}.
   \item[\textup{(iii.2)\ }]
      $\mathsurround=0pt\mathcal{D}\neq\mathcal{E}\ \,\mathsf{and}\ \,
      {0}_{\mathcal{D}}\:{\parallel}_{\!\mathcal{H}}\:{0}_{\mathcal{E}}.$

      Then ${x}'\coloneq {0}_{\mathcal{D}}$ and ${y}'\coloneq {0}_{\mathcal{E}}$ satisfy (b1) of Proposition~\ref{prop.hybr}.
   \item[\textup{(iii.3)\ }]
      $\mathsurround=0pt\mathcal{D}\neq\mathcal{E}\ \,\mathsf{and}\ \,
      {0}_{\mathcal{D}}\:{\nparallel}_{\!\mathcal{H}}\:{0}_{\mathcal{E}}.$

      Then by (c) of Definition~\ref{def.c.f.grafts} we may assume without loss of generality that
      ${0}_\mathcal{E}\in(\hspace{-0.5pt}\maxel\mathcal{D}){\hspace{-1pt}\downfootline}_{\mathcal{T}}.$
      We have ${x}\:{\parallel}_{\!\mathcal{H}}\:{0}_{\mathcal{E}}$ ---
      otherwise ${x}\leqslant_{\mathcal{H}\!}{0}_{\mathcal{E}},$
      which contradicts ${x}\:{\parallel}_{\!\mathcal{H}}\:{y},$
      or ${x}>_{\mathcal{H}\!}{0}_{\mathcal{E}},$
      which contradicts ${0}_{\mathcal{E}}\in(\hspace{-0.5pt}\maxel\mathcal{D}) {\downfootline}_{\mathcal{T}}.$
      Let us consider ${x}_1\coloneq{x}$ and ${y}_1\coloneq{0}_{\mathcal{E}}.$
      Then ${x}_1\,{\parallel}_{\!\mathcal{H}}\:{y}_1$
      and $\big|\{{x}_1,{y}_1\}\cap\hspace{0.4pt}\supp(\mathcal{T}\hspace{-1pt},\gamma)\big|=1,$
      so by (ii) there are corresponding ${x}'_1\!\in{x}_1\!{\upfilledspoon}_{\!\mathcal{H}}\!$
      and ${y}'_1\!\in{y}_1\!{\upfilledspoon}_{\!\mathcal{H}}.$
      Then ${x}'\coloneq {x}'_1\!\in{x}\hspace{-0.5pt}{\upfilledspoon}_{\!\mathcal{H}}$
      and ${y}'\coloneq {y}'_1\!\in{y}{\upfilledspoon}_{\!\mathcal{H}}$ satisfy (b1) or (b2) of Proposition~\ref{prop.hybr}.
   \end{itemize}
\end{itemize}
\medskip

(c) Suppose $\mathcal{T}\hspace{-1pt}$ has the least node.
Then ${0}_{\mathcal{T}}\!\in\supp(\mathcal{T}\hspace{-1pt},\gamma)$ by (\hspace{0.3pt}b\hspace{-0.5pt}) of Lemma~\ref{lem.graft},
therefore ${0}_{\mathcal{T}}$ is the least node of $\mathcal{H}$ by (\hspace{0.3pt}b\hspace{-0.5pt}) and (\hspace{-0.4pt}d) of Lemma~\ref{lem.hybr}.
\medskip

(\hspace{-0.4pt}d) Suppose $\maxel\mathcal{T}=\varnothing.$
Let ${x}\in\nodes\mathcal{H}.$
If ${x}\in\supp(\mathcal{T}\hspace{-1pt},\gamma)\setminus
\{{0}_\mathcal{G}:\mathcal{G}\in\gamma\},$
then $\sons_{\hspace{0.5pt}\mathcal{H}\hspace{-0.5pt}}(\hspace{-0.7pt}{x}\hspace{-1pt})\neq\varnothing$ by (a) of Proposition~\ref{prop.hybr} and by (a) of Lemma~\ref{about trees},
hence ${x}\notin\maxel\mathcal{H}.$
If  ${x}\in\{\hspace{-1pt}{0}_\mathcal{G}\hspace{-1pt}\}\cup\hspace{0.4pt}\impl\mathcal{G}$
for some $\mathcal{G}\in\gamma,$
then ${x}\notin\maxel\mathcal{G}$ by (a) of Lemma~\ref{lem.graft},
so ${x}\notin\maxel\mathcal{H}$ by (a) of Lemma~\ref{lem.hybr}.
\medskip

(e) Suppose  $\mathcal{T}\hspace{-1pt}$ is $\kappa\mathsurround=0pt $-branching and for each $\mathcal{G}\,{\in}\,\gamma,$ the $\mathcal{G}$ is $\kappa\mathsurround=0pt $-branching.
Then $\mathcal{H}$ is $\kappa\mathsurround=0pt $-branching by (a) of Proposition~\ref{prop.hybr} and by (a) of Lemma~\ref{about trees}.
\medskip

(f) Suppose $\height\hspace{-0.4pt}\mathcal{T}\hspace{-1pt}\leqslant\omega$ and for each $\mathcal{G}\in\gamma,$
we have $\height\hspace{-0.4pt}\mathcal{G}\leqslant\omega.$ It is enough to prove that for each ${x}\in\nodes\mathcal{H},$ the ${x}\hspace{-0.3pt}{\upspoon}_{\!\mathcal{H}}$ is finite.

If ${x}\in\supp(\mathcal{T}\hspace{-1pt},\gamma),$
then ${x}\hspace{-0.3pt}{\upspoon}_{\!\mathcal{H}}\cap\supp(\mathcal{T}\hspace{-1pt},\gamma)\subseteq
{x}\hspace{-0.3pt}{\upspoon}_{\!\mathcal{T}},$
so ${x}\hspace{-0.3pt}{\upspoon}_{\!\mathcal{H}}\cap\supp(\mathcal{T}\hspace{-1pt},\gamma)$ is finite.
Suppose ${G}\in\gamma.$
If ${x}\hspace{-0.3pt}{\upspoon}_{\!\mathcal{H}}\cap\impl\mathcal{G}\neq\varnothing,$
then ${0}_\mathcal{G}\in{x}\hspace{-0.3pt}{\upspoon}_{\!\mathcal{H}}$ by (c) of Lemma~\ref{lem.hybr},
so ${0}_\mathcal{G}\in{x}\hspace{-0.3pt}{\upspoon}_{\!\mathcal{T}}.$
Then ${x}\in(\hspace{-0.5pt}\maxel\mathcal{G}){\hspace{-1pt}\downfootline}_{\mathcal{T}}$ by (\hspace{-0.4pt}d) of Lemma~\ref{lem.graft},
so ${x}\hspace{-0.3pt}{\upspoon}_{\!\mathcal{H}}\cap\impl\mathcal{G}\subseteq
\big(\!\roott_{\mathcal{T}}({x},\maxel\mathcal{G})\big){\upspoon}_{\!\mathcal{G}}$
by (b4) of Definition~\ref{def.hybr}.
This means that ${x}\hspace{-0.3pt}{\upspoon}_{\!\mathcal{H}}\cap\impl\mathcal{G}$ is finite, since
${v}\hspace{-0.6pt}{\upspoon}_{\!\mathcal{G}}$ is finite for every ${v}\in\nodes\mathcal{G}.$
So it is enough to show that the set $\{\mathcal{G}\in\gamma:{0}_\mathcal{G}\in{x}\hspace{-0.3pt}{\upspoon}_{\!\mathcal{T}}\}$
is finite. Since ${x}\hspace{-0.3pt}{\upspoon}_{\!\mathcal{T}}$ is finite, the (f) of Lemma~\ref{lem.graft} implies that this is indeed the case.

If ${x}\in\impl\mathcal{G}$ for some $\mathcal{G}\in\gamma,$
then ${x}\hspace{-0.3pt}{\upspoon}_{\!\mathcal{H}}={x}\hspace{-0.3pt}{\upspoon}_{\!\mathcal{G}}\cup
(\hspace{-1pt}{0}_\mathcal{G}\hspace{-1.4pt})\hspace{-1pt}{\upspoon}_{\!\mathcal{H}}$ by (g) of Lemma~\ref{lem.hybr}.
Since ${0}_\mathcal{G}\in\supp(\mathcal{T}\hspace{-1pt},\gamma),$ the $(\hspace{-1pt}{0}_\mathcal{G}\hspace{-1.4pt})\hspace{-1pt}{\upspoon}_{\!\mathcal{H}}$ is finite by the above,
therefore ${x}\hspace{-0.3pt}{\upspoon}_{\!\mathcal{H}}$ is finite.
\end{proof}

Finally we establish two properties of branches in $\hybr(\mathcal{T}\hspace{-1pt},\gamma):$

\begin{lem}\label{lem.reb.hybr}
   Suppose that  $\gamma$ is a consistent family of grafts for a tree $\mathcal{T}\hspace{-1pt}$
   and ${B}$ is a branch in $\hybr(\mathcal{T}\hspace{-1pt},\gamma).$
   Then\textup{:}

   \begin{itemize}
   \item [\textup{(a)}]
      If $\mathcal{G}\,{\in}\,\gamma$ and ${B}\cap\nodes\mathcal{G}\neq\varnothing,$

      then
      ${B}\cap\nodes\mathcal{G}$ is a branch in $\mathcal{G}.$
   \item [\textup{(\hspace{0.3pt}b\hspace{-1pt})}]
      If every graft in $\gamma$ has bounded chains,

      then
      ${B}\cap\supp(\mathcal{T}\hspace{-1pt},\gamma)$ is $\mathsurround=0pt\,\hybr(\mathcal{T}\hspace{-1pt},\gamma)$-cofinal in $\!{B}.$
   \end{itemize}
\end{lem}

\begin{proof} Let $\mathcal{H}\coloneq\hybr(\mathcal{T}\hspace{-1pt},\gamma).$

(a) Suppose that $\mathcal{G}\,{\in}\,\gamma,$ $\!{B} $ is a branch in $\mathcal{H},$
and ${x}\in{C}_\mathcal{G}\coloneq{B}\cap\nodes\mathcal{G}.$
We must prove that ${C}_\mathcal{G}$ is a branch in $\mathcal{G}.$
We consider two cases\textup{:}
\smallskip

\textit{\hspace{-1pt}Case 1.} $\exists {y}\in\hspace{1pt}\!{B}\setminus\big(\!(\hspace{-1pt}{0}_\mathcal{G}\hspace{-1.4pt})\hspace{-1pt}{\upspoon}_{\!\mathcal{H}\!}
\cup\nodes\mathcal{G}\big).$

By (f) of Lemma~\ref{about trees}, ${x}\hspace{-0.5pt}{\upfilledspoon}_{\!\mathcal{H}}\subseteq {B},$ so ${0}_{\mathcal{G}}\in{B}.$
Then since $\!{B}$ is a chain in $\mathcal{H}$ and
${y}\notin(\hspace{-1pt}{0}_\mathcal{G}\hspace{-1.4pt})\hspace{-1pt}{\upspoon}_{\!\mathcal{H}\!}\cup\{{0}_{\mathcal{G}}\},$
we have ${y}>_{\mathcal{H}\!}{0}_{\mathcal{G}}.$
Then ${y}\in(\hspace{-0.5pt}\maxel\mathcal{G}){\hspace{-1pt}\downfootline}_\mathcal{H}$
by (f) of Lemma~\ref{lem.hybr}.
Let ${r}\coloneq\roott_\mathcal{H}({y},\maxel\mathcal{G}).$
We have ${y}\in{B},$ so by (f) of Lemma~\ref{about trees} ${y}{\upfilledspoon}_{\!\mathcal{H}}\subseteq {B},$
hence ${r}{\upfilledspoon}_{\!\mathcal{H}}\subseteq {B}.$
Now ${r}{\upfilledspoon}_{\!\mathcal{G}}\subseteq\nodes\mathcal{G}$
and by (a) of Lemma~\ref{lem.hybr},
${r}{\upfilledspoon}_{\!\mathcal{G}}\subseteq {r}{\upfilledspoon}_{\!\mathcal{H}},$
so ${r}{\upfilledspoon}_{\!\mathcal{G}}\subseteq{B}\cap\nodes\mathcal{G}={C}_\mathcal{G}.$
Further,  ${r}{\upfilledspoon}_{\!\mathcal{G}}$ is branch in $\mathcal{G}$ by (g) of Lemma~\ref{about trees},
${r}{\upfilledspoon}_{\!\mathcal{G}}\subseteq {C}_\mathcal{G},$
and ${C}_\mathcal{G}$ is a chain in $\mathcal{G},$
therefore  ${C}_\mathcal{G}$ is a branch in $\mathcal{G}.$
\smallskip

\textit{\hspace{-1pt}Case 2.} $\!{B}\subseteq(\hspace{-1pt}{0}_\mathcal{G}\hspace{-1.4pt})\hspace{-1pt}{\upspoon}_{\!\mathcal{H}\!}\cup\nodes\mathcal{G}.$

Since ${C}_\mathcal{G}$ is a chain in $\mathcal{G},$
then by (c) of Lemma~\ref{about trees} there is $\!{B}_\mathcal{G}\in\branches\mathcal{G}$
such that ${C}_\mathcal{G}\subseteq {B}_\mathcal{G}.$
Now $(\hspace{-1pt}{0}_\mathcal{G}\hspace{-1.4pt})\hspace{-1pt}{\upspoon}_{\!\mathcal{H}\!}$ and $\!{B}_\mathcal{G}$
are chains in $\mathcal{H}$
and  $\!{B}_\mathcal{G}\subseteq\nodes\mathcal{G}\subseteq
(\hspace{-1pt}{0}_\mathcal{G}\hspace{-1.4pt})\!{\downfilledspoon}_{\!\mathcal{H}\!},$
therefore $(\hspace{-1pt}{0}_\mathcal{G}\hspace{-1.4pt})\hspace{-1pt}{\upspoon}_{\!\mathcal{H}}\cup {B}_\mathcal{G}$
is a chain in $\mathcal{H}.$
Furthermore, $\!{B}$ is a branch in $\mathcal{H},$ by Case~2
$${B}\subseteq(\hspace{-1pt}{0}_\mathcal{G}\hspace{-1.4pt})\hspace{-1pt}{\upspoon}_{\!\mathcal{H}}\cup\big(\!{B}\cap\nodes\mathcal{G}\big)
=(\hspace{-1pt}{0}_\mathcal{G}\hspace{-1.4pt})\hspace{-1pt}{\upspoon}_{\!\mathcal{H}}\cup {C}_\mathcal{G}\subseteq(\hspace{-1pt}{0}_\mathcal{G}\hspace{-1.4pt})\hspace{-1pt}{\upspoon}_{\!\mathcal{H}}\cup {B}_\mathcal{G},$$
and $(\hspace{-1pt}{0}_\mathcal{G}\hspace{-1.4pt})\hspace{-1pt}{\upspoon}_{\!\mathcal{H}}\cup {B}_\mathcal{G}$
is a chain in $\mathcal{H},$ so
$${B}=(\hspace{-1pt}{0}_\mathcal{G}\hspace{-1.4pt})\hspace{-1pt}{\upspoon}_{\!\mathcal{H}}\cup {C}_\mathcal{G}
=(\hspace{-1pt}{0}_\mathcal{G}\hspace{-1.4pt})\hspace{-1pt}{\upspoon}_{\!\mathcal{H}}\cup {B}_\mathcal{G}.$$
Then ${C}_\mathcal{G}={B}_\mathcal{G}$ because
$(\hspace{-1pt}{0}_\mathcal{G}\hspace{-1.4pt})\hspace{-1pt}{\upspoon}_{\!\mathcal{H}}\cap{C}_\mathcal{G}=\varnothing$
and $(\hspace{-1pt}{0}_\mathcal{G}\hspace{-1.4pt})\hspace{-1pt}{\upspoon}_{\!\mathcal{H}}\cap{B}_\mathcal{G}=\varnothing,$
so ${C}_\mathcal{G}$ is a branch in $\mathcal{G}.$

\medskip
(\hspace{0.3pt}b\hspace{-0.5pt}) Suppose that every $\mathcal{G}\in\gamma$ has bounded chains
and $\!{B}\in\branches\mathcal{H}.$
Let ${x}\in{B}$ and ${C}\coloneq {B}\cap\hspace{0.4pt}\supp(\mathcal{T}\hspace{-1pt},\gamma).$
We must prove that ${x}\in{C}\,{\upfootline}_{\!\mathcal{H}}.$
If ${x}\in\supp(\mathcal{T}\hspace{-1pt},\gamma),$ then ${x}\in{C},$
so ${x}\in{C}\,{\upfootline}_{\!\mathcal{H}}.$
If ${x}\notin\supp(\mathcal{T}\hspace{-1pt},\gamma),$ then there is $\mathcal{G}\in\gamma$
such that ${x}\in\impl\mathcal{G}.$
We have $\!{B}\cap\nodes\mathcal{G}\neq\varnothing,$
so by (a), $\!{B}_\mathcal{G}\coloneq {B}\cap\nodes\mathcal{G}$ is a branch in $\mathcal{G}.$
Now, by (h) of Lemma~\ref{about trees}, there is ${m}\in\maxel\mathcal{G}$
such that $\!{B}_\mathcal{G}={m}{\upfilledspoon}_{\!\mathcal{G}}.$
Then
${x}\in{B}_\mathcal{G}={m}{\upfilledspoon}_{\!\mathcal{G}}\subseteq{m}{\upfilledspoon}_{\!\mathcal{H}}$
and ${m}\in\supp(\mathcal{T}\hspace{-1pt},\gamma),$
whence ${m}\in{C},$ so ${x}\in{C}\,{\upfootline}_{\!\mathcal{H}}.$
\end{proof}

\section{Foliage hybrid operation}
\label{sect.f.hybr}

In this section we construct the foliage hybrid operation and establish its properties --- see Definition~\ref{def.f.hybr} and Proposition~\ref{prop.f.hybr}.
The foliage hybrid operation  modifies a given foliage tree $\mathbf{F}$ with the help of a family $\varphi$ of special foliage trees, which we call \emph{foliage grafts}. This operation deals with nonincreasing foliage trees and it acts as follows. At first, applying the hybrid operation (see Section~\ref{sect.hybr}) to $\skeleton\mathbf{F}$ and $\{\skeleton\mathbf{G}\,{:}\,\mathbf{G}\in\varphi\},$ we obtain a tree. After that we define leaves at nodes of this tree by using leaves of $\mathbf{F}$ and leaves of foliage grafts~$\mathcal{G},$ $\mathcal{G}\in\varphi.$

\begin{deff}\label{def.fol.graft}
   Let $\mathbf{F}$ be a nonincreasing foliage tree.
   Then
   a \textbf{foliage graft} for $\mathbf{F}$ is a foliage tree~$\mathbf{G}$
   such that\textup{:}

   \begin{itemize}
   \item[\textup{(a)}]
      $\mathsurround=0pt\mathbf{G}$ is nonincreasing;
   \item[\textup{(\hspace{0.3pt}b\hspace{-1pt})}]
      $\mathsurround=0pt\skeleton\mathbf{G}$ is a graft for $\skeleton\mathbf{F}$\\
      (hence ${0}_{\hspace{-0.2pt}\mathbf{G}}\in\nodes\mathbf{F}$ and
      $\mathsurround=0pt\hspace{1pt}\maxel\mathbf{G}\subseteq\nodes\mathbf{F}$);
   \item[\textup{(c)}]
      $\mathsurround=0pt\mathbf{G}_{\hspace{-0.3pt}{0}_{\hspace{-0.6pt}\mathbf{G}\!}}\subseteq\mathbf{F}_{\hspace{-1.2pt}
      {0}_{\hspace{-0.6pt}\mathbf{G}}};$
   \item[\textup{(\hspace{-0.4pt}d)}]
      $\mathsurround=0pt\forall{m}\,{\in}\maxel\mathbf{G}
      \;[\mathbf{G}_{\hspace{-0.3pt}{m\!}}=\mathbf{F}_{\hspace{-1.2pt}{m}}
      ].$
   \end{itemize}
   The set\vspace{-1ex}
   $$\cut(\mathbf{F},\mathbf{G})\coloneq
   \mathbf{F}_{\hspace{-1.2pt}{0}_{\hspace{-0.6pt}\mathbf{G}}}\!\setminus
   \mathbf{G}_{\hspace{-0.3pt}{0}_{\hspace{-0.6pt}\mathbf{G}}}
   $$
   is called the \textbf{cut} from
   $\mathbf{F}$ by $\mathbf{G}.$
\end{deff}

\begin{deff}\label{def.c.f.adj}
   Let $\mathbf{F}$ be a nonincreasing foliage tree.
   Then
   $\varphi$ is a \textbf{consistent} family of foliage grafts for $\mathbf{F}$
   iff

   \begin{itemize}
   \item [\textup{(a)}]
      $\forall\mathbf{G}\,{\in}\,\varphi\;
      [\hspace{1pt}\mathbf{G}$ is a foliage graft for $\mathbf{F}\mathsurround=0pt\,];$
   \item [\textup{(\hspace{0.3pt}b\hspace{-1pt})}]
      $\forall\mathbf{D}\,{\neq}\,\mathbf{E}\,{\in}\,\varphi\;
      [\hspace{1pt}\skeleton\mathbf{D}\neq\skeleton\mathbf{E}\,];$
   \item [\textup{(c)}]
      $\{\skeleton\mathbf{G}\,{:}\:\mathbf{G}\,{\in}\,\varphi\}$ is a consistent family of grafts for $\skeleton\mathbf{F}.$
   \end{itemize}
   The set\vspace{-0.5ex}
   $$
   \loss(\mathbf{F},\varphi)\coloneq\bigcup_{\mathbf{G}\in\varphi}\cut(\mathbf{F},\mathbf{G})\vspace{-1ex}
   $$
   is called the \textbf{loss} of $\mathbf{F}$ on $\varphi.$
\end{deff}

Now we define the foliage hybrid operation:

\begin{deff}\label{def.f.hybr}
   Let $\varphi$ be a consistent family of foliage grafts for a nonincreasing foliage tree $\mathbf{F}.$
   Then
   the \textbf{foliage hybrid} of $\mathbf{F}$ and $\varphi$
    --- in symbols, $\fhybr(\mathbf{F},\varphi)$ ---
   is a foliage tree $\mathbf{H}$
   such that\textup{:}

   \begin{itemize}
   \item[\textup{(a)}]
      $\mathsurround=0pt\skeleton\mathbf{H}\coloneq
      \hybr\!\big(\!\skeleton\mathbf{F},\{\skeleton\mathbf{G}:\mathbf{G}\,{\in}\,\varphi\}\!\big);$
   \item[\textup{(\hspace{0.3pt}b\hspace{-1pt})}]
      $\mathbf{H}_{{x}}\coloneq
      \begin{cases}
         \mathbf{G}_{\hspace{-0.3pt}{x}}\!\setminus\loss(\mathbf{F},\varphi),
         &\text
         {if $\ {x}\in\impl\mathbf{G}\ $ for some $\ \mathbf{G}\in\varphi;$}
      \\ \,\mathbf{F}_{\hspace{-1.2pt}{x}}\!\setminus\loss(\mathbf{F},\varphi)\vphantom{\tilde{\big\langle\rangle}},
         &\text
         {otherwise
         \textup{\big(}i.e., when ${x}\in\supp(\mathbf{F},\varphi)
         \textup{\big)},$}
      \end{cases}
      $
   \end{itemize}
   where\vspace{-0.5ex}
   $$\supp(\mathbf{F},\varphi)\coloneq
   \supp\!\big(\!\skeleton\mathbf{F},\{\skeleton\mathbf{G}\,{:}\:\mathbf{G}\,{\in}\,\varphi\}\!\big).$$
\end{deff}

Note that the hybrid of $\skeleton\mathbf{F}$ and $\{\skeleton\mathbf{G}\hspace{-1pt}:\hspace{-1pt}\mathbf{G}\,{\in}\,\varphi\}$
is a tree by Proposition~\ref{hybr.is.tree}, so a foliage hybrid is indeed a foliage tree.

\begin{lem}\label{lem.reb}
   Suppose that $\varphi$ is a consistent family of foliage grafts for a nonincreasing foliage tree $\mathbf{F}$ and $\,\mathbf{H}=\fhybr(\mathbf{F},\varphi).$
   Then\textup{:}

   \begin{itemize}
   \item [\textup{(a)}]
      $\mathsurround=0pt\forall\mathbf{G}\,{\in}\,\varphi\ \forall{x}\,{\in}\nodes\mathbf{G}\;
      \big[\,\mathbf{H}_{{x}}=\mathbf{G}_{\hspace{-0.3pt}{x}}\setminus\loss(\mathbf{F},\varphi)
      \,\big].$
   \item [\textup{(\hspace{0.3pt}b\hspace{-1pt})}]
      For any set ${A}\!:$
      \\
      if $\ \forall\mathbf{G}\,{\in}\,\varphi\;
      \big[\,
      {A}\subseteq\mathbf{G}_{\hspace{-0.3pt}{0}_{\hspace{-0.6pt}\mathbf{G}}}
      \enskip\mathsf{or}\enskip
      {A}\cap\mathbf{F}_{\hspace{-1.2pt}{0}_{\hspace{-0.6pt}\mathbf{G}}}=\varnothing
      \,\big],$
      then
      ${A}\cap\loss(\mathbf{F},\varphi)=\varnothing.$
      \hfill$\Box$
   \end{itemize}
\end{lem}

Now we establish several properties of the foliage hybrid operation:

\begin{prop}\label{prop.f.hybr}
   Suppose that $\varphi$ is a consistent family of foliage grafts for a nonincreasing foliage tree $\mathbf{F}$ and $\mathbf{H}=\fhybr(\mathbf{F},\varphi).$
   Then\textup{:}

   \begin{itemize}
   \item [\textup{(a)}]
      $\mathsurround=0pt\mathbf{H}\,$ is nonincreasing.
   \item [\textup{(\hspace{0.3pt}b\hspace{-1pt})}]
      If $\hspace{1pt}\mathbf{F}$ and each $\mathbf{G}\in\varphi$ are splittable,\\
      then
      $\mathbf{H}$ is splittable.
   \item [\textup{(c)}]
      If $\hspace{1pt}\mathbf{F}$ and each $\mathbf{G}\in\varphi$  are locally strict,\\
      then
      $\mathbf{H}$ is locally strict.
   \item [\textup{(\hspace{-0.4pt}d)}]
      If $\hspace{1pt}\mathbf{F}$ is complete \textup{(}has strict branches\textup{)} and splittable, and each $\mathbf{G}\in\varphi$ has bounded chains,\\
      then $\mathbf{H}$ is complete \textup{(}has strict branches\textup{)}.
   \item [\textup{(e)}]
      If $\hspace{1pt}\mathbf{F}$ and each $\mathbf{G}\in\varphi$  are open in a space ${X},$
      \\
      then
      $\mathbf{H}$ is open in the subspace ${X}\setminus\loss(\mathbf{F},\varphi)\!$ of ${X}.$
   \end{itemize}
\end{prop}

\begin{proof}
(a) We must prove that the foliage tree $\mathbf{H}$ is nonincreasing.
Suppose ${x},{y}\in\nodes\mathbf{H}$ and ${x}<_{\mathbf{H}}{y}.$
Then one of conditions (b1)--(b5) of Definition~\ref{def.hybr} holds.
For example, if (b4) holds, then there is $\mathbf{G}\in\varphi$ such that
$${x}\in\impl\mathbf{G},\quad {y}\in\supp(\mathbf{F},\varphi),\quad \text{and}\quad {x}<_{\mathbf{G}}{r}\coloneq\roott_\mathbf{F}\hspace{-0.5pt}({y},\maxel\mathbf{G})
\leqslant_{\mathbf{F}}{y}.$$
Then $\mathbf{G}_{\hspace{-0.3pt}{x}}\supseteq\mathbf{G}_{\hspace{-0.3pt}{r}}$
and $\mathbf{F}_{\hspace{-1.2pt}{r}}\supseteq\mathbf{F}_{\hspace{-1.2pt}{y}}$
because $\mathbf{G}$ and $\mathbf{F}$ are nonincreasing by Definition~\ref{def.fol.graft}.
Then $\mathbf{H}_{{x}}\supseteq\mathbf{H}_{r}$ by (a) of Lemma~\ref{lem.reb}
and $\mathbf{H}_{r}\supseteq\mathbf{H}_{{y}}$ by (\hspace{0.3pt}b\hspace{-0.5pt}) of Definition~\ref{def.f.hybr},
so $\mathbf{H}_{{x}}\supseteq\mathbf{H}_{{y}}.$ The other cases are similar.
\medskip

(\hspace{0.3pt}b\hspace{-0.5pt}) Suppose $\mathbf{F}$ and each $\mathbf{G}\in\varphi$ are splittable; we must prove that
$\mathbf{H}$ is also splittable.
By (a), $\mathbf{H}$ is nonincreasing.
Let ${x},{y}\in\nodes\mathbf{H}$ and ${x}\:{\parallel}_{\!\mathbf{H}}\:{y}.$
Then by (\hspace{0.3pt}b\hspace{-0.5pt}) of Proposition~\ref{prop.hybr}, there a ${x}'\in{x}\hspace{-0.5pt}{\upfilledspoon}_{\!\mathbf{H}}\!$ and ${y}'\in{y}{\upfilledspoon}_{\!\mathbf{H}}$
such that\vspace{-1ex}
\begin{align}\label{supp}
\qquad\qquad\qquad&\text{either}\qquad&{x}'\!,{y}'\,{\in}\,\supp(\mathbf{F},\varphi)\enskip\mathsf{and}\enskip {x}'\,{\parallel}_{\!\mathbf{F}}\:{y}'\:&\qquad\qquad\qquad\\ \label{add}
\qquad\qquad\qquad&\text{or}\qquad&\exists\mathbf{G}\,{\in}\,\varphi\,\big[\,{x}'\!,{y}'\,{\in}\hspace{1pt}\nodes\mathbf{G}\enskip\mathsf{and}\enskip {x}'\,{\parallel}_{\!\mathbf{G}}\:{y}'\,\big].&\qquad\qquad\qquad
\end{align}
If (\ref{supp}) holds, then
$\mathbf{H}_{{x}}\subseteq\mathbf{H}_{{x}'\!}\subseteq\mathbf{F}_{\hspace{-1.2pt}{x}'\!}$
and $\mathbf{H}_{{y}}\subseteq\mathbf{H}_{{y}'\!}\subseteq\mathbf{F}_{\hspace{-1.2pt}{y}'\!}$
since $\mathbf{H}$ is nonincreasing and by (\hspace{0.3pt}b\hspace{-0.5pt}) of Definition~\ref{def.f.hybr},
and $\mathbf{F}_{\hspace{-1.2pt}{x}'\!}\cap\mathbf{F}_{\hspace{-1.2pt}{y}'\!}=\varnothing$ because $\mathbf{F}$ is splittable,
so $\mathbf{H}_{{x}}\cap\,\mathbf{H}_{{y}}=\varnothing.$
If (\ref{add}) holds, then
$\mathbf{H}_{{x}'\!}\subseteq\mathbf{G}_{\hspace{-0.3pt}{x}'\!}$
and $\mathbf{H}_{{y}'\!}\subseteq\mathbf{G}_{\hspace{-0.3pt}{y}'\!}$
by (a) of Lemma~\ref{lem.reb},
and $\mathbf{G}_{\hspace{-0.3pt}{x}'\!}\cap\mathbf{G}_{\hspace{-0.3pt}{y}'\!}=\varnothing$ since $\mathbf{G}$ is splittable,
so $\mathbf{H}_{{x}}\cap\mathbf{H}_{{y}}=\varnothing$ again.
\medskip

(c) Suppose that $\mathbf{F}$ and each $\mathbf{G}\in\varphi$  are locally strict; we must prove that
$\mathbf{H}$ is also locally strict. Let ${x}\in\nodes\mathbf{H}\setminus\maxel\mathbf{H}.$
Then $\sons_{\hspace{0.2pt}\mathbf{H}\hspace{-0.7pt}}(\hspace{-0.7pt}{x}\hspace{-1pt})\neq\varnothing$ by (a) of Lemma~\ref{about trees}.
We consider two cases\textup{:}
\smallskip

\textit{\hspace{-1pt}Case 1.} $\exists\mathbf{G}\,{\in}\,\varphi\,\big[\,{x}\in\{{0}_{\hspace{-0.2pt}\mathbf{G}}\}\cup\hspace{0.4pt}\impl\mathbf{G}\,\big].$

By (a) of Proposition~\ref{prop.hybr} we have
$\sons_{\hspace{0.2pt}\mathbf{G}\hspace{-1.2pt}}(\hspace{-0.7pt}{x}\hspace{-1pt})=
\sons_{\hspace{0.2pt}\mathbf{H}\hspace{-0.7pt}}(\hspace{-0.7pt}{x}\hspace{-1pt})\neq\varnothing,$
so ${x}\in\nodes\mathbf{G}\setminus\maxel\mathbf{G}.$
Then
$$\mathbf{G}_{\hspace{-0.3pt}{x}}\,\equiv\!\bigsqcup_{{s}\in\hspace{0.4pt}\sons_{\hspace{0.2pt}\mathbf{G}\hspace{-1.2pt}}
(\hspace{-0.7pt}{x}\hspace{-1pt})}\mathbf{G}_{\hspace{-0.3pt}{s}}\vspace{-2ex}
$$
since $\mathbf{G}$ is locally strict, hence
$$\mathbf{G}_{\hspace{-0.3pt}{x}}\,{\setminus}\,\loss(\mathbf{F},\varphi)\:\equiv
\!\bigsqcup_{{s}\in\hspace{0.4pt}\sons_{\hspace{0.2pt}\mathbf{H}\hspace{-0.7pt}}
(\hspace{-0.7pt}{x}\hspace{-1pt})}
\big(\mathbf{G}_{\hspace{-0.3pt}{s}}{\setminus}\loss(\mathbf{F},\varphi)\big).
$$
Since ${x}\in\nodes\mathbf{G}$ and
$\sons_{\hspace{0.2pt}\mathbf{H}\hspace{-0.7pt}}(\hspace{-0.7pt}{x}\hspace{-1pt})=
\sons_{\hspace{0.2pt}\mathbf{G}\hspace{-1.2pt}}(\hspace{-0.7pt}{x}\hspace{-1pt})\subseteq
\nodes\mathbf{G},$
then by (a) of Lemma~\ref{lem.reb} we have
$$\mathbf{H}_{{x}}\,\equiv\!\bigsqcup_{{s}\in\hspace{0.4pt}
\sons_{\hspace{0.2pt}\mathbf{H}\hspace{-0.7pt}}(\hspace{-0.7pt}{x}\hspace{-1pt})}
\mathbf{H}_{{s}}\,.\vspace{-1ex}
$$

\textit{\hspace{-1pt}Case 2.} ${x}\in\supp(\mathbf{F},\varphi)\,{\setminus}\,\{{0}_{\hspace{-0.2pt}\mathbf{G}}:\mathbf{G}\in\varphi\}.$

By (a) of Proposition~\ref{prop.hybr} we have
$\sons_{\hspace{0.5pt}\mathbf{F}\hspace{-0.5pt}}(\hspace{-0.7pt}{x}\hspace{-1pt})=\sons_{\hspace{0.2pt}\mathbf{H}\hspace{-0.7pt}}(\hspace{-0.7pt}{x}\hspace{-1pt})\neq\varnothing,$
so ${x}\in\nodes\mathbf{F}\setminus\maxel\mathbf{F}.$
Then
$$\mathbf{F}_{\hspace{-1.2pt}{x}}\,\equiv\!\bigsqcup_{{s}\in\hspace{0.4pt}
\sons_{\hspace{0.5pt}\mathbf{F}\hspace{-0.5pt}}(\hspace{-0.7pt}{x}\hspace{-1pt})}\mathbf{F}_{\hspace{-1.2pt}{s}}
\vspace{-2ex}
$$
since $\mathbf{F}$ is locally strict, whence
$$\mathbf{F}_{\hspace{-1.2pt}{x}}\,{\setminus}\,\loss(\mathbf{F},\varphi)\,\equiv\!
\bigsqcup_{{s}\in\hspace{0.4pt}\sons_{\hspace{0.2pt}\mathbf{H}\hspace{-0.7pt}}
(\hspace{-0.7pt}{x}\hspace{-1pt})}\big(\mathbf{F}_{\hspace{-1.2pt}{s}}{\setminus}\loss(\mathbf{F},\varphi)\big).
$$
Since $\sons_{\hspace{0.2pt}\mathbf{H}\hspace{-0.7pt}}(\hspace{-0.7pt}{x}\hspace{-1pt})=
\sons_{\hspace{0.5pt}\mathbf{F}\hspace{-0.5pt}}(\hspace{-0.7pt}{x}\hspace{-1pt})\subseteq
\nodes\mathbf{F}$
and $\sons_{\hspace{0.2pt}\mathbf{H}\hspace{-0.7pt}}(\hspace{-0.7pt}{x}\hspace{-1pt})\subseteq
\nodes\mathbf{H},$
we have
$$\sons_{\hspace{0.2pt}\mathbf{H}\hspace{-0.7pt}}(\hspace{-0.7pt}{x}\hspace{-1pt})
\:\subseteq\:\nodes\mathbf{F}\cap\nodes\mathbf{H}
\:=\:\supp(\mathbf{F},\varphi)$$
by (\hspace{0.3pt}b\hspace{-0.5pt}) of Lemma~\ref{lem.hybr}.
Also we have ${x}\in\supp(\mathbf{F},\varphi),$
so by (\hspace{0.3pt}b\hspace{-0.5pt}) of Definition~\ref{def.f.hybr} we get
$$\mathbf{H}_{{x}}\,\equiv\!\bigsqcup_{{s}\in\hspace{0.4pt}
\sons_{\hspace{0.2pt}\mathbf{H}\hspace{-0.7pt}}(\hspace{-0.7pt}{x}\hspace{-1pt})}
\mathbf{H}_{{s}}.
$$

(\hspace{-0.4pt}d) First, suppose that $\mathbf{F}$ is complete and splittable, and each $\mathbf{G}\in\varphi$ has bounded chains. We must prove that $\mathbf{H}$ is complete.
Since $\mathbf{F}$ is complete, we have $\nodes\mathbf{F}\neq\varnothing,$
so $\nodes\mathbf{H}\neq\varnothing$
because either $\varphi=\varnothing$ and $\nodes\mathbf{H}=\nodes\mathbf{F}$
or ${0}_{\hspace{-0.2pt}\mathbf{G}}\in\nodes\mathbf{H}$ for some $\mathbf{G}\in\varphi.$
Suppose that $\!{B}_\mathbf{H}\in\branches\mathbf{H}$
and ${C}\coloneq {B}_\mathbf{H}\cap\hspace{0.4pt}\supp(\mathbf{F},\varphi).$
Then it follows by (\hspace{0.3pt}b\hspace{-0.5pt}) of Lemma~\ref{lem.reb.hybr} that
${C}$ is $\mathsurround=0pt\mathbf{H}$-cofinal in ${B}_\mathbf{H},$
and then ${C}\neq\varnothing$ since $\!{B}_\mathbf{H}\neq\varnothing.$
By (a), $\mathbf{H}$ is nonincreasing, so by (a) of Lemma~\ref{about foliage trees} we have
\begin{equation}\label{*143.1}
\fruit_{\hspace{0.4pt}\mathbf{H}}\hspace{-0.7pt}(\hspace{-0.5pt}{B}_\mathbf{H}\hspace{-0.5pt})=
\fruit_{\hspace{0.4pt}\mathbf{H}}\hspace{-0.7pt}(\hspace{-0.5pt}{C}).
\end{equation}
Since ${C}\subseteq\supp(\mathbf{F},\varphi),$
then by (\hspace{0.3pt}b\hspace{-0.5pt}) of Definition~\ref{def.f.hybr} we get
\begin{equation}\label{*143.2}
\fruit_{\hspace{0.4pt}\mathbf{H}}\hspace{-0.7pt}(\hspace{-0.5pt}{C})\,=\hspace{0.5pt}
\bigcap_{{x}\in\hspace{0.3pt}{C}}\big(\mathbf{F}_{\hspace{-1.2pt}{x}}{\setminus}\loss(\mathbf{F},\varphi)\big)\,=\,
\fruit_{\hspace{0.4pt}\mathbf{F}}\hspace{-0.7pt}(\hspace{-0.5pt}{C})
\setminus\loss(\mathbf{F},\varphi).
\end{equation}
Further, since ${C}$ is a chain in $\mathbf{H}$ and ${C}\subseteq\supp(\mathbf{F},\varphi),$ we see
by (\hspace{0.3pt}b\hspace{-0.5pt}) of Lemma~\ref{lem.hybr} that ${C}$ is a chain in $\mathbf{F}.$
Then by (c) of Lemma~\ref{about trees} there is $\!{B}_\mathbf{F}\in\branches\mathbf{F}$ such that
${C}\subseteq {B}_\mathbf{F},$
so we have
\begin{equation*}
\fruit_{\hspace{0.4pt}\mathbf{F}}\hspace{-0.7pt}(\hspace{-0.5pt}{C})\supseteq
\fruit_{\hspace{0.4pt}\mathbf{F}}\hspace{-0.7pt}(\hspace{-0.5pt}{B}_\mathbf{F}\hspace{-0.5pt})
\neq\varnothing
\end{equation*}
because $\mathbf{F}$ is complete. It follows that
\begin{equation*}
\fruit_{\hspace{0.4pt}\mathbf{H}}\hspace{-0.7pt}(\hspace{-0.5pt}{B}_\mathbf{H}\hspace{-0.5pt})\supseteq
\fruit_{\hspace{0.4pt}\mathbf{F}}\hspace{-0.7pt}(\hspace{-0.5pt}{B}_\mathbf{F}\hspace{-0.5pt})
\!\setminus\!\loss(\mathbf{F},\varphi)
\,\quad\mathsf{and}\quad
\fruit_{\hspace{0.4pt}\mathbf{F}}\hspace{-0.7pt}(\hspace{-0.5pt}{B}_\mathbf{F}\hspace{-0.5pt})
\neq\varnothing,
\end{equation*}
so it is enough to prove
$$\fruit_{\hspace{0.4pt}\mathbf{F}}\hspace{-0.7pt}(\hspace{-0.5pt}{B}_\mathbf{F}\hspace{-0.5pt})
\cap\loss(\mathbf{F},\varphi)=\varnothing.
$$
Then, by (\hspace{0.3pt}b\hspace{-0.5pt}) of Lemma~\ref{lem.reb}, it is enough to show that
for each $\mathbf{G}\in\varphi,$
\begin{equation}\label{*4.5}
\text{either}\qquad
\fruit_{\hspace{0.4pt}\mathbf{F}}\hspace{-0.7pt}(\hspace{-0.5pt}{B}_\mathbf{F}\hspace{-0.5pt})
\subseteq\mathbf{G}_{\hspace{-0.3pt}{0}_{\hspace{-0.6pt}\mathbf{G}}}
\qquad\text{or}\qquad
\fruit_{\hspace{0.4pt}\mathbf{F}}\hspace{-0.7pt}(\hspace{-0.5pt}{B}_\mathbf{F}\hspace{-0.5pt})
\cap\mathbf{F}_{\hspace{-1.2pt}{0}_{\hspace{-0.6pt}\mathbf{G}}}=\varnothing.
\end{equation}
To show it we consider two cases\textup{:}
\smallskip

\textit{\hspace{-1pt}Case 1.} ${0}_{\hspace{-0.2pt}\mathbf{G}}\in{B}_\mathbf{F}.$

First let us prove that ${0}_{\hspace{-0.2pt}\mathbf{G}}\in{B}_\mathbf{H}.$
If not, then by (\hspace{-0.4pt}d) of Lemma~\ref{about trees} there is ${b}\in{B}_\mathbf{H}$
such that ${b}\:{\parallel}_{\!\mathbf{H}}\:{0}_{\hspace{-0.2pt}\mathbf{G}}.$
Since ${C}$ is $\mathsurround=0pt\mathbf{H}$-cofinal in ${B}_\mathbf{H},$
there is ${c}\in{C}$ such that ${c}\geqslant_\mathbf{H}\!{b},$
so ${c}\:{\parallel}_{\!\mathbf{H}}\:{0}_{\hspace{-0.2pt}\mathbf{G}}$ by (\hspace{0.3pt}b\hspace{-0.5pt}) of Lemma~\ref{about trees}.
Both ${c}$ and ${0}_{\hspace{-0.2pt}\mathbf{G}}$ lie in $\supp(\mathbf{F},\varphi),$
so we have ${c}\:{\parallel}_{\!\mathbf{F}}\:{0}_{\hspace{-0.2pt}\mathbf{G}},$
but this contradicts ${c}\:{\in}\,{C}\,{\subseteq}\,{B}_\mathbf{F}\,{\ni}\,{0}_{\hspace{-0.2pt}\mathbf{G}}.$

Now ${0}_{\hspace{-0.2pt}\mathbf{G}}\in{B}_\mathbf{H}.$
Then by (a) of Lemma~\ref{lem.reb.hybr}, $\!{B}_\mathbf{H}\cap\nodes\mathbf{G}\in\branches\hspace{-0.5pt}\mathbf{G},$
so by (h) of Lemma~\ref{about trees}, $\!{B}_\mathbf{H}\cap\nodes\mathbf{G}={m}{\upfilledspoon}_{\!\mathbf{G}}$
for some ${m}\in\maxel\mathbf{G}.$
Since $\maxel\mathbf{G}\subseteq\supp(\mathbf{F},\varphi),$
we have ${m}\in{B}_\mathbf{H}\cap\hspace{0.4pt}\supp(\mathbf{F},\varphi)={C}\subseteq {B}_\mathbf{F},$ that is, ${m}\in{B}_\mathbf{F}.$
Then\vspace{-0.5ex}
$$\fruit_{\hspace{0.4pt}\mathbf{F}}\hspace{-0.7pt}(\hspace{-0.5pt}{B}_\mathbf{F}\hspace{-0.5pt})
\,\subseteq\,\mathbf{F}_{\hspace{-1.2pt}{m}}\hspace{1pt}=\,\mathbf{G}_{\hspace{-0.3pt}{m}}
\hspace{1pt}\subseteq\,\mathbf{G}_{\hspace{-0.3pt}{0}_{\hspace{-0.6pt}\mathbf{G}}}
$$
by (\hspace{-0.4pt}d) of Definition~\ref{def.fol.graft} and because $\mathbf{G}$ is nonincreasing,
so (\ref{*4.5}) satisfies.
\smallskip

\textit{\hspace{-1pt}Case 2.} ${0}_{\hspace{-0.2pt}\mathbf{G}}\notin {B}_\mathbf{F}.$

Then by (\hspace{-0.4pt}d) of Lemma~\ref{about trees}, there is ${b}\in{B}_\mathbf{F}$
such that ${b}\:{\parallel}_{\!\mathbf{F}}\:{0}_{\hspace{-0.2pt}\mathbf{G}}.$
Since $\mathbf{F}$ is splittable, we have
$\mathbf{F}_{\hspace{-1.2pt}{b}}\cap\mathbf{F}_{\hspace{-1.2pt}{0}_{\hspace{-0.6pt}\mathbf{G}}}\!=\varnothing.$
Then since ${b}\in{B}_\mathbf{F},$ we have
$\fruit_{\hspace{0.4pt}\mathbf{F}}\hspace{-0.7pt}(\hspace{-0.5pt}{B}_\mathbf{F}\hspace{-0.5pt})
\subseteq\mathbf{F}_{\hspace{-1.2pt}{b}},$
so (\ref{*4.5}) satisfies again.

\medskip
Now suppose that $\mathbf{F}$ is splittable and has strict branches, and each $\mathbf{G}\in\varphi$ has bounded chains.
We must prove that $\mathbf{H}$ has strict branches; suppose it does not.
Since $\mathbf{F}$ is complete, we already know that $\mathbf{H}$ is also complete, so there is $\!{B}_\mathbf{H}\in\branches\mathbf{H}$ such that
$|\fruit_{\hspace{0.4pt}\mathbf{H}}\hspace{-0.7pt}(\hspace{-0.5pt}{B}_\mathbf{H}\hspace{-0.5pt})\,|>\!1.$
Let ${C}$ and $\!{B}_\mathbf{F}$ be as above.
It follows by (\ref{*143.1}) and (\ref{*143.2}) that $|\fruit_{\hspace{0.4pt}\mathbf{F}}\hspace{-0.7pt}(\hspace{-0.5pt}{C})\,|>\!1,$ and
$|\fruit_{\hspace{0.4pt}\mathbf{F}}\hspace{-0.7pt}(\hspace{-0.5pt}{B}_\mathbf{F}\hspace{-0.5pt})\,|=1$
since $\mathbf{F}$ has strict branches, so we have
$\fruit_{\hspace{0.4pt}\mathbf{F}}\hspace{-0.7pt}(\hspace{-0.5pt}{C})\neq
\fruit_{\hspace{0.4pt}\mathbf{F}}\hspace{-0.7pt}(\hspace{-0.5pt}{B}_\mathbf{F}\hspace{-0.5pt}).$
Then, using (a) of Lemma~\ref{about foliage trees}, we see that
${C}$ is not $\hspace{1pt}\mathsurround=0pt\mathbf{F}$-cofinal in $\!{B}_\mathbf{F}$
because $\varnothing\neq{C}\subseteq {B}_\mathbf{F}\subseteq\nodes\mathbf{F}.$
Further, since ${C}\subseteq {B}_\mathbf{F},$ $\!{B}_\mathbf{F}$ is a chain in $\mathbf{F},$
and ${C}$ is not $\mathsurround=0pt\mathbf{F}$-cofinal in $\!{B}_\mathbf{F},$
it is not hard to show that there is ${x}\in{B}_\mathbf{F}$ such that
${C}\subseteq {x}\hspace{-0.3pt}{\upspoon}_{\!\mathbf{F}}.$
Now we consider two cases\textup{:}
\smallskip

\textit{\hspace{-1pt}Case 1.} ${x}\in\supp(\mathbf{F},\varphi).$

Then $ {x}\hspace{-0.3pt}{\upspoon}_{\!\mathbf{F}}\cap\hspace{0.4pt}\supp(\mathbf{F},\varphi)
\subseteq {x}\hspace{-0.3pt}{\upspoon}_{\!\mathbf{H}}.$
We have ${C}\subseteq {x}\hspace{-0.3pt}{\upspoon}_{\!\mathbf{F}}$
and ${C}\subseteq\supp(\mathbf{F},\varphi),$ so ${C}\subseteq {x}\hspace{-0.3pt}{\upspoon}_{\!\mathbf{H}}.$
Then ${C}{\upfootline}_{\!\mathbf{H}}\subseteq {x}\hspace{-0.3pt}{\upspoon}_{\!\mathbf{H}},$
so $\!{B}_\mathbf{H}\subseteq {x}\hspace{-0.3pt}{\upspoon}_{\!\mathbf{H}}$
because ${C}$ is $\hspace{1pt}\mathsurround=0pt\mathbf{H}$-cofinal in $\!{B}_\mathbf{H},$
whence $\!{B}_\mathbf{H}\subset {x}\hspace{-0.5pt}{\upfilledspoon}_{\!\mathbf{H}}.$
This contradicts $\!{B}_\mathbf{H}\in\branches\mathbf{H},$
since ${x}\hspace{-0.5pt}{\upfilledspoon}_{\!\mathbf{H}}$ is a chain in $\mathbf{H}.$
\smallskip

\textit{\hspace{-1pt}Case 2.} ${x}\notin\supp(\mathbf{F},\varphi).$

We have ${x}\in\nodes\mathbf{F}\setminus\supp(\mathbf{F},\varphi),$ so by definition of $\supp(\mathbf{F},\varphi)$
there is $\mathbf{G}\in\varphi$ such that ${x}\in\expl(\mathbf{F},\mathbf{G}).$
Then (e) of Lemma~\ref{lem.graft} implies
$${x}\hspace{-0.3pt}{\upspoon}_{\!\mathbf{F}}\cap\hspace{0.4pt}\supp(\mathbf{F},\varphi)\subseteq
(\hspace{-1pt}{0}_{\hspace{-0.2pt}\mathbf{G}}\hspace{-1.3pt})\hspace{-1pt}{\upfilledspoon}_{\!\mathbf{F}},$$
so ${C}\subseteq (\hspace{-1pt}{0}_{\hspace{-0.2pt}\mathbf{G}}\hspace{-1.3pt})\hspace{-1pt}{\upfilledspoon}_{\!\mathbf{F}}.$
Since ${0}_{\hspace{-0.2pt}\mathbf{G}}\in\supp(\mathbf{F},\varphi),$ we have
$$(\hspace{-1pt}{0}_{\hspace{-0.2pt}\mathbf{G}}\hspace{-1.3pt})\hspace{-1pt}{\upfilledspoon}_{\!\mathbf{F}}\cap\hspace{0.4pt}\supp(\mathbf{F},\varphi)
\subseteq(\hspace{-1pt}{0}_{\hspace{-0.2pt}\mathbf{G}}\hspace{-1.3pt})\hspace{-1pt}{\upfilledspoon}_{\!\mathbf{H}},$$
whence ${C}\subseteq(\hspace{-1pt}{0}_{\hspace{-0.2pt}\mathbf{G}}\hspace{-1.3pt})\hspace{-1pt}{\upfilledspoon}_{\!\mathbf{H}}.$
This implies ${C}{\upfootline}_{\!\mathbf{H}}\subseteq (\hspace{-1pt}{0}_\mathcal{G}\hspace{-1.4pt})\hspace{-1pt}{\upfilledspoon}_{\!\mathbf{H}},$
so ${B}_\mathbf{H}\subseteq (\hspace{-1pt}{0}_{\hspace{-0.2pt}\mathbf{G}}\hspace{-1.3pt})\hspace{-1pt}{\upfilledspoon}_{\!\mathbf{H}}$
because ${C}$ is $\mathsurround=0pt\mathbf{H}$-cofinal in ${B}_\mathbf{H}.$
Then $\!{B}_\mathbf{H}=(\hspace{-1pt}{0}_{\hspace{-0.2pt}\mathbf{G}}\hspace{-1.3pt})\hspace{-1pt}{\upfilledspoon}_{\!\mathbf{H}},$
since $\!{B}_\mathbf{H}$ is a branch in $\mathbf{H}$ and
$(\hspace{-1pt}{0}_{\hspace{-0.2pt}\mathbf{G}}\hspace{-1.3pt})\hspace{-1pt}{\upfilledspoon}_{\!\mathbf{H}}$ is a chain in $\mathbf{H}.$
Thus we have ${0}_{\hspace{-0.2pt}\mathbf{G}}\in{B}_\mathbf{H}.$
Now (a) of Lemma~\ref{lem.reb.hybr} with (h) of Lemma~\ref{about trees} imply that
${B}_\mathbf{H}\cap\nodes\mathbf{G}={m}{\upfilledspoon}_{\!\mathbf{G}}$
for some  ${m}\in\maxel\mathbf{G}.$
Then, since $\maxel\mathbf{G}\subseteq\supp(\mathbf{F},\varphi),$
we have $m\in\hspace{1pt}\!{B}_\mathbf{H}\cap\hspace{0.4pt}\supp(\mathbf{F},\varphi)={C}.$
So ${m}\in{C}$ and ${m}>_{\mathbf{F}}{0}_{\hspace{-0.2pt}\mathbf{G}}$ by (\hspace{-0.4pt}d) of Definition~\ref{def.graft}.
This contradicts ${C}\subseteq(\hspace{-1pt}{0}_{\hspace{-0.2pt}\mathbf{G}}\hspace{-1.3pt})\hspace{-1pt}{\upfilledspoon}_{\!\mathbf{F}}.$

\medskip
(e) Suppose that $\mathbf{F}$ and each $\mathbf{G}\in\varphi$  are open in a space ${X}.$ Then $\mathbf{H}$ is open in the subspace ${X}\setminus\loss(\mathbf{F},\varphi)$ of ${X}$ by (\hspace{0.3pt}b\hspace{-0.5pt}) of Definition~\ref{def.f.hybr}.
\end{proof}

\section{Application of the foliage hybrid operation}
\label{sect.LPB.in.terms.of.f.trees}

We will apply the foliage hybrid operation to a $\hspace{1pt}\mathsurround=0pt\pi$-tree $\mathbf{F}$ of a space ${X}\hspace{-1pt}$ in such a way that the $\fhybr(\mathbf{F},\varphi)$ will be a $\hspace{1pt}\mathsurround=0pt\pi$-tree on a subspace ${Y}\hspace{-1pt}$ of ${X}\hspace{-1pt}.$ To carry out this construction we need to answer (that is, to find some sufficient conditions) the following questions:

   \begin{itemize}
   \item [\textup{(i)}]
      When the $\fhybr(\mathbf{F},\varphi)$ is a Baire foliage tree on ${Y}$?
   \item [\textup{(ii)}]
      When the $\fhybr(\mathbf{F},\varphi)$ grows into ${Y}$?
   \end{itemize}
The answer to question (i) is given the following lemma:

\begin{lem}\label{lem.when.f.hybr.is.pi.tree}
    Suppose that $\mathbf{F}$ is a Baire foliage tree on a space ${X}$ and $\varphi$ is a consistent family of foliage grafts for $\mathbf{F}$ such that every $\mathbf{G}$ in $\varphi$ is $\aleph_0$-branching, locally strict, open in ${X},$ has bounded chains, and has $\height\mathbf{G}\leqslant\omega.$
    Then the $\fhybr(\mathbf{F},\varphi)$ is a Baire foliage tree on ${X}\setminus\loss(\mathbf{F},\varphi).$
\end{lem}

\begin{proof}
   Let $\mathbf{H}\coloneq\fhybr(\mathbf{F},\varphi).$ It follows from (c)--(f) of Proposition~\ref{prop.hybr} and (i) of Lemma~\ref{about trees} that $\mathbf{H}$ is a foliage $\omega,\aleph_{{0}}$-tree and  ${0}_\mathbf{H}={0}_\mathbf{F},$ so $\mathbf{H}_{\hspace{0.5pt}{0}_{\hspace{-0.5pt}\mathbf{H}}}=
   \mathbf{F}_{\hspace{-1pt}{0}_{\hspace{-0.5pt}\mathbf{F}}}\setminus\loss(\mathbf{F},\varphi)=
   {X}\setminus\loss(\mathbf{F},\varphi).$ By (b) of Lemma~\ref{about foliage trees}, $\mathbf{F}$ is splittable, and then $\mathbf{H}$ is open in ${X}\setminus\loss(\mathbf{F},\varphi),$ locally strict, and has strict branches by (c)--(e) of Proposition~\ref{prop.f.hybr}.
\end{proof}

The answer to question (ii) is given in Lemma~\ref{l.grows.into.subspace}, and %
%: the $\fhybr(\mathbf{F},\varphi)$ grows into ${Y}\hspace{-1pt}$ whenever $\flesh(\fhybr(\mathbf{F},\varphi))={Y}\hspace{-1pt}$ and $\fhybr(\mathbf{F},\varphi)$ shoots into $\mathbf{F}$ (see Definition~\ref{def.sprt.through}).
this answer raises another question: When the $\fhybr(\mathbf{F},\varphi)$ shoots into $\mathbf{F}$?
%   \begin{itemize}
%   \item [\textup{(iii)}]
%      When the $\fhybr(\mathbf{F},\varphi)$ shoots into $\mathbf{F}$?
%   \end{itemize}
The answer to this question is given in Lemma~\ref{lem.when.f.hyb.sprts.thrgh}.%: the $\fhybr(\mathbf{F},\varphi)$ shoots into $\mathbf{F}$ whenever each $\mathbf{G}\in\varphi$ preserves shoots of $\mathbf{F}$ (see Definition~\ref{def.pres.shoots}).

\begin{deff}\label{def.sprt.through}
   Let $\mathbf{H}$ and $\mathbf{F}$ be foliage trees. Then

   \begin{itemize}
   \item [\ding{46}\ ]
      $\mathsurround=0pt
      \mathbf{H}\,$ \textbf{shoots into}  $\mathbf{F}
      \quad{\colon}{\longleftrightarrow}\quad
      \forall{p}\,{\in}\flesh\hspace{-1pt}\mathbf{H}\ \forall{y}\,{\in}\scope_\mathbf{F}\hspace{-0.5pt}(\hspace{-0.3pt}{p}\hspace{-0.5pt})\ \exists{x}\,{\in}\scope_\mathbf{H}\hspace{-0.5pt}(\hspace{-0.3pt}{p}\hspace{-0.5pt})
      \;\big[\hspace{-1pt}\shoot_\mathbf{H}\hspace{-0.5pt}(\hspace{-0.7pt}{x}\hspace{-1pt})
      \gg\shoot_\mathbf{F}\hspace{-0.5pt}({y})\big].$
   \end{itemize}
\end{deff}

\begin{lem}\label{l.grows.into.subspace}
   Suppose that a foliage tree $\mathbf{H}\hspace{-1pt}$ shoots into a foliage tree $\mathbf{F}\hspace{-1pt}$
   and $\mathbf{F}\hspace{-1pt}$ grows into a space ${X}.$
   Then
   $\mathbf{H}\hspace{-1pt}$ grows into the subspace ${X}\hspace{-1pt}\cap\flesh\hspace{-1pt}\mathbf{H}$ of $\hspace{1pt}{X}.$
\end{lem}

\begin{proof} Let  ${Y}\coloneq{X}\hspace{-1pt}\cap\flesh\hspace{-1pt}\mathbf{H},$ ${p}\in{Y},$ and ${U}\!\in\nbhds({p},{Y}\hspace{-0.5pt}).$
Then there is ${V}\!\in\nbhds({p},{X}\hspace{-0.5pt})$
such that ${U}={V}\cap{Y},$
and there is ${y}\in\scope_\mathbf{F}\hspace{-0.5pt}(\hspace{-0.3pt}{p}\hspace{-0.5pt})$
such that $\shoot_\mathbf{F}\hspace{-0.5pt}({y})\gg\{{V}\}$
because $\mathbf{F}$ grows into ${X}.$
Since $\mathbf{H}$ shoots into $\mathbf{F},$
there is ${x}\in\scope_\mathbf{H}\hspace{-0.5pt}(\hspace{-0.3pt}{p}\hspace{-0.5pt})$
with the property $\shoot_\mathbf{H}\hspace{-0.5pt}(\hspace{-0.7pt}{x}\hspace{-1pt})
\gg\shoot_\mathbf{F}\hspace{-0.5pt}({y}).$
It follows that there is ${G}\in\shoot_\mathbf{H}\hspace{-0.5pt}(\hspace{-0.7pt}{x}\hspace{-1pt})\setminus\{\varnothing\}$ such that ${G}\subseteq{V}.$
Since ${G}\subseteq\flesh\mathbf{H},$
then ${G}\subseteq{V}\hspace{-1pt}\cap\flesh\mathbf{H}
\subseteq{X}\hspace{-1pt}\cap\flesh\mathbf{H}={Y},$
so ${G}\subseteq{V}\cap{Y}\!={U}.$
%This implies ${U}\in\gamma$
%because ${G}\in\gamma$ and ${U}\subseteq\flesh\mathbf{H}.$
Therefore we have found
${x}\in\scope_\mathbf{H}\hspace{-0.5pt}(\hspace{-0.3pt}{p}\hspace{-0.5pt})$
such that $\shoot_{\mathbf{H}}\hspace{-0.5pt}(\hspace{-0.7pt}{x}\hspace{-1pt})\gg\{{U}\}.$
\end{proof}

\begin{deff}\label{def.pres.shoots}
   Let $\mathbf{F}$ be a nonincreasing foliage tree and let $\mathbf{G}$ be a foliage graft for $\mathbf{F}.$
   Then
   $\mathbf{G}$ \textbf{preserves shoots} of $\mathbf{F}$ iff

   \begin{itemize}
   \item[\ding{226}\ ]
      for each ${p}\in\flesh\hspace{-1pt}\mathbf{G}$ and
      for each ${y}\in\scope_\mathbf{F}\hspace{-0.5pt}(\hspace{-0.3pt}{p}\hspace{-0.5pt})
      \cap\big(\{{0}_{\hspace{-0.2pt}\mathbf{G}}\}\cup\hspace{0.4pt}\expl(\mathbf{F},\mathbf{G})\big)$
   \item[\ding{226}\ ]
      there is  ${x}\in\scope_\mathbf{G}\hspace{-0.5pt}(\hspace{-0.3pt}{p}\hspace{-0.5pt})
      \cap\big(\{{0}_{\hspace{-0.2pt}\mathbf{G}}\}\cup\hspace{0.4pt}\impl\mathbf{G}\big)$
      such that
      \begin{itemize}
      \item [\ding{51}\ \ ]
         $\mathsurround=0pt
         \shoot_{\hspace{-0.2pt}\mathbf{G}\hspace{-0.8pt}}(\hspace{-0.7pt}{x}\hspace{-1pt})
         \vphantom{\bar{\langle\rangle}}\gg\shoot_\mathbf{F}\hspace{-0.5pt}({y}).$
      \end{itemize}
   \end{itemize}
\end{deff}

\begin{lem}\label{lem.when.f.hyb.sprts.thrgh}
   Suppose that

   \begin{itemize}
   \item[\ding{226}\ ]
      $\mathsurround=0pt
      \mathbf{F}$ is a nonincreasing foliage tree, %with nonempty leaves,
   \item[\ding{226}\ ]
      $\mathsurround=0pt
      \varphi$ is a consistent family of foliage grafts for $\mathbf{F},$
   \item[\ding{226}\ ]
      the foliage hybrid of $\,\mathbf{F}$ and $\varphi$ has nonempty leaves, and
   \item[\ding{226}\ ]
      each $\mathbf{G}\in\varphi$ preserves shoots of $\,\mathbf{F}.$
   \end{itemize}
   Then the foliage hybrid of $\,\mathbf{F}$ and $\varphi$ shoots into $\mathbf{F}.$
\end{lem}

    \begin{proof} Let $\mathbf{H}\coloneq\fhybr(\mathbf{F},\varphi),$
    ${p}\,{\in}\flesh\hspace{-1pt}\mathbf{H},$ and ${y}\,{\in}\scope_\mathbf{F}\hspace{-0.5pt}(\hspace{-0.3pt}{p}\hspace{-0.5pt}).$
    We consider two cases\textup{:}
    \smallskip

    \textit{\hspace{-1pt}Case 1.}
    ${y}\in\supp(\mathbf{F},\varphi)\,{\setminus}\,\{{0}_{\hspace{-0.2pt}\mathbf{G}}:\mathbf{G}\in\varphi\}.$

    By (a) of Proposition~\ref{prop.hybr} we have
    $\sons_{\hspace{0.2pt}\mathbf{H}\hspace{-0.7pt}}({y})=\sons_{\hspace{0.5pt}\mathbf{F}\hspace{-0.5pt}}({y}),$ so %$\sons_{\hspace{0.2pt}\mathbf{H}\hspace{-0.7pt}}({y})$ is infinite and
    $$\sons_{\hspace{0.2pt}\mathbf{H}\hspace{-0.7pt}}({y})\,\subseteq\,\nodes\mathbf{F}\cap\nodes\mathbf{H}
    \,=\,\supp(\mathbf{F},\varphi)$$
    by (\hspace{0.3pt}b\hspace{-0.5pt}) of Lemma~\ref{lem.hybr}. Then by (\hspace{0.3pt}b\hspace{-0.5pt}) of Definition~\ref{def.f.hybr} we have
    $$\mathbf{H}_{{y}}=\,\mathbf{F}_{\hspace{-1.2pt}{y}}\,{\setminus}\,\loss(\mathbf{F},\varphi)
    \quad\textsf{ and }\quad
    \forall{s}\in\hspace{0.4pt}\sons_{\hspace{0.2pt}\mathbf{H}\hspace{-0.7pt}}({y})
    \,\big[\mathbf{H}_{{s}}=\,\mathbf{F}_{\hspace{-1.2pt}{s}}{\setminus}\loss(\mathbf{F},\varphi)\big].
    $$
    Further, ${p}\in\mathbf{F}_{\hspace{-1.2pt}{y}}$ and ${p}\,{\in}\flesh\hspace{-1pt}\mathbf{H},$ so ${p}\notin\loss(\mathbf{F},\varphi),$ whence ${p}\in\mathbf{H}_{{y}},$ that is,
    ${y}\in\scope_\mathbf{H}\hspace{-0.5pt}(\hspace{-0.3pt}{p}\hspace{-0.5pt}).$
    Now, for ${x}\coloneq{y}$ and for each ${s}\in\hspace{0.4pt}\sons_{\hspace{0.2pt}\mathbf{H}\hspace{-0.7pt}}(\hspace{-0.7pt}{x}\hspace{-1pt})
    =\sons_{\hspace{0.2pt}\mathbf{H}\hspace{-0.7pt}}({y})
    =\sons_{\hspace{0.5pt}\mathbf{F}\hspace{-0.5pt}}({y}),$ we have $\varnothing\neq\mathbf{H}_{{s}}\subseteq\mathbf{F}_{\hspace{-1.2pt}{s}}.$
    This implies
    $\shoot_\mathbf{H}(\hspace{-0.7pt}{x}\hspace{-1pt})\gg\shoot_\mathbf{F}\hspace{-0.5pt}({y}).$
    \smallskip

    \textit{\hspace{-1pt}Case 2.} $\exists\mathbf{G}\,{\in}\,\varphi
    \;\big[\,{y}\in\{{0}_{\hspace{-0.2pt}\mathbf{G}}\}\cup\hspace{0.4pt}\expl(\mathbf{F},\mathbf{G})\big].$

    The foliage tree $\mathbf{F}$ is nonincreasing, ${p}\in\mathbf{F}_{\hspace{-1.2pt}{y}},$ and ${y}\geqslant_{\mathbf{F}}{0}_{\hspace{-0.2pt}\mathbf{G}},$
    so ${p}\in\mathbf{F}_{\hspace{-1.2pt}{{0}_{\hspace{-0.5pt}\mathbf{G}}}}.$ We have
    ${p}\,{\in}\flesh\hspace{-1pt}\mathbf{H},$ so  ${p}\notin\loss(\mathbf{F},\varphi),$
    hence ${p}\in\mathbf{F}_{\hspace{-1.2pt}{{0}_{\hspace{-0.5pt}\mathbf{G}}}}$ implies ${p}\in\mathbf{G}_{\hspace{-0.3pt}{{0}_{\hspace{-0.5pt}\mathbf{G}}}}.$
    Then ${p}\in\flesh\hspace{-1pt}\mathbf{G}$ and
    $$
    {y}\in\scope_\mathbf{F}\hspace{-0.5pt}(\hspace{-0.3pt}{p}\hspace{-0.5pt})\cap
    \big(\{{0}_{\hspace{-0.2pt}\mathbf{G}}\}\cup\hspace{0.4pt}\expl(\mathbf{F},\mathbf{G})\big),
    $$
    so, since $\mathbf{G}$ preserves shoots $\mathbf{F},$ there is
    $$
    {x}\in\scope_\mathbf{G}\hspace{-0.5pt}(\hspace{-0.3pt}{p}\hspace{-0.5pt})\cap
    \big(\{{0}_{\hspace{-0.2pt}\mathbf{G}}\}\cup\hspace{0.4pt}\impl\mathbf{G}\big)
    $$ such that
    $\shoot_{\hspace{-0.2pt}\mathbf{G}\hspace{-0.8pt}}(\hspace{-0.7pt}{x}\hspace{-1pt})
    \gg\shoot_\mathbf{F}\hspace{-0.5pt}({y}).    $
    Again, by (a) of Proposition~\ref{prop.hybr} we have
    $\sons_{\hspace{0.2pt}\mathbf{H}\hspace{-0.7pt}}(\hspace{-0.7pt}{x}\hspace{-1pt})=
    \sons_{\hspace{0.2pt}\mathbf{G}\hspace{-1.2pt}}(\hspace{-0.7pt}{x}\hspace{-1pt}),$ so
    $\sons_{\hspace{0.2pt}\mathbf{H}\hspace{-0.7pt}}(\hspace{-0.7pt}{x}\hspace{-1pt})\subseteq
    \nodes\mathbf{G}.$
    Then by (a) of Lemma~\ref{lem.reb} we have
    $$\mathbf{H}_{{x}}=\mathbf{G}_{\hspace{-0.3pt}{x}}\,{\setminus}\,\loss(\mathbf{F},\varphi)
    \quad\textsf{ and }\quad
    \forall{s}\in\hspace{0.4pt}\sons_{\hspace{0.2pt}\mathbf{H}\hspace{-0.7pt}}(\hspace{-0.7pt}{x}\hspace{-1pt})\,\big[\mathbf{H}_{{s}}=\mathbf{G}_{\hspace{-0.3pt}{s}}{\setminus}\loss(\mathbf{F},\varphi)\big].
    $$
    We have ${p}\in\mathbf{G}_{\hspace{-0.3pt}{x}}$ and ${p}\notin\loss(\mathbf{F},\varphi),$ so ${p}\in\mathbf{H}_{{x}},$ that is,
    ${x}\in\scope_\mathbf{H}\hspace{-0.5pt}(\hspace{-0.3pt}{p}\hspace{-0.5pt}).$
    Now, for each ${s}\in\hspace{0.4pt}\sons_{\hspace{0.2pt}\mathbf{H}\hspace{-0.7pt}}(\hspace{-0.7pt}{x}\hspace{-1pt})=\sons_{\hspace{0.2pt}\mathbf{G}\hspace{-1.2pt}}(\hspace{-0.7pt}{x}\hspace{-1pt}),$ we have $\varnothing\neq\mathbf{H}_{{s}}\subseteq\mathbf{G}_{\hspace{-0.3pt}{s}}.$
    This implies
    $\shoot_\mathbf{H}(\hspace{-0.7pt}{x}\hspace{-1pt})\gg
    \shoot_{\hspace{-0.2pt}\mathbf{G}\hspace{-0.8pt}}(\hspace{-0.7pt}{x}\hspace{-1pt}),$
    so $\shoot_\mathbf{H}(\hspace{-0.7pt}{x}\hspace{-1pt})\gg
    \shoot_\mathbf{F}\hspace{-0.5pt}({y})$ because $\gg$ is transitive.
    \end{proof}

\section{Main construction}

In this section we prove Theorem~\ref{main.theorem}, which can be viewed as
the main technical result of this paper. This theorem is a statement about the Baire space $\mathcal{N}$ and the standard foliage tree of ${}^\omega\hspace{-1pt}\omega,$ which we denote by $\mathbf{S}$ --- see Notation~\ref{def.stdrt.f.tr.of.B.sp}. The connection between $\mathcal{N}$ with $\mathbf{S}$ on the one hand and a space ${X}$ with a $\hspace{1pt}\mathsurround=0pt\pi$-tree on the other hand is explained by Lemma~\ref{lem.pi.and.B.f.trees.vs.S}.

\begin{notation}\label{not.pi.dense}
   Let ${A}\subseteq{}^\omega\hspace{-1pt}\omega$ and ${x}\in{}^{<\hspace{0.2pt}\omega}\hspace{-1pt}\omega.$ Recall that $\mathbf{S}_{x}=\{{p}\,{\in}\,{}^\omega\hspace{-1pt}\omega:{x}\subseteq{p}\}.$
   Then

   \begin{itemize}
   \item [\ding{46}\ ]
      ${A}$ is $\hspace{1pt}\mathsurround=0pt\pi$-\textbf{dense} at ${x}
      \quad{\colon}{\longleftrightarrow}\quad
      \forall{y}\in{}^{<\hspace{0.2pt}\omega}\hspace{-1pt}\omega\:
      \big[\,{y}\supseteq{x}
      \ \to\ \big|\{{n}\in\omega:\mathbf{S}_{{y}\hspace{0.7pt}\hspace{0.6pt}
      \hat{}\hspace{1pt}{\scriptscriptstyle\langle}\hspace{-0.4pt}{n}\hspace{-0.4pt}{\scriptscriptstyle\rangle}}
      \subseteq{A}\}\big|
      =\aleph_{0}\,\big];$
   \item [\ding{46}\ ]
      ${A}$ is $\hspace{1pt}\mathsurround=0pt\pi$-\textbf{dense} in the Baire space
      $\quad{\colon}{\longleftrightarrow}\quad
      \forall{y}\in{}^{<\hspace{0.2pt}\omega}\hspace{-1pt}\omega\:
      \big[\,\big|\{{n}\in\omega:\mathbf{S}_{{y}\hspace{0.7pt}\hspace{0.6pt}\hat{}
      \hspace{1pt}{\scriptscriptstyle\langle}\hspace{-0.4pt}{n}\hspace{-0.4pt}{\scriptscriptstyle\rangle}}
      \subseteq{A}\}\big|
      =\aleph_{0}\,\big].$
   \end{itemize}
\end{notation}

\begin{rem}\label{rem.pi.dense}
   \begin{itemize}
   \item [\textup{(a)}]
      If ${K}$ is a compact subset of $\mathcal{N},$
      then ${}^\omega\hspace{-1pt}\omega\,{\setminus}\,{K}$ is an open $\hspace{1pt}\mathsurround=0pt\pi$-dense subset of $\mathcal{N}.$
   \item [\textup{(\hspace{0.3pt}b\hspace{-1pt})}]
      If a set ${D}$ is $\hspace{1pt}\mathsurround=0pt\pi$-dense in $\mathcal{N},$
      then ${D}$ is dense in $\mathcal{N}.$\hfill$\Box$
   \end{itemize}
\end{rem}

\begin{teo}\label{main.theorem}
   Suppose that $\hspace{1pt}{Y}=\bigcap_{{n}\in\hspace{0.3pt}\omega}{U}_{n},$ where each $\hspace{1pt}{U}_{n}$ is an open $\hspace{1pt}\mathsurround=0pt\pi$-dense subset of the Baire space.
   Then there is a Baire foliage tree on $\hspace{1pt}{Y}\hspace{-1pt}$ that shoots into
   the standard foliage tree of $\hspace{1pt}{}^\omega\hspace{-1pt}\omega$ \textup{(see Definitions~\ref{def.B.f.tree}, \ref{def.stdrt.f.tr.of.B.sp}, \ref{def.sprt.through}, and \ref{not.pi.dense})}.
\end{teo}

\begin{que} Does Theorem~\ref{main.theorem} remains true if we replace ``$\mathsurround=0pt\pi$-dense'' by ``dense''\,?
\end{que}

We will build this Baire foliage tree on ${Y}\hspace{-1pt}$ which shoots into $\mathbf{S}$ by applying the foliage hybrid operation to $\mathbf{S}$ and $\varphi,$ where $\varphi$ is a consistent family of foliage grafts for $\mathbf{S}.$ We construct the family $\varphi$ in the proof of \hspace{1pt}Theorem~\ref{main.theorem}, see below. The construction of a single foliage graft $\mathbf{G}$ (that will be a member of  $\varphi$) is described in the following lemma:

\begin{lem}\label{Construction of an LPB-fine foliage graft}
   Suppose that $\mathbf{v}\in{}^{<\hspace{0.2pt}\omega}\hspace{-1pt}\omega$ and ${O}\subset\mathbf{S}_\mathbf{v}$
   is open in the Baire space and is $\hspace{1pt}\mathsurround=0pt\pi$-dense at $\mathbf{v}.$
   Then there is a foliage tree $\mathbf{G}$ such that

   \begin{itemize}
   \item [\textup{(a1)}\ ]
      $\mathsurround=0pt
      {0}_{\hspace{-0.2pt}\mathbf{G}}=\mathbf{v},$
   \item [\textup{(a2)}\ ]
      $\mathsurround=0pt
      \height\hspace{-0.4pt}\mathbf{G}\leqslant\omega,$
   \item [\textup{(a3)}\ ]
      $\mathsurround=0pt\mathbf{G}$ is
      $\mathsurround=0pt\aleph_0$-branching\textup{,}
   \item [\textup{(a4)}\ ]
      $\mathsurround=0pt\mathbf{G}$ has bounded chains\textup{,}
   \item [\textup{(a5)}\ ]
      $\mathsurround=0pt\mathbf{G}$ is
      locally strict\textup{,}
   \item [\textup{(a6)}\ ]
      $\mathsurround=0pt\mathbf{G}$ is
      open in the Baire space,
    \item [\textup{(a7)}\ ]
      $\mathsurround=0pt\mathbf{G}$ is a
      foliage graft for $\mathbf{S},$
    \item [\textup{(a8)}\ ]
      $\mathsurround=0pt\mathbf{G}$
      preserves shoots of $\mathbf{S},$
   \item [\textup{(a9)}\ ]
      $\mathsurround=0pt
      \impl\mathbf{G}\neq\varnothing,$
    \item [\textup{(a10)}\ ]
      $\mathsurround=0pt
      \cut(\mathbf{S},\mathbf{G})=\mathbf{S}_\mathbf{v}{\setminus}\,{O},$ and
   \item [\textup{(a11)}\ ]
      $\mathsurround=0pt
      {O}\equiv\bigsqcup_{{z}\in\maxel\mathbf{G}}\mathbf{S}_{z}.$
   \end{itemize}
\end{lem}

In the proof of Lemma~\ref{Construction of an LPB-fine foliage graft} (see below)  we verify clause (a8), which says that $\mathbf{G}$ preserves shoots of $\mathbf{S}.$  We do this by using the following lemma:

\begin{lem}\label{about shoots}
   Suppose that $\mathbf{A},\mathbf{B}$ are foliage trees with nonempty leaves, ${x}\in\nodes\mathbf{A},$ and
   ${y}\in\nodes\mathbf{B}.$
   Assume that $\big|\hspace{-1pt}\sons_{\hspace{0.2pt}\mathbf{A}\hspace{-1.2pt}}
   (\hspace{-0.7pt}{x}\hspace{-1pt})\hspace{0.7pt}\big|\geqslant\aleph_0$ and that there is finite ${F}$ such that
   $$\forall{s}\,{\in}\,\sons_{\hspace{0.2pt}\mathbf{A}\hspace{-1.2pt}}
   (\hspace{-0.7pt}{x}\hspace{-1pt})\,{\setminus}\,{F}
   \;\big[\,{s}\in\hspace{0.4pt}\sons_\mathbf{B}({y})\enskip\mathsf{and}\enskip
   \mathbf{A}_{\hspace{-0.3pt}{s}}\subseteq\mathbf{B}_{\hspace{0.5pt}{s}}\,\big].
   $$
   Then $\,\shoot_{\hspace{-0.2pt}\mathbf{A}\hspace{-0.8pt}}(\hspace{-0.7pt}{x}\hspace{-1pt})\gg\shoot_\mathbf{B}({y}).$
   \hfill$\Box$
\end{lem}

\begin{proof}[\textbf{\textup{Proof of Lemma~\ref{Construction of an LPB-fine foliage graft}}}]
Let \vspace{-1.5ex}
$${\Omega}\coloneq\{{z}\in\mathbf{v}\hspace{-0.5pt}{\downfilledspoon}_\mathbf{S}:
\mathbf{S}_{z}\subseteq{O}\},\quad
{\Delta}\coloneq\mathbf{v}\hspace{-0.5pt}{\downfilledspoon}_\mathbf{S}\hspace{-1pt}\setminus{\Omega},
\quad\text{and}\quad
\mathsf{MAX}\coloneq\minel\,({\Omega},<_\mathbf{S}).\vspace{-1ex}$$
Then we have

   \begin{itemize}
   \item [\textup{(b1)}\ ]
      $\mathsurround=0pt
      \mathbf{v}\in{\Delta}\quad\mathsf{and}\quad|{\Delta}|=\aleph_0;$
   \item [\textup{(b2)}\ ]
      $\mathsurround=0pt
      \Delta=(\Delta{\upfootline}_{\hspace{-1pt}\mathbf{S}})
      \cap(\mathbf{v}\hspace{-0.5pt}{\downfilledspoon}_\mathbf{S});$
   \item [\textup{(b3)}\ ]
      $\mathsurround=0pt
      \mathsf{MAX}\,$ is an antichain in $\mathbf{S};$
   \item [\textup{(b4)}\ ]
      $\mathsurround=0pt
      \mathsf{MAX}{\downfootline}_\mathbf{S}={\Omega};$
   \item [\textup{(b5)}\ ]
       $\mathsurround=0pt{O}\,\equiv\,\bigsqcup_{{z}\in\hspace{0.2pt}\mathsf{MAX}}\mathbf{S}_{z}.$
   \end{itemize}

For each ${x}\in{\Delta},$ define\vspace{-1.5ex}
$${\Delta}_{x}\coloneq\,{\Delta}\cap({x}\hspace{-0.7pt}{\downfilledspoon}_\mathbf{S})
\quad\text{and}\quad
{\Omega}_{{x}}\coloneq\,\sons_{\hspace{0.5pt}\mathbf{S}\hspace{-0.5pt}}(\hspace{-0.7pt}{x}\hspace{-1pt})\cap{\Omega}.\vspace{-2ex}$$
Then

   \begin{itemize}
   \item [\textup{(c1)}\ ]
      $\mathsurround=0pt\hspace{-0.7pt}\forall{x}\in{\Delta}
      \:[\, {x}\in{\Delta}_{x}\enskip\mathsf{and}\enskip
      |{\Delta}_{x}|=\aleph_0\,];$
   \item [\textup{(c2)}\ ]
      $\mathsurround=0pt\hspace{-0.7pt}\forall{x}\in{\Delta}
      \:[\,{\Omega}_{{x}}\subseteq\,\mathsf{MAX}
      \enskip\mathsf{and}\enskip|\,{\Omega}_{{x}}|=\aleph_0\,];$
   \item [\textup{(c3)}\ ]
      $\mathsurround=0pt\mathsf{MAX}\,\equiv\hspace{1pt}
      \bigsqcup_{{x}\in{\Delta}}{\Omega}_{{x}}.$
   \end{itemize}

Now for each ${x}\in{\Delta}$ and all ${d}\in{\Delta}_{x},$ we can find infinite sets ${\Omega}_{{x},{d}}\subseteq{\Omega}_{{x}}$
in such a way that
\begin{equation}\label{*17.2}
\textstyle\forall{x}\,{\in}\,{\Delta}\;\big[\:{\Omega}_{{x}}\equiv\,
\bigsqcup_{{d}\in{\Delta}_{x}}\!{\Omega}_{{x},{d}}\,\big].\vspace{-2ex}
\end{equation}
Put
$$\textsf{IMP}\coloneq\Big\{\,{\mathsf{node}}_{\hspace{0.1pt}{x}}^{\hspace{0.3pt}{l}}\,:\,{x}\in{\Delta}
\ \mathsf{and}\ {l}\,{\in}\,\big\{0,\ldots,{l}(\hspace{-0.7pt}{x}\hspace{-1pt})\big\}\Big\}
$$
where\vspace{-1ex}
$${l}(\hspace{-0.7pt}{x}\hspace{-1pt})\coloneq\length{x}-\length\mathbf{v}$$
and ${\mathsf{node}}_{\hspace{0.1pt}{x}}^{\hspace{0.3pt}{l}}$ are different new nodes for the skeleton of the foliage tree $\mathsurround=0pt\mathbf{G}$ such that $\mathsf{IMP}\,\cap\,\nodes\mathbf{S}=\varnothing.$
Put
$$\mathsf{NOD}\coloneq\{\hspace{-0.5pt}\mathbf{v}\hspace{-0.7pt}\}\cup\hspace{0.4pt}\mathsf{MAX}\cup\mathsf{IMP}$$
(we intend to have $\nodes\mathbf{G}=\mathsf{NOD},$ ${0}_{\hspace{-0.2pt}\mathbf{G}}=\mathbf{v},$ $\maxel\mathbf{G}=\mathsf{MAX},$ and
$\impl\mathbf{G}=\mathsf{IMP}$).

For ${x}\in{}^{<\hspace{0.2pt}\omega}\hspace{-1pt}\omega$ and ${l}\in\{0,\ldots,\length{x}\},$ define
$${x}_{-\,{l}}\;\coloneq\;{x}\:{\upharpoonright}\;\big(\!(\length{x})\,{-}\,{l}\big)$$
--- that is, ${x}_{-\,{l}}=\langle{x}_0,\ldots,{x}_{(\length{x})-\,{l}-1}\rangle\in\omega^{(\length{x})-\,{l}},$
${x}_{{-}\,0}={x},$ and if ${x}\in\mathbf{v}\hspace{-0.5pt}{\downfilledspoon}_\mathbf{S},$ then ${x}_{-\,{l}(\hspace{-0.7pt}{x}\hspace{-1pt})}=\mathbf{v}.$
Using (b2) we have

   \begin{itemize}
   \item [\textup{(d1)}\ ]
      $\mathsurround=0pt
      \forall{x}\,{\in}\,{\Delta}\,\forall{l}\,{\in}\,\big\{0,\ldots,{l}(\hspace{-0.7pt}{x}\hspace{-1pt})\big\}
      \,\big[\,{x}_{-\,{l}}\in{\Delta}\enskip\mathsf{and}\enskip{x}\in{\Delta}_{{x}_{{-}{l}}}\,\big];$
   \item [\textup{(d2)}\ ]
      $\mathsurround=0pt
      \Big\{({x}_{-\,{l}},{x}):{x}\,{\in}\,\Delta\ \mathsf{and}\ {l}\,{\in}\,\big\{0,\ldots,{l}(\hspace{-0.7pt}{x}\hspace{-1pt})\big\}\Big\}=
      \big\{({z},{d}):{z}\,{\in}\,\Delta\ \mathsf{and}\ {d}\,{\in}\,\Delta_{z}\big\}.$
   \end{itemize}

Now we build a tree $(\mathsf{NOD},<),$ which will be a skeleton for the foliage tree $\mathbf{G}.$
First we define a relation $\lessdot$ on the set $\mathsf{NOD}$ as the relation that satisfies exactly the following:

   \begin{itemize}
   \item[\ding{226}\ ]
      for each ${x}\in{\Delta},$

      $\mathsurround=0pt\mathbf{v}\,\lessdot\,
      {\mathsf{node}}_{\hspace{0.1pt}{x}}^{\hspace{0.3pt}{l}(\hspace{-0.7pt}{x}\hspace{-1pt})}\,\lessdot\,
      {\mathsf{node}}_{\hspace{0.1pt}{x}}^{\hspace{0.3pt}{l}(\hspace{-0.7pt}{x}\hspace{-1pt}){-}1}
      \,\lessdot\,\ldots
      \,\lessdot\,{\mathsf{node}}_{\hspace{0.1pt}{x}}^{{1}}\,\lessdot\,{\mathsf{node}}_{\hspace{0.1pt}{x}}^{{0}};$
   \item[\ding{226}\ ]
      for each ${x}\in{\Delta}$ and each ${l}\in\big\{0,\ldots,{l}(\hspace{-0.7pt}{x}\hspace{-1pt})\big\},$

      $\mathsurround=0pt{\mathsf{node}}_{\hspace{0.1pt}{x}}^{\hspace{0.3pt}{l}}\lessdot\,{z}\quad$ for all ${z}\in{\Omega}_{{x}_{-{l}},\,{x}}\,.$
   \end{itemize}
Note that the last clause is correct by (d1).
Then let relation $<$ be the transitive closure of relation~$\lessdot.$
That is, for each ${a},{b}\in\mathsf{NOD},$
$${a}<{b}      \quad{\colon}{\longleftrightarrow}\quad
\exists\,{n}\,{\in}\,\omega\ \exists\,{z}_{0},\ldots,{z}_{{n}{+}{1}}\,{\in}\:\mathsf{NOD}
\ [\,{a}={z}_{0}\lessdot\,{z}_{1}\lessdot\,\ldots\,\lessdot\,{z}_{{n}{+}{1}}={b}\,].$$

Let $\mathcal{T}\coloneq(\mathsf{NOD},<).$ Then it is not hard to show the following:

   \begin{itemize}
    \item [\textup{(e1)}\ ]
      $\mathsurround=0pt\sons_{\hspace{0.3pt}\mathcal{T}}({\mathsf{node}}_{\hspace{0.1pt}{x}}^{{0}})=
      {\Omega}_{{x}_{-{0}},\,{x}}\quad$
      for all ${x}\in{\Delta};$

      $\mathsurround=0pt\sons_{\hspace{0.3pt}\mathcal{T}}({\mathsf{node}}_{\hspace{0.1pt}{x}}^{\hspace{0.3pt}{l}})=
      {\Omega}_{{x}_{-{l}},\,{x}}\cup\big\{{\mathsf{node}}_{\hspace{0.1pt}{x}}^{\hspace{0.3pt}{l}{-}1}\big\}\quad$
      for all ${x}\in{\Delta}$ and ${l}\in\big\{1,\ldots,{l}(\hspace{-0.7pt}{x}\hspace{-1pt})\big\};$

      $\mathsurround=0pt\sons_{\hspace{0.3pt}\mathcal{T}}(\hspace{-0.5pt}\mathbf{v}\hspace{-0.7pt})\hspace{1pt}=
      \big\{{\mathsf{node}}_{\hspace{0.1pt}{x}}^{\hspace{0.3pt}{l}(\hspace{-0.7pt}{x}\hspace{-1pt})}:
      {x}\in{\Delta}\big\}.$
   \item [\textup{(e2)}\ ]
      $\mathsurround=0pt\hspace{-0.7pt}\forall{x}\in{\Delta}
      \;[\,\mathbf{v}\,\sqsubset_\mathcal{T}\,{\mathsf{node}}_{\hspace{0.1pt}{x}}^{\hspace{0.3pt}{l}
      (\hspace{-0.7pt}{x}\hspace{-1pt})}
      \,\sqsubset_\mathcal{T}\,{\mathsf{node}}_{\hspace{0.1pt}{x}}^{\hspace{0.3pt}{l}
      (\hspace{-0.7pt}{x}\hspace{-1pt}){-}1}
      \,\sqsubset_\mathcal{T}\,\ldots\,\sqsubset_\mathcal{T}\,
      {\mathsf{node}}_{\hspace{0.1pt}{x}}^{1}
      \,\sqsubset_\mathcal{T}\,{\mathsf{node}}_{\hspace{0.1pt}{x}}^{{0}}\,];$

      in particular, $\ \mathbf{v}\,\sqsubset_\mathcal{T}\,{\mathsf{node}}_{\hspace{0.5pt}\mathbf{v}}^{{l
      (\hspace{-0.5pt}\mathbf{v}\hspace{-0.7pt})}}
      ={\mathsf{node}}_{\hspace{0.5pt}\mathbf{v}}^{{0}}\,.$
   \item [\textup{(e3)}\ ]
      $\mathsurround=0pt\maxel\mathcal{T}=
      \mathsf{MAX}.$

      Indeed, using (d2), (c3), and (\ref{*17.2}), we get
      $$\maxel\mathcal{T}\,=\,\bigcup\Big\{\,{\Omega}_{{x}_{-{l}},\,{x}}:\,
      {x}\in{\Delta}\ \mathsf{and}\
      {l}\in\big\{0,\ldots,{l}(\hspace{-0.7pt}{x}\hspace{-1pt})\big\}\Big\}\hspace{1pt}=\,
      \bigcup\big\{\,{\Omega}_{{z},{d}}\,:\,
      {z}\in{\Delta}\ \mathsf{and}\
      {d}\in\Delta_{z}\big\}
      \,=\,\mathsf{MAX}.$$
   \item [\textup{(e4)}\ ]
      $\mathsurround=0pt\mathcal{T}\,$ is an $\mathsurround=0pt\aleph_0$-branching tree with the least node and  ${0}_\mathcal{T}=\mathbf{v}.$
   \item [\textup{(e5)}\ ]
      $\mathsurround=0pt\mathcal{T}\,$ has bounded chains and $\height\hspace{-0.4pt}\mathcal{T}\leqslant\omega.$

      To prove (e5) it is enough to show that each chain in $\mathcal{T}\hspace{-1pt}$ is finite.
      If ${C}$ is a chain in $\mathcal{T},$ then by (c) of Lemma~\ref{about trees}, there is ${B}\in\branches\mathcal{T}$ such that ${C}\subseteq{B},$ and it follows using (e) of
      Lemma~\ref{about trees} that there exists some ${s}$ in ${B}\cap\hspace{0.4pt}\sons_{\hspace{0.3pt}\mathcal{T}}({0}_\mathcal{T}).$ Then
      ${s}={\mathsf{node}}_{\hspace{0.1pt}{x}}^{\hspace{0.3pt}{l}(\hspace{-0.7pt}{x}\hspace{-1pt})}$ for some ${x}\in{\Delta},$ so $|{B}\hspace{0.5pt}|\leqslant{l}(\hspace{-0.7pt}{x}\hspace{-1pt})+3.$
   \item [\textup{(e6)}\ ]
      $\mathsurround=0pt\mathcal{T}\,$ is a graft for $\mathbf{S}$ and
      $\impl\mathcal{T}=\mathsf{IMP}.$
   \item [\textup{(e7)}\ ]
      $\mathsurround=0pt\expl(\mathbf{S},\mathcal{T})=
      \Delta\,{\setminus}\,\{\hspace{-0.5pt}\mathbf{v}\hspace{-0.7pt}\}.$

      Indeed, using (b4), we have\vspace{-0.7ex}
      $$\expl(\mathbf{S},\mathcal{T})
      \,\coloneq\,
      (\hspace{-1pt}{0}_\mathcal{T}\hspace{-1.4pt})\hspace{-0.8pt}{\downspoon}_{\mathbf{S}}\!\setminus(\hspace{-0.4pt}\maxel\mathcal{T}) {\downfootline}_{\mathbf{S}}\,=\,\mathbf{v}\hspace{-0.5pt}{\downspoon}_{\mathbf{S}}\setminus
      \mathsf{MAX}{\downfootline}_{\mathbf{S}}\,=
      $$
      $$\big(\mathbf{v}\hspace{-0.5pt}{\downfilledspoon}_\mathbf{S}\hspace{-1pt}\setminus
      \{\hspace{-0.5pt}\mathbf{v}\hspace{-0.7pt}\}\big)
      \setminus{\Omega}\,=\,
      \big(\mathbf{v}\hspace{-0.5pt}{\downfilledspoon}_\mathbf{S}\hspace{-1pt}\setminus{\Omega}\big)
      \setminus\{\hspace{-0.5pt}\mathbf{v}\hspace{-0.7pt}\}\,=\,
      \Delta\,{\setminus}\,\{\hspace{-0.5pt}\mathbf{v}\hspace{-0.7pt}\}.$$
   \end{itemize}

Now we build a foliage tree $\mathbf{G}$ with $\skeleton\mathbf{G}=\mathcal{T}$ as follows:

   \begin{itemize}
   \item[\ding{226}\ ]
      $\mathsurround=0pt
      \mathbf{G}_{\hspace{-0.3pt}{z}}\coloneq\mathbf{S}_{z}\quad$
      for all ${z}\in\mathsf{MAX};$
   \item[\ding{226}\ ]
      $\mathsurround=0pt\displaystyle
      \mathbf{G}_{\hspace{-0.3pt}{\mathsf{node}}_{\hspace{0.1pt}{x}}^{{0}}}\:\coloneq\ \bigcup\big\{\mathbf{S}_{z}:{z}\in{\Omega}_{{x}_{-{0}},\,{x}}\big\}\quad$
      for all ${x}\in{\Delta};$
   \item[\ding{226}\ ]
      $\mathsurround=0pt\displaystyle
      \mathbf{G}_{\hspace{-0.3pt}{\mathsf{node}}_{\hspace{0.1pt}{x}}^{\hspace{0.3pt}{l}}}\:\coloneq\ \mathbf{G}_{\hspace{-0.3pt}{\mathsf{node}}_{\hspace{0.1pt}{x}}^{\hspace{0.3pt}{l}-1}}\cup\:
      \bigcup\big\{\mathbf{S}_{z}:{z}\in{\Omega}_{{x}_{-{l}},\,{x}}\big\}\quad$
      for all ${x}\in{\Delta}$ and
      ${l}\in\big\{1,\ldots,{l}(\hspace{-0.7pt}{x}\hspace{-1pt})\big\}  $
      (by recursion on $\hspace{1pt}\mathsurround=0pt{l}$);
   \item[\ding{226}\ ]
      $\mathsurround=0pt\displaystyle
      \mathbf{G}_\mathbf{v}\:\coloneq\ \bigcup\big\{
      \mathbf{G}_{\hspace{-0.3pt}{\mathsf{node}}_{\hspace{0.1pt}{x}}^{\hspace{0.3pt}{l}(\hspace{-0.7pt}{x}\hspace{-1pt})}}:
      {{x}\in{\Delta}}\big\}.$
   \end{itemize}
Then (e1), (c3), (\ref{*17.2}), and disjointness of the union from (b5) imply that $\mathbf{G}$ is locally strict.
Also it is not hard to show that $\mathbf{G}$ is nonincreasing,
${0}_{\hspace{-0.2pt}\mathbf{G}}=\mathbf{v},$
$\height\hspace{-0.4pt}\mathbf{G}\leqslant\omega,$
$\mathbf{G}$ is $\mathsurround=0pt\aleph_0$-branching,
$\mathbf{G}$ has bounded chains,
$\mathbf{G}$ is open in the Baire space,
$\mathbf{G}$ is a foliage graft for $\mathbf{S},$
$\impl\mathbf{G}\neq\varnothing,$
and ${O}\equiv\bigsqcup_{{z}\in\maxel\mathbf{G}}\mathbf{S}_{z}.$
To prove that $\cut(\mathbf{S},\mathbf{G})=\mathbf{S}_\mathbf{v}{\setminus}\,{O}$ we must show that
$\mathbf{G}_{\hspace{-0.3pt}{0}_{\hspace{-0.6pt}\mathbf{G}}}={O}.$
Since $\mathbf{G}$ is nonincreasing,
we have $\mathbf{G}_{\hspace{-0.3pt}{0}_{\hspace{-0.6pt}\mathbf{G}}}=\flesh\hspace{-1pt}\mathbf{G},$
so using (\hspace{0.3pt}b\hspace{-0.5pt}) of Lemma~\ref{about foliage trees}, (h) of Lemma~\ref{about trees}, and  (b5) we have\vspace{-0.5ex}
$$\textstyle\mathbf{G}_{\hspace{-0.3pt}{0}_{\hspace{-0.6pt}\mathbf{G}}}=\,\flesh\hspace{-1pt}\mathbf{G}\,=\,
\yield\mathbf{G}\,=\,\bigcup\!\big\{\hspace{-0.7pt}\fruit_\mathbf{G}(\hspace{-1.1pt}{B}\hspace{-0.6pt})
:{B}\in\branches\hspace{-0.5pt}\mathbf{G}\big\}\,=\vspace{-0.5ex}$$
$$\textstyle\bigcup\!\big\{\hspace{-0.7pt}\fruit_{\mathbf{G}}\hspace{0.1pt}\!
(\hspace{-0.7pt}{z}\hspace{-0.2pt}{\upfilledspoon}_{\hspace{-1.5pt}\mathbf{G}}\hspace{-0.7pt})
\hspace{0.5pt}:\hspace{0.5pt}{z}\in\maxel\mathbf{G}\hspace{0.5pt}\big\}=\,\bigcup\{\mathbf{G}_{\hspace{-0.3pt}{z}}:{z}\in
\maxel\mathbf{G}\}=\,\bigcup\{\mathbf{S}_{z}:{z}\in\mathsf{MAX}\}\,=\,{O}.\vspace{1ex}$$

It remains to prove that $\mathbf{G}$ preserves shoots of $\mathbf{S}.$ Suppose
$${p}\in\flesh\hspace{-1pt}\mathbf{G}=\mathbf{G}_{\hspace{-0.3pt}{0}_{\hspace{-0.6pt}\mathbf{G}}}\!=
{O},\quad {y}\in\{{0}_{\hspace{-0.2pt}\mathbf{G}}\}\cup\hspace{0.4pt}\expl(\mathbf{S},\mathbf{G})=
\Delta,\quad\mathsf{and}\quad\mathbf{S}_{{y}}\ni{p}.$$
We must find
${x}\in\{\hspace{-0.5pt}\mathbf{v}\hspace{-0.7pt}\}\cup\hspace{0.4pt}\hspace{0.7pt}\mathsf{IMP}$ such that
$\mathbf{G}_{\hspace{-0.3pt}{x}}\ni{p}$ and
$\shoot_{\hspace{-0.2pt}\mathbf{G}\hspace{-0.8pt}}(\hspace{-0.7pt}{x}\hspace{-1pt})
\gg\shoot_\mathbf{S}({y}).$
Note that $\mathbf{G}$ has nonempty leaves and
$|\hspace{-1pt}\sons_{\hspace{0.2pt}\mathbf{G}\hspace{-1.2pt}}
(\hspace{-0.7pt}{x}\hspace{-1pt})\hspace{0.7pt}|
\geqslant\aleph_{0}$ for all ${x}\in\{\hspace{-0.5pt}\mathbf{v}\hspace{-0.7pt}\}\cup\hspace{0.4pt}\hspace{0.7pt}\mathsf{IMP}$
since $\mathbf{G}$ is $\mathsurround=0pt\aleph_0$-branching and
$\{\hspace{-0.5pt}\mathbf{v}\hspace{-0.7pt}\}\cup\hspace{0.4pt}\hspace{0.7pt}\mathsf{IMP}\subseteq
\nodes\mathbf{G}\hspace{1pt}{\setminus}\hspace{0.7pt}\maxel\mathbf{G}.$
Lemma~\ref{about shoots} says that if there is finite ${F}$ such that
$$
\forall{s}\,{\in}\,\sons_{\hspace{0.2pt}\mathbf{G}\hspace{-1.2pt}}(\hspace{-0.7pt}{x}\hspace{-1pt})\,{\setminus}\,{F}
\;\big[\,{s}\in\hspace{0.4pt}\sons_{\hspace{0.5pt}\mathbf{S}\hspace{-0.5pt}}({y})
\enskip\mathsf{and}\enskip\mathbf{G}_{\hspace{-0.3pt}{s}}
\subseteq\mathbf{S}_{s}\,\big],
$$
then $\shoot_{\hspace{-0.2pt}\mathbf{G}\hspace{-0.8pt}}(\hspace{-0.7pt}{x}\hspace{-1pt})\gg\shoot_\mathbf{S}({y}).$
If ${x}\in\mathsf{IMP},$ then by (e1), (c3), and (\ref{*17.2}) there is finite ${L}$ such that
$\sons_{\hspace{0.2pt}\mathbf{G}\hspace{-1.2pt}}(\hspace{-0.7pt}{x}\hspace{-1pt})\setminus{L}\subseteq\mathsf{MAX},$
so for all
${s}\in\hspace{0.4pt}\sons_{\hspace{0.2pt}\mathbf{G}\hspace{-1.2pt}}(\hspace{-0.7pt}{x}\hspace{-1pt})\setminus{L}$
we have $\mathbf{G}_{\hspace{-0.3pt}{s}}=\mathbf{S}_{s}.$

Summarizing the above reasoning we come to the following.
Suppose ${y}\in\Delta$ and ${p}\in{O}\cap\mathbf{S}_{{y}}.$
Then to finish the proof it is enough to find ${x}\in\mathsf{IMP}$ and finite ${F}$ such that
\begin{equation}\label{*161.4}
\mathbf{G}_{\hspace{-0.3pt}{x}}\ni{p}\quad\mathsf{and}\quad
\sons_{\hspace{0.2pt}\mathbf{G}\hspace{-1.2pt}}(\hspace{-0.7pt}{x}\hspace{-1pt})\,{\setminus}\,{F}
\subseteq\hspace{0.4pt}\sons_{\hspace{0.5pt}\mathbf{S}\hspace{-0.5pt}}({y})\hspace{0.7pt}.
\end{equation}

Since ${p}\in{O},$ then by (b5) there is ${\dot{z}}\in\mathsf{MAX}$
such that ${p}\in\mathbf{S}_{\dot{z}}.$
Then $\mathbf{S}_{\dot{z}}\cap\mathbf{S}_{y}\neq\varnothing,$ so
either ${y}\geqslant_\mathbf{S}\dot{z}$ or ${y}<_\mathbf{S}\dot{z}$
since $\mathbf{S}$ is splittable.
If ${y}\geqslant_\mathbf{S}\dot{z},$ then by (b4) ${y}\in\Omega,$
which contradicts ${y}\in\Delta,$ so ${y}<_\mathbf{S}\dot{z}.$
Let ${w}\coloneq\dot{z}_{-1}.$ Then we have
$$\mathbf{v}\,\leqslant_\mathbf{S}\,{y}\,\leqslant_\mathbf{S}\,{w}
\,\sqsubset_\mathbf{S}\,\dot{z}\,\in\,{\mathsf{MAX}}\,=\,
\minel\,({\Omega},<_\mathbf{S})\subseteq\,\Omega,
$$
which implies ${w}\in\Delta$ and $\dot{z}\in{\Omega}_{w}.$
Then it follows by~(\ref{*17.2}) that
there is ${d}\in{\Delta}_{w}$ such that
$\dot{z}\in{\Omega}_{{w},{d}}.$
Now we have
$$\mathbf{v}\,\leqslant_\mathbf{S}\,{y}\,\leqslant_\mathbf{S}\,{w}
\,\leqslant_\mathbf{S}{d}\,\in\,\Delta,\quad
\dot{z}\,\in\,{\Omega}_{{w},{d}}\hspace{1pt},\quad\mathsf{and}\quad
{p}\,\in\,\mathbf{S}_{\dot{z}}.
$$
Let $l\coloneq\length{d}-\length{y}$ and $m\coloneq\length{d}-\length{w}.$
Then ${d}_{-\hspace{0.5pt}{l}}={y},$ ${d}_{-\hspace{0.5pt}{m}}={w},$
and $0\leqslant{m}\leqslant{l}\leqslant{l}(\hspace{-1pt}{d}\hspace{0.1pt}),$
so we may consider nodes ${\mathsf{node}}_{\hspace{0.3pt}{{d}}}^{\hspace{0.3pt}{l}}$
and ${\mathsf{node}}_{\hspace{0.3pt}{{d}}}^{\hspace{0.2pt}{m}}$ in $\mathsf{IMP}.$
Then ${x}\coloneq{\mathsf{node}}_{\hspace{0.3pt}{{d}}}^{\hspace{0.3pt}{l}}$
satisfies condition~(\ref{*161.4}). Indeed,
${\mathsf{node}}_{\hspace{0.3pt}{{d}}}^{\hspace{0.3pt}{l}}\leqslant_\mathbf{G}
{\mathsf{node}}_{\hspace{0.3pt}{{d}}}^{\hspace{0.2pt}{m}}$ by (e2)
and $\mathbf{G}$ is nonincreasing, so
$$\textstyle\mathbf{G}_{\hspace{-0.3pt}{x}}\,=\,\mathbf{G}_{\hspace{-0.3pt}{\mathsf{node}}_{\hspace{0.3pt}{{d}}}^{\hspace{0.3pt}{l}}}
\,\supseteq\,\mathbf{G}_{\hspace{-0.3pt}{\mathsf{node}}_{\hspace{0.3pt}{{d}}}^{\hspace{0.2pt}{m}}}
\,\supseteq\:\bigcup\{\mathbf{S}_{z}:{z}\in{\Omega}_{{{d}}_{-{m}},{d}}\}\,=\:
\bigcup\{\mathbf{S}_{z}:{z}\in{\Omega}_{{w},{d}}\}
\,\supseteq\,\mathbf{S}_{\dot{z}}\,\ni\,{p}.$$
Finally, by (e1) there is finite ${F}$ such that
$$\sons_{\hspace{0.2pt}\mathbf{G}\hspace{-1.2pt}}(\hspace{-0.7pt}{x}\hspace{-1pt})
\,{\setminus}\,{F}\,=\:\hspace{0.2pt}
\sons_{\hspace{0.2pt}\mathbf{G}\hspace{-1.2pt}}
({\mathsf{node}}_{\hspace{0.3pt}{{d}}}^{\hspace{0.3pt}{l}})
\,{\setminus}\,{F}\:=\:{\Omega}_{{d}_{-{l}},{d}}\,=\:{\Omega}_{{y},{d}}\:\subseteq\:
{\Omega}_{{y}}\:\subseteq\:\hspace{0.4pt}
\sons_{\hspace{0.5pt}\mathbf{S}\hspace{-0.5pt}}({y}).\vspace{-2ex}$$
\end{proof}

\begin{proof}[\textbf{\textup{Proof of \hspace{1pt}Theorem~\ref{main.theorem}}}]
Let ${\mathbf{v}}\in{}^{<\hspace{0.2pt}\omega}\hspace{-1pt}\omega$ and ${n}\in\omega.$
Put ${O}\coloneq{U}_{n}\cap\hspace{0.5pt}\mathbf{S}_{\mathbf{v}}$
and assume that ${O}\neq\mathbf{S}_{\mathbf{v}}.$
Then there is a foliage tree $\mathbf{G}$ that satisfies conditions (a1)--(a11) of Lemma~\ref{Construction of an LPB-fine foliage graft}. Let us denote this foliage tree $\mathbf{G}$ by $\mathbf{G}({\mathbf{v}},{n}).$
Using this notation, we construct sequences $({Z}_{n})_{{n}\in\hspace{0.3pt}\omega},$ $(\psi_{n})_{{n}\in\hspace{0.3pt}\omega},$
and $({M}_{n})_{{n}\in\{\hspace{-0.8pt}{-}\hspace{-0.3pt}1\hspace{-0.5pt}\}\cup\hspace{1pt}\omega}$ by recursion on ${n}$ as follows:

   \begin{itemize}
   \item [\textup{(f1)}\ ]
      $\mathsurround=0pt
      {M}_{{-}1}\coloneq\{{0}_{\mathbf{S}}\};$
   \item [\textup{(f2)}\ ]
      $\mathsurround=0pt
      {Z}_{n}\coloneq\{{x}\in{M}_{{n}-1}:{U}_{n}\cap\mathbf{S}_{x}\neq\mathbf{S}_{x}\};$
   \item [\textup{(f3)}\ ]
      $\mathsurround=0pt
      \psi_{n}\coloneq\big\{\mathbf{G}({x},{n}):{x}\in{Z}_{n}\big\};$
   \item [\textup{(f4)}\ ]
      $\mathsurround=0pt
      {M}_{n}\coloneq({M}_{{n}-1}{\setminus}\,{Z}_{n})
      \,\cup\,\bigcup_{\mathbf{G}\in\psi_{n}}\!\maxel\mathbf{G}.$
   \end{itemize}

For each ${n}\in\omega,$ we will prove the following:

   \begin{itemize}
   \item [\textup{(g1)}\ ]
      $\mathsurround=0pt
      {Z}_{n}=\{{0}_\mathbf{G}:\mathbf{G}\in\psi_{n}\};$
   \item [\textup{(g2)}\ ]
      $\mathsurround=0pt
      {M}_{n}$ is an antichain in $\mathbf{S};$
   \item [\textup{(g3)}\ ]
      $\mathsurround=0pt
      ({M}_{n}){\hspace{-1pt}\downfootline}_\mathbf{S}\cap\bigcup_{{i}\leqslant{n}}{Z}_{i}=\varnothing;$
   \item [\textup{(g4)}\ ]
      $\mathsurround=0pt
      \bigcup_{{i}\leqslant\hspace{0.3pt}{n}}\psi_{i}\;$ is a consistent family of foliage grafts for $\mathbf{S};$
   \item [\textup{(g5)}\ ]
      $\mathsurround=0pt
      \bigcup_{{y}\in{M}_{n}}\mathbf{S}_{y}=\bigcap_{{i}\leqslant\hspace{0.3pt}{n}}{U}_{i};$
   \item [\textup{(g6)}\ ]
      $\mathsurround=0pt
      \bigcup\{\cut(\mathbf{S},\mathbf{G}):\mathbf{G}\in\bigcup_{{i}\leqslant\hspace{0.3pt}{n}}\psi_{i}\}=
      {}^\omega\hspace{-1pt}\omega\setminus\bigcap_{{i}\leqslant\hspace{0.3pt}{n}}{U}_{i}.$
   \end{itemize}

Let us first show that (g1)--(g6) yield the conclusion of the theorem. Put $\varphi\coloneq\bigcup_{{n}\in\hspace{0.3pt}\omega}\psi_{n},$ so (g4) implies that $\varphi$ is a consistent family of foliage grafts for $\mathbf{S}.$ Then $\mathbf{H}\coloneq\fhybr(\mathbf{S},\varphi)$ satisfies the requirements of the theorem.
Indeed, (g6) imply that $\loss(\mathbf{S},\varphi)=
{}^\omega\hspace{-1pt}\omega\setminus\bigcap_{{n}\in\hspace{0.3pt}\omega}{U}_{n},$ so
it follows from Lemma~\ref{lem.when.f.hybr.is.pi.tree}, (a) of Lemma~\ref{lem.pi.and.B.f.trees.vs.S}, and (a2)--(a6) that $\mathbf{H}$ is a Baire foliage tree on  ${Y}.$ Then $\mathbf{H}$ has nonempty leaves, therefore $\mathbf{H}$ shoots into $\mathbf{S}$ by (a8) and Lemma~\ref{lem.when.f.hyb.sprts.thrgh}.

It remains to prove that (g1)--(g6) hold for all ${n}\in\omega.$
Condition (g1) easily follows from definitions of ${Z}_{n}$ and $\psi_{n}:$
if ${x}\in{Z}_{n},$ then ${x}={0}_{\mathbf{G}({x},n)}$ by (a1),
so ${Z}_{n}=\{{0}_{\mathbf{G}({x},n)}:{x}\in{Z}_{n}\}=
\{{0}_\mathbf{G}:\mathbf{G}\in\psi_{n}\}.$
Conditions (g2)--(g6) will be proved by induction.
Using (a1)--(a11), and (\hspace{-0.4pt}d)--(e) of Definition~\ref{def.graft},
it is not hard to show that (g2)--(g6) are satisfied when ${n}={0}.$
Assume as induction hypothesis that (g2)--(g6) hold for all ${n}\leqslant{k}.$
We must prove that (g2)--(g6) hold for ${n}={k}\,{+}\,1.$

\medskip
(g2) We prove that ${M}_{{k}+1}$ is an antichain in $\mathbf{S}.$
Suppose ${v}\neq{w}\in{M}_{{k}+1}.$ We consider several cases:
   \begin{itemize}
   \item[\textup{(i)\ }]
      $\mathsurround=0pt{v},{w}\in{M}_{{k}}.$

      Then ${v}\parallel_\mathbf{S}{w}$ by the induction hypothesis.
   \item[\textup{(ii)\ }]
      $\mathsurround=0pt{v},{w}\in{M}_{{k}+1}\,{\setminus}\,{M}_{{k}}.$

      It follows by (f4) that there are $\mathbf{D},\mathbf{E}\in\psi_{{k}+1}$
      such that ${v}\in\maxel\mathbf{D}$ and ${w}\in\maxel\mathbf{E}.$
      \begin{itemize}
      \item[\textup{(ii.1)\ }]
         $\mathsurround=0pt{0}_\mathbf{D}\neq{0}_\mathbf{E}.$

         We have ${0}_\mathbf{D},{0}_\mathbf{E}\in{Z}_{{k}+1}\subseteq{M}_{k}$ by (g1) and (f2),
         so ${0}_\mathbf{D}\parallel_\mathbf{S}{0}_\mathbf{E}$ by the induction hypothesis.
         Then ${v}\parallel_\mathbf{S}{w}$ by using (\hspace{0.3pt}b\hspace{-0.5pt}) of Lemma~\ref{about trees} twice.
      \item[\textup{(ii.2)\ }]
         $\mathsurround=0pt{0}_\mathbf{D}={0}_\mathbf{E}.$

         It follows from (f3) and (a1) that
         $\mathbf{D}=\mathbf{G}({0}_\mathbf{D},{k}\,{+}\,1\hspace{-1pt})=
         \mathbf{G}({0}_\mathbf{E},{k}\,{+}\,1\hspace{-1pt})=\mathbf{E},$
         so we have ${v},{w}\in\maxel\mathbf{D}.$
         Consequently ${v}\parallel_\mathbf{S}{w}$ by (a7) and (e) of Definition~\ref{def.graft}.
      \end{itemize}
   \item[\textup{(iii)\ }]
      $\mathsurround=0pt\big|\hspace{1pt}\{{v},{w}\}\cap{M}_{{k}}\hspace{1pt}\big|=1.$

      We may assume without lost of generality that ${v}\in{M}_{{k}+1}\,{\setminus}\,{M}_{{k}}$ and
      ${w}\in{M}_{{k}+1}\cap{M}_{{k}}.$ Again, as in (ii), there is $\mathbf{D}\in\psi_{{k}+1}$
      such that ${v}\in\maxel\mathbf{D},$ and then ${v}>_\mathbf{S}{0}_\mathbf{D}$
      and ${0}_\mathbf{D}\in{Z}_{{k}+1}\subseteq{M}_{k}.$
      \begin{itemize}
      \item[\textup{(iii.1)\ }]
         $\mathsurround=0pt{0}_\mathbf{D}\neq{w}.$

         We have ${v}>_\mathbf{S}{0}_\mathbf{D}$
         and ${0}_\mathbf{D}\parallel_\mathbf{S}{w}$ the induction hypothesis
         (because $\hspace{1pt}\mathsurround=0pt{0}_\mathbf{D},{w}\in{M}_{k}$),
         so ${v}\parallel_\mathbf{S}{w}$ by (\hspace{0.3pt}b\hspace{-0.5pt}) of Lemma~\ref{about trees}.
      \item[\textup{(iii.2)\ }]
         $\mathsurround=0pt{0}_\mathbf{D}={w}.$

         We have ${w}\in{M}_{{k}+1}$  and ${w}={0}_\mathbf{D}\in{Z}_{{k}+1},$
         so it follows from (f4) that there is $\mathbf{F}\in\psi_{{k}+1}$
         such that ${w}\in\maxel\mathbf{F}.$
         Consequently, ${w}>_\mathbf{S}{0}_\mathbf{F}$ and ${0}_\mathbf{F}\in{Z}_{{k}+1}.$
         This contradicts the induction hypothesis because
         ${w},{0}_\mathbf{F}\in{Z}_{{k}+1}\subseteq{M}_{k}.$
      \end{itemize}
   \end{itemize}

\medskip
(g3) We must prove that
$({M}_{{k}+1}){\hspace{-1pt}\downfootline}_\mathbf{S}
\cap\hspace{1pt}\bigcup_{{i}\leqslant{{k}+1}}\hspace{-1pt}{Z}_{i}=\varnothing.$
Suppose on the contrary that there is some ${x}\in({M}_{{k}+1}){\hspace{-1pt}\downfootline}_\mathbf{S}
\cap\hspace{1pt}\bigcup_{{i}\leqslant{{k}+1}}\hspace{-1pt}{Z}_{i}.$
Since by (g2) with ${n}={k}\,{+}\,1$ (which is already proved),
${M}_{{k}+1}$ is an antichain in $\mathbf{S},$ then we may consider
$${r}\,\coloneq\,\roott_\mathbf{S}({x},{M}_{{k}+1})\,\in\,{M}_{{k}+1}\,=\,
({M}_{k}{\setminus}\,{Z}_{{k}+1})\,\cup\!\!\!
\bigcup_{\mathbf{G}\in\psi_{{k}+1}\!\!}\!\!\!\!\maxel\mathbf{G}.\hspace{-2ex}
$$
   \begin{itemize}
   \item[\textup{(i)\ }]
      $\mathsurround=0pt\exists\hspace{0.5pt}\mathbf{G}\,{\in}\,\psi_{{k}+1}
      \,[\,{r}\in\maxel\mathbf{G}\,].$

      Then ${x}\geqslant_\mathbf{S}{r}>_\mathbf{S}{0}_\mathbf{G}\in{Z}_{{k}+1}\subseteq{M}_{k},$
      therefore ${x}\in({M}_{k}){\hspace{-1pt}\downfootline}_\mathbf{S},$ so
      ${x}\notin\bigcup_{{i}\leqslant{k}}{Z}_{i}$ by (g3) with ${n}={k},$ and
      hence ${x}\in{Z}_{{k}+1}\subseteq{M}_{k}.$
      Now we have ${x}>_\mathbf{S}{0}_\mathbf{G}$ and ${x},{0}_\mathbf{G}\in{M}_{k},$
      which contradicts (g2) with ${n}={k}.$
   \item[\textup{(ii)\ }]
      $\mathsurround=0pt{r}\in{M}_{k}\setminus{Z}_{{k}+1}.$

      Then we have ${x}\geqslant_\mathbf{S}{r}\in{M}_{k},$ therefore as in (i) we get
      ${x}\in({M}_{k}){\hspace{-1pt}\downfootline}_\mathbf{S},$
      ${x}\notin\bigcup_{{i}\leqslant{k}}{Z}_{i},$ and ${x}\in{Z}_{{k}+1}\subseteq{M}_{k}.$
      Also we have ${x}\neq{r}$ because ${r}\notin{Z}_{{k}+1},$
      consequently ${x}>_\mathbf{S}{r}$ and ${x},{r}\in{M}_{k},$
      which again contradicts (g2) with ${n}={k}.$
   \end{itemize}

\medskip
(g4) We must prove that
$\bigcup_{{i}\leqslant\hspace{0.3pt}{k}+1}\psi_{i}$ is a consistent family of foliage grafts for $\mathbf{S}.$
Every $\mathbf{G}\in\bigcup_{{i}\leqslant\hspace{0.3pt}{k}+1}\psi_{i}$ is a
foliage graft for $\mathbf{S}$ by (a7).
Suppose $\mathbf{D}\neq\mathbf{E}\in\bigcup_{{i}\leqslant\hspace{0.3pt}{k}+1}\psi_{i}.$
We may assume that $\impl\mathbf{D}\cap\impl\mathbf{E}=\varnothing$ by construction,
and then $\skeleton\mathbf{D}\neq\skeleton\mathbf{E}$ because implants of
$\mathbf{D}$ and $\mathbf{E}$ are nonempty by (a9).
It remains to check clause (c) of Definition~\ref{def.c.f.grafts}.
We consider several cases:

   \begin{itemize}
   \item[\textup{(i)\ }]
      $\mathsurround=0pt\mathbf{D},\mathbf{E}\in\bigcup_{{i}\leqslant\hspace{0.3pt}{k}}\psi_{i}.$

      Then (c) of Definition~\ref{def.c.f.grafts} is satisfied by the induction hypothesis.
   \item[\textup{(ii)\ }]
      $\mathsurround=0pt\mathbf{D},\mathbf{E}\in\psi_{{k}+1}.$

      Then by (f3) $\mathbf{D}=\mathbf{G}({x},{k}\,{+}\,1\hspace{-1pt})$ and $\mathbf{E}=\mathbf{G}({y},{k}\,{+}\,1\hspace{-1pt})$
      for some ${x}\neq{y}\in{Z}_{{k}+1},$
      so it follows by using (f2), (g1), and (a1) that
      $${M}_{k}\,\supseteq\,{Z}_{{k}+1}\,\ni\,{0}_\mathbf{D}\,=\,
      {0}_{\mathbf{G}({x},{k}+1\hspace{-1pt})}\,=\,{x}\,\neq\,{y}\,=\,{0}_{\mathbf{G}({y},{k}+1)}\,=\,
      {0}_\mathbf{E}\,\in\,{Z}_{{k}+1}\,\subseteq\,{M}_{k}.
      $$
      Consequently, ${0}_\mathbf{D}\parallel_\mathbf{S}{0}_\mathbf{E}$ by (g2) with ${n}={k}.$
   \item[\textup{(iii)\ }]
      $\mathsurround=0pt\big|\hspace{1pt}\{\mathbf{D},\mathbf{E}\}\cap\psi_{{k}+1}\hspace{1pt}\big|=1.$

      Suppose without lost of generality that $\mathbf{D}\in\bigcup_{{i}\leqslant\hspace{0.3pt}{k}}\psi_{i}$
      and $\mathbf{E}\in\psi_{{k}+1}.$
      Then by (g1) ${0}_\mathbf{D}\in\bigcup_{{i}\leqslant{k}}{Z}_{i}$
      and ${0}_\mathbf{E}\in{Z}_{{k}+1}\subseteq{M}_{k}\subseteq
      ({M}_{k}){\hspace{-1pt}\downfootline}_\mathbf{S},$
      so it follows by using (g3) with ${n}={k}$ that ${0}_\mathbf{D}\neq{0}_\mathbf{E}.$
      If ${0}_\mathbf{D}\parallel_\mathbf{S}{0}_\mathbf{E},$ then clause (c) of Definition~\ref{def.c.f.grafts} holds. It remains to consider the following two cases:
      \begin{itemize}
      \item[\textup{(iii.1)\ }]
         $\mathsurround=0pt{0}_\mathbf{D}>_\mathbf{S}{0}_\mathbf{E}.$

         We have ${0}_\mathbf{D}\in\bigcup_{{i}\leqslant\hspace{0.3pt}{k}}{Z}_{i}$ and ${0}_\mathbf{D}>_\mathbf{S}{0}_\mathbf{E}\in{Z}_{{k}+1}\subseteq{M}_{k},$
         so ${0}_\mathbf{D}\in({M}_{k}){\downfootline_\mathbf{S}}.$
         This contradicts (g3) with ${n}={k}.$
      \item[\textup{(iii.2)\ }]
         $\mathsurround=0pt{0}_\mathbf{E}>_\mathbf{S}{0}_\mathbf{D}.$

         Now, $\mathbf{D}\in\bigcup_{{i}\leqslant\hspace{0.3pt}{k}}\psi_{i}$
         and $\bigcup_{{i}\leqslant\hspace{0.3pt}{k}}\psi_{i}$ is a consistent family of foliage grafts for $\mathbf{S}$ by the induction hypothesis.
         Further, ${0}_\mathbf{E}\in{Z}_{{k}+1}\subseteq{M}_{k}$
         and it is not hard to show that
         $$\textstyle{M}_{k}\subseteq\{{0}_\mathbf{S}\}\cup
         \bigcup\{\maxel\mathbf{G}:\mathbf{G}\in\bigcup_{{i}\leqslant\hspace{0.3pt}{k}}\psi_{i}\}
         $$
         by induction on ${k}.$
         Then it follows from (\hspace{0.3pt}b\hspace{-0.5pt}) of Lemma~\ref{lem.graft} that  ${0}_\mathbf{E}\in
         \supp(\mathbf{S},\bigcup_{{i}\leqslant\hspace{0.3pt}{k}}\psi_{i}\hspace{-0.5pt}).$
         Furthermore, (\hspace{-0.4pt}d) of Lemma~\ref{lem.graft} says that
         ${0}_\mathbf{E}\in
         \supp(\mathbf{S},\bigcup_{{i}\leqslant\hspace{0.3pt}{k}}\psi_{i}\hspace{-0.5pt})$
         plus ${0}_\mathbf{E}>_\mathbf{S}{0}_\mathbf{D}$
         imply ${0}_\mathbf{E}\in(\maxel\mathbf{D}){\downfootline_\mathbf{S}},$
         so (c) of Definition~\ref{def.c.f.grafts} holds.
      \end{itemize}
   \end{itemize}

\medskip
(g5) We must prove that
$\bigcup_{{y}\in{M}_{{k}+1}}\!\mathbf{S}_{y}=\bigcap_{{i}\leqslant\hspace{0.3pt}{k}+1}{U}_{i}.$
Put ${B}\coloneq\bigcup_{{y}\in{M}_{{k}}\hspace{-0.5pt}{\setminus}{Z}_{{k}+1}}\mathbf{S}_{y}.$
Then (f2) implies
\begin{equation}\label{*g4}
\textstyle{B}=\bigcup\{\hspace{1pt}{U}_{{k}+1}\cap\mathbf{S}_{y}:
{y}\in{M}_{{k}}{\setminus}\,{Z}_{{k}+1}\hspace{1pt}\}.
\end{equation}
Now, using (f4), (f3), (a11), \eqref{*g4}, and (g5) with ${n}={k},$ we have
$$\bigcup_{{y}\in{M}_{{k}+1}\!}\!\!\!\!\mathbf{S}_{y}\:=\:
{B}\,\cup\:
\bigcup\!\big\{\mathbf{S}_{y}:{y}\in
{\textstyle\bigcup_{\mathbf{G}\in\psi_{{k}+1}}\!\maxel\mathbf{G}}\big\}\:=\:
{B}\,\cup\:
\bigcup\!\big\{\mathbf{S}_{y}:
{\textstyle{y}\in\bigcup_{{x}\in{Z}_{{k}+1}}\!\maxel\hspace{1pt}\mathbf{G}({x},{k}\,{+}\,1\hspace{-1pt})}\big\}
\:=\vspace{-2ex}
$$
$${B}\:\cup\!\!
\bigcup_{{x}\in{Z}_{{k}+1}\!}\!\!\!\!\Big(\!\bigcup\!\big\{\mathbf{S}_{y}:
{y}\in\maxel\hspace{1pt}\mathbf{G}({x},{k}\,{+}\,1\hspace{-1pt})\big\}\!\Big)\:=\:
{B}\:\cup\!\!
\bigcup_{{x}\in{Z}_{{k}+1}\!\!}\!\!\!({U}_{{k}+1}\hspace{-1pt}\cap\mathbf{S}_{x}\hspace{-0.5pt})\:=
$$
$$\bigcup_{{x}\in{M}_{{k}}}\!\!({U}_{{k}+1}\hspace{-1pt}\cap\mathbf{S}_{x})\:=\:
{U}_{{k}+1}\:\cap\bigcup_{{x}\in{M}_{{k}}\!\!}\!\!\mathbf{S}_{x}\:=\:
{U}_{{k}+1}\:\cap\:\bigcap_{{i}\leqslant\hspace{0.3pt}{{k}}}{U}_{i}\:=
\bigcap_{{i}\leqslant\hspace{0.3pt}{k}+1\!}\!\!{U}_{i}.
$$

\medskip
(g6) We must prove that
$\bigcup\{\cut(\mathbf{S},\mathbf{G})\:{:}\:\mathbf{G}\in\bigcup_{{i}\leqslant\hspace{0.3pt}{k}+1}\psi_{i}\}=
{}^\omega\hspace{-1pt}\omega\setminus\bigcap_{{i}\leqslant\hspace{0.3pt}{k}+1}{U}_{i}.$
Put ${A}\coloneq\bigcap_{{i}\leqslant\hspace{0.3pt}{k}}{U}_{i},$ so that
the induction hypothesis asserts \vspace{-1ex}
\begin{equation}\label{*g5}
\textstyle\bigcup\{\cut(\mathbf{S},\mathbf{G})\:{:}\:
\mathbf{G}\in\bigcup_{{i}\leqslant\hspace{0.3pt}{{k}}}\psi_{i}\}=
{}^\omega\hspace{-1pt}\omega\setminus{A}.
\end{equation}
Then using \eqref{*g5}, (f3), (a10), (f2), and (g5) with ${n}={k},$ we have
$${\textstyle\bigcup\{\cut(\mathbf{S},\mathbf{G})\:{:}\:
\mathbf{G}\in\bigcup_{{i}\leqslant\hspace{0.3pt}{k}+1}\psi_{i}\}}\:=\:
({}^\omega\hspace{-1pt}\omega\,{\setminus}\,{A})
\;\cup\!\!
\bigcup_{\mathbf{G}\in\hspace{0.5pt}\psi_{{k}+1}\!\!\!}\!\!\!\!\cut(\mathbf{S},\mathbf{G})\:=\:
({}^\omega\hspace{-1pt}\omega\,{\setminus}\,{A})
\;\cup\!
\bigcup_{{x}\in{Z}_{{k}+1}\!\!\!}\!\!\!\cut\hspace{-1pt}
\big(\mathbf{S},\mathbf{G}({x},{k}\,{+}\,1\hspace{-1pt})\big)
\;\cup\;\varnothing\:=\vspace{-1ex}
$$
$$({}^\omega\hspace{-1pt}\omega\,{\setminus}\,{A})
\;\cup\!\!
\bigcup_{{x}\in{Z}_{{k}+1}\!}\!\!\!\!\big(\mathbf{S}_{x}\!\setminus
({U}_{{k}+1}\hspace{-1pt}\cap\mathbf{S}_{x}\hspace{-0.5pt})\big)
\;\cup\!\!\!
\!\!\bigcup_{{x}\in{M}_{{k}}\!{\setminus}{Z}_{{k}+1}\!\!\!\!\!\!}\!\!\!\!\!\!\varnothing
\:=\:
({}^\omega\hspace{-1pt}\omega\,{\setminus}\,{A})
\;\cup\!\!
\bigcup_{{x}\in{Z}_{{k}+1}\!}\!\!\!(\mathbf{S}_{x}{\setminus}\,{U}_{{k}+1})
\;\cup
\!\!\!\!\!\bigcup_{{x}\in{M}_{{k}}\!{\setminus}{Z}_{{k}+1}\!\!\!\!\!\!}
\!\!\!\!\!\!(\mathbf{S}_{x}{\setminus}\,{U}_{{k}+1})
\:=$$
$$({}^\omega\hspace{-1pt}\omega\,{\setminus}\,{A})
\:\cup
\bigcup_{{x}\in{M}_{{k}}\!}\!(\mathbf{S}_{x}{\setminus}\,{U}_{{k}+1})\:=\:
({}^\omega\hspace{-1pt}\omega\,{\setminus}\,{A})
\:\cup\,
\big((\!\bigcup_{{x}\in{M}_{{k}}\!}\!\!\hspace{-1pt}\mathbf{S}_{x})\setminus{U}_{{k}+1}\big)\:=
$$
$$\textstyle({}^\omega\hspace{-1pt}\omega\,{\setminus}\,{A})
\:\cup\:
({A}\setminus{U}_{{k}+1})\:=\:
\textstyle{}^\omega\hspace{-1pt}\omega\,{\setminus}\,({A}\cap\hspace{0.5pt}{U}_{{k}+1})\:=\:
{}^\omega\hspace{-1pt}\omega\setminus\bigcap_{{i}\leqslant\hspace{0.3pt}{k}+1}{U}_{i}.$$
\end{proof}

\section{Main results}\label{sect.LPB.theorems}

In this section we prove theorems that allow to construct $\hspace{1pt}\mathsurround=0pt\pi$-trees for subspaces of a space that already has a $\hspace{1pt}\mathsurround=0pt\pi$-tree. Recall that $\mathbf{S}$ is the standard foliage tree of ${}^\omega\hspace{-1pt}\omega,$ see Notation~\ref{def.stdrt.f.tr.of.B.sp}.

 \begin{teo}\label{corollary.1}
    Suppose that $\mathbf{S}$ is a $\hspace{1pt}\mathsurround=0pt\pi$-tree on a space $({}^\omega\hspace{-1pt}\omega,\tau).$
    Let $\hspace{2pt}{Y}=\bigcap_{{n}\in\hspace{0.3pt}\omega}{U}_{n},$ where each $\hspace{1pt}{U}_{n}$ is an open $\hspace{1pt}\mathsurround=0pt\pi$-dense\footnote{see Definition~\ref{not.pi.dense}} subset of the Baire space $({}^\omega\hspace{-1pt}\omega,\tau_{\hspace{-1pt}\scriptscriptstyle\mathcal{N}}).$
    Then $\hspace{1pt}{Y}\hspace{-2pt}$ as a subspace of $({}^\omega\hspace{-1pt}\omega,\tau)$ has a $\hspace{1pt}\mathsurround=0pt\pi$-tree.
 \end{teo}

Using Lemma~\ref{lem.pi.and.B.f.trees.vs.S}, we can apply this theorem not only to a space of the form $({}^\omega\hspace{-1pt}\omega,\tau),$ but to an arbitrary space with a  $\hspace{1pt}\mathsurround=0pt\pi$-tree.

\begin{que} Does Theorem~\ref{corollary.1} remains true if we replace ``$\mathsurround=0pt\pi$-dense'' by ``dense''\,?
\end{que}

\begin{proof}[\textbf{\textup{Proof of \hspace{1pt}Theorem~\ref{corollary.1}}}]
Let $\rho(\hspace{-1pt}\tau\hspace{-1pt})$ and $\rho(\hspace{-1pt}\tau_{\hspace{-1pt}\scriptscriptstyle\mathcal{N}}\hspace{-1pt})$ be topologies on~${Y}\hspace{-1pt}$ inherited from $\tau$ and $\tau_{\hspace{-1pt}\scriptscriptstyle\mathcal{N}}$ respectively.
Theorem~\ref{main.theorem} asserts that there is a Baire foliage tree $\mathbf{H}$ on $\big({Y},\rho(\hspace{-1pt}\tau_{\hspace{-1pt}\scriptscriptstyle\mathcal{N}}\hspace{-1pt})\big)$
that shoots into $\mathbf{S}.$
Then $\mathbf{H}$ is a Baire foliage tree on $\big({Y},\rho(\hspace{-1pt}\tau\hspace{-1pt})\big)$ because
$\rho(\hspace{-1pt}\tau\hspace{-1pt})\supseteq
\rho(\hspace{-1pt}\tau_{\hspace{-1pt}\scriptscriptstyle\mathcal{N}}\hspace{-1pt})$ by (b\hspace{-0.5pt}) Lemma~\ref{lem.pi.and.B.f.trees.vs.S},
and $\mathbf{H}$ grows into $\big({Y},\rho(\hspace{-1pt}\tau\hspace{-1pt})\big)$ by Lemma~\ref{l.grows.into.subspace}
because $\flesh\mathbf{H}={Y}.$
Therefore $\mathbf{H}$ is a $\hspace{1pt}\mathsurround=0pt\pi$-tree on $\big({Y},\rho(\hspace{-1pt}\tau\hspace{-1pt})\big).$
\end{proof}

\begin{rem} The construction of a $\hspace{1pt}\mathsurround=0pt\pi$-tree in the proof of \hspace{1pt}Theorem~\ref{corollary.1} does not depend on topology~$\tau.$
\end{rem}

\begin{teo}\label{corollary.2}
   Suppose that a space ${X}$ has a $\hspace{1pt}\mathsurround=0pt\pi$-tree
   and $\hspace{1pt}{Y}={X}\,{\setminus}\,{F},$
   where ${F}$ is a $\hspace{1pt}\mathsurround=0pt\sigma$-compact subset of ${X}.$
   Then ${Y}\hspace{-2pt}$ also has a $\hspace{1pt}\mathsurround=0pt\pi$-tree.
\end{teo}

\begin{sle}\label{corollary.3}
   Suppose that a space ${X}$ has a $\hspace{1pt}\mathsurround=0pt\pi$-tree
   and $\hspace{1pt}{Y}={X}\setminus{C},$
   where ${C}$ is at most countable.
   Then ${Y}\hspace{-2pt}$ also has a $\hspace{1pt}\mathsurround=0pt\pi$-tree.
   \hfill$\Box$
\end{sle}

\begin{proof}[\textbf{\textup{Proof of \hspace{1pt}Theorem~\ref{corollary.2}}}]
Let ${F}=\bigcup_{{n}\in\hspace{0.3pt}\omega}{K}_{n},$ where each ${K}_{n}$ is a compact subset of ${X}.$
By (d) of Lemma~\ref{lem.pi.and.B.f.trees.vs.S}, there is a homeomorphism ${f}$ from
${X}$ onto a space $({}^\omega\hspace{-1pt}\omega,\tau)$ such that $\mathbf{S}$ is a $\hspace{1pt}\mathsurround=0pt\pi$-tree on $({}^\omega\hspace{-1pt}\omega,\tau).$ Also
it follows from (\hspace{0.3pt}b\hspace{-0.5pt}) of Lemma~\ref{lem.pi.and.B.f.trees.vs.S} that each ${f}({K}_{n})$ is a compact subset of the Baire space $({}^\omega\hspace{-1pt}\omega,\tau_{\hspace{-1pt}\scriptscriptstyle\mathcal{N}}).$ Then every ${U}_{n}\coloneq{}^\omega\hspace{-1pt}\omega\setminus{f}({K}_{n})$
is an open $\hspace{1pt}\mathsurround=0pt\pi$-dense subset of the Baire space by Remark~\ref{rem.pi.dense}, so the subspace ${}^\omega\hspace{-1pt}\omega\setminus\bigcup_{{n}\in\hspace{0.3pt}\omega}{f}({K}_{n})
=\bigcap_{{n}\in\hspace{0.3pt}\omega}{U}_{n}$ of $({}^\omega\hspace{-1pt}\omega,\tau)$ has a $\hspace{1pt}\mathsurround=0pt\pi$-tree by Theorem~\ref{corollary.1}.
\end{proof}

%=================Список литературы====================

\end{document}